\documentclass{amsart}
\usepackage[utf8]{inputenc}
\usepackage{nicefrac}
\usepackage{booktabs} %table
 \usepackage{graphicx}
 \usepackage{subcaption}

\theoremstyle{definition}
\newtheorem{theorem}{Theorem}[section]
\newtheorem{lemma}[theorem]{Lemma}
\newtheorem{definition}[theorem]{Definition}
\newtheorem{example}[theorem]{Example}

\newtheorem{prop}[theorem]{Proposition}
\newtheorem{cor}[theorem]{Corollary}

\theoremstyle{remark}
\newtheorem{remark}[theorem]{Remark}

\numberwithin{equation}{section}
\newcommand{\am}{\mathrm{am} }
\newcommand{\cn}{\mathrm{cn}}
\newcommand{\dn}{\mathrm{dn}}
\newcommand{\sn}{\mathrm{sn}}

\newcommand{\sech}{\mathrm{sech}}

\newcommand{\N}{\ensuremath{\mathbb{N}}}
\newcommand{\Z}{\ensuremath{\mathbb{Z}}}
\newcommand{\Sph}{\ensuremath{\mathbb{S}^1}}
\newcommand{\Hyp}{\ensuremath{\mathbb{H}^2}}
\newcommand{\R}{\ensuremath{\mathbb{R}}}
\newcommand{\E}{\ensuremath{\mathcal{E}}}

\newcommand*\diff{\mathop{}\!\mathrm{d}}

\allowdisplaybreaks

\begin{document}
\title[On the Convergence of the Elastic Flow in the Hyperbolic Plane]
{On the Convergence of the Elastic Flow\\in the Hyperbolic Plane}

\author{Marius M\"uller}
\address{Universit\"at Ulm, Helmholtzstra\ss{}e 18, 89081 Ulm, Germany}
\email{marius.mueller@uni-ulm.de}
%\thanks{.}

\author{Adrian Spener}
\address{Universit\"at Ulm, Helmholtzstra\ss{}e 18, 89081 Ulm, Germany}
%\curraddr{Universit\"at Ulm, Helmholtzstra\ss{}e 18, 89081 Ulm, Germany}
\email{adrian.spener@uni-ulm.de}
%\thanks{.}

%    General info 
\date{\today}

% \dedicatory{.}


\begin{abstract}
We examine the $L^2$-gradient flow of Euler's elastic energy for closed curves in hyperbolic space and prove convergence to the global minimizer for initial curves with elastic energy bounded by 16. We show the sharpness of this bound by constructing a class of curves {whose lengths} blow up in infinite time. The convergence results follow from a constrained sharp Reilly-type inequality.
\end{abstract}

\maketitle

%\tableofcontents

\section{Introduction}
\subsection{History and Context}
Our object of study  -- Euler's elastic energy -- measures the bending of a curve in some Riemannian manifold. Since Euler's characterization of its critical points in Euclidean space in 1744 it was widely studied and led to a variety of mathematical methods, see for instance \cite{HistoryElasticity}. For a smooth curve $\gamma\colon  I \rightarrow M$ in a Riemannian manifold $M$ it is defined by
\[\mathcal{E}(\gamma) := \int_\gamma \kappa^2 \diff s,
 \]
where $\kappa$ denotes the curvature and $\diff s$ denotes the integration with respect to the arclength parameter of $\gamma$ in $M$. Its $L^2$-gradient flow \eqref{eq:Flow} is called \emph{elastic flow}. {Critical points of $\mathcal{E}$ are called \emph{free elastica}}.

In \cite{Polden, Koiso,DKS,MR2911840,DS17} long time existence of the elastic flow in Euclidean  (i.e. $M = \R^d$, $d \geq 2$) and hyperbolic (i.e. $M = \mathbb{H}^2$) space was shown, but convergence results are only established under length penalization, i.e. for the $L^2$-gradient flow of $\mathcal{E}_\lambda(\gamma) := \mathcal{E}(\gamma) + \lambda \mathcal{L}(\gamma)$ for some $\lambda > 0 $, where $\mathcal{L}(\gamma)$ denotes the length of $\gamma$. Efforts were made to understand the behavior of the penalized flow as $\lambda \rightarrow 0 $. In \cite{MR1432203} it is shown that for energies satisfying a Palais-Smale condition, the convergence behavior is preserved provided there are no `migrating critical points', which is suggested by \cite{Steinberg} {for the elastic flow in the hyperbolic half plane $M = \mathbb{H}^2$}. However, as \cite{MR1432203} points out, the Palais-Smale condition fails to hold true in $\mathbb{H}^2$. This article % provides some progress on the 
answers the open problem posed by Linn\'er in \cite[Section 1.13]{MR1432203}, namely whether a positive penalization $\lambda > 0$ is necessary for the convergence of the gradient flow in the hyperbolic plane. The answer we give is the following: For full convergence of the flow length penalization \emph{is} necessary, whereas small initial energies still lead to convergent evolutions without length penalization.

Our particular interest of closed curves in $\mathbb{H}^2$ is due to the close connection to the Willmore energy of surfaces of revolution, namely 
\begin{equation}
 \label{eq:WillmoreElastic}
 \mathcal{E}(\gamma) = \frac{2}{\pi} \int_{S(\gamma)} H^2\diff A,
\end{equation}
where $S(\gamma)$ denotes the toroidal surface that arises from revolving $\gamma$ {about} the $x$-axis, see \cite{LangerSinger2} or \cite[Theorem 4.1]{DS18}. Moreover, {free elastica} in hyperbolic space define Willmore surfaces of revolution which were extensively investigated for instance in \cite{LiYau,MR2480063,MR2729304,Eichmann,2017arXiv170502177M}. In \cite{LangerSinger2} the relation between the Willmore energy and the elastic energy was used to show that the global minimum of the Willmore energy of all surfaces of revolution is attained at the {stereographic projection of the} Clifford torus. This torus can be obtained by the revolution of (any rescaled and translated version of){\begin{equation}
 \label{eq:CliffordElastica} 
 \tau(t) = \begin{pmatrix}
              0\\1
             \end{pmatrix}
             + \frac{1}{\sqrt{2}}
             \begin{pmatrix}
              \cos t\\\sin t
             \end{pmatrix}
\end{equation}}%
around the $x$-axis. This is the reason we call the  curve in \eqref{eq:CliffordElastica} \textit{Clifford elastica} in the following. Note that it is the global minimum of the elastic energy of closed curves in the hyperbolic plane. Consequences of the results of this paper for the Willmore flow of tori in Euclidean three-space will be the content of {future research}.
{
%A related and essentially open problem, studied for example in \cite{Blatt,chill2009,MR2911840,MR3182810,MR3263933,Blatt2}, is the formation of singularities of the $L^2$-gradient flow of the Willmore energy. {Moreover,} the Willmore flow preserves the property of revolution symmetry%for the time it exists
%, see \cite[p. 412]{Blatt}. A result related to our work with different topological obstructions for the evolving surface is given in \cite[Theorem 5.2]{MR2119722}.
}

The aforementioned articles together with \cite{LangerSinger,EichmannPhD} lay the methodological groundwork for our approach. We however work directly with the unpenalized flow, examining evolution and asymptotic behavior of the length. 
A large part of this examination will be an explicit parametrization and a close examination of constrained elastic curves in the hyperbolic plane. Previously found parametrizations, for instance in \cite{LangerSinger,Heller,2017arXiv170502177M}, either apply only for free elastica or are too general for our purposes.

\subsection{Overview and Main Results}

In the following we present the two main results of this paper: {The first one is the optimal Reilly-type inequality for curves with small energy in Theorem \ref{thm:main1} and its consequence, the convergence of the elastic flow for initial values with small energy in Theorem  \ref{thm:main2}. The second one is the existence of non-converging evolutions of initial values with higher energy in Theorem \ref{thm:nonconvergence}}.

First of all note that smooth long-time existence of the elastic flow
\begin{equation}
  \label{eq:Flow}
 \left\{ \begin{array}{rll}\partial_t f &= - \nabla_{\!L^2}\E_\lambda(f) , & \mbox{ on }\mathbb{S}^1 \times (0,\infty), \\
f(\cdot, 0) &= f_0, & \mbox{ on } \mathbb{S}^1.% \, ,
\end{array}\right. 
 \end{equation}
has already been proved for $\lambda \geq 0$, see {\cite[Theorem 1.1 (i)]{DS17}} for each $f_0 \in C^\infty(\mathbb{S}^1,\mathbb{H}^2)$. Another striking insight that \cite{DS17} reveals is that each solution $f \in C^\infty( \mathbb{S}^1 \times [0, \infty) ; \mathbb{H}^2)$ with uniformly-in-time bounded hyperbolic length already subconverges up to isometries of $\mathbb{H}^2$ and reparametrization. For details see Theorem \ref{thm:LTE}. For the penalized flow with $\lambda > 0$ the length is naturally bounded, since it is obviously controlled by the energy $\E_\lambda$. For $\lambda = 0$ however one has to answer the 
\textbf{main question}: 
\begin{center}
Do evolutions by unpenalized elastic flow in $\mathbb{H}^2$ have uniformly bounded hyperbolic length? 
\end{center}

To get a flavor for the question let us remind the reader of the necessity of length penalization in Euclidean space. If we start the flow with an initial curve that is a circle of radius $r_0$, then the solution of the elastic flow with $\lambda = 0$ is given by a circle of radius $r(t) = (r_0^4 + 2t)^\frac{1}{4}$, whose length is unbounded as $t \to \infty$. 
%Due to the compactness of the sphere one can show subconvergence without factoring out the corresponding isometries for $\lambda > 0$, see \cite[Theorem 1.1 (ii)]{Sphera}, but it is unknown whether the penalization is necessary.
In $\mathbb{H}^2$ one does not expect the same behavior as in the Euclidean case. 
This is mainly due to the fact that \emph{scaling}, i.e. $z \mapsto \theta z$ for some $\theta > 0$, is an isometry in the hyperbolic plane. In particular streching up the curve does not decrease the energy as it would do in the Euclidean case. 

In \cite{DS18} the evolution of circles under the flow \eqref{eq:Flow} was studied: It was shown that the unpenalized flow converges for each initial datum $f_0$ that parametrizes a circle. The limit is always an isometric image of \eqref{eq:CliffordElastica}. In the sphere it is unknown whether length penalization is necessary, but for $\lambda > 0$ convergence is shown in \cite[Theorem 1.1 (ii)]{Sphera}.
%and the following result concerning self-similar solutions was shown, which motivates our study of the unpenalized elastic flow.
 
%\begin{prop}[{\cite[Proposition 1.1]{DS18}}] \label{prop:ConvergentCircles}
% If the initial curve $f_0$ is a circle, then the unpenalized elastic flow stays circular during the flow and converges monotonically to the Clifford elastica \eqref{eq:CliffordElastica}. %Moreover, this convergence is monotone in the sense that the quotient of the $y$-coordinate of the midpoint and the radius (of the circle seen as a Euclidean circle) converges monotonically to the corresponding quotient of the Clifford elastica.
% \\
% This still holds for any $\lambda > -\frac{1}{2}$, but for nonzero $\lambda$ the limit circle is different.
% \end{prop}
Since the gradient flow structure yields a natural bound on the elastic energy $\mathcal{E}$, an idea would be to bound the quotient of elastic energy and length from below, to ensure that the length remains bounded along the flow as well. 
 Such a bound is called \textit{Reilly inequality} as it was first obtained in \cite{Reilly} for Euclidean hypersurfaces. A stronger version of the inequality also holds true in hypersurfaces in hyperbolic space \cite[Theoreme 1]{ElSoufi1992}, but in general not for curves (see \cite[Fig. 8]{LangerSinger}). Our first result shows that such an inequality holds below a certain energy level, which we have also shown to be sharp. {A similar result for open curves was obtained in \cite[Section 5]{Eichmann}.} 
 %Note that for the sphere such an inequality can not hold, as multiple coverings of the equator have zero elastic energy but arbitrary large length. 
 
\begin{theorem}[A Reilly-Type Inequality]\label{thm:main1} Let $\delta> 0 $. Then there exists $c_\delta > 0$ such that 
\begin{equation}\label{eq:2.1}
\inf\left\lbrace \frac{\mathcal{E}(\gamma)}{\mathcal{L}(\gamma)} : \gamma \in C^\infty( \mathbb{S}^1, \mathbb{H}^2 ), \gamma \; \textrm{immersed} , \; \mathcal{E}(\gamma) \leq 16 - \delta \right\rbrace \geq c_\delta .
\end{equation}
Moreover,
\begin{equation}\label{eq:2.2}
\inf\left\lbrace \frac{\mathcal{E}(\gamma)}{\mathcal{L}(\gamma)} : \gamma \in C^\infty( \mathbb{S}^1, \mathbb{H}^2 ), \gamma \; \textrm{immersed} , \; \mathcal{E}(\gamma) \leq 16 + \delta  \right\rbrace = 0  .
\end{equation}
\end{theorem}

The proof of Theorem \ref{thm:main1} is given in Section \ref{sec:MainResultProofs}. For the first part we apply the direct method in Section \ref{sec:Ratio} to show that the infimum of the elastic energy over curves with fixed length is attained at a constrained elastica, i.e. a critical point of $\mathcal{E} + \lambda \mathcal{L}$ for some $\lambda \in \R$. A detailed analysis of all possible constrained elastica in Section \ref{sec:explicit} and Section \ref{sec:closing} shows \eqref{eq:2.1}. {To show the second part we construct curves of large length with energy arbitrarily close to $16$ in Sections \ref{sec:Inv} and \ref{sec:Opt}}.

Note that we can replace $\mathcal{C}^\infty$ in Theorem \ref{thm:main1} by $W^{2,2}$, as we show this in Theorem \ref{thm:reilly} for \eqref{eq:2.1}, and this replacement is trivial for assertion \eqref{eq:2.2}. %Theorem \ref{thm:main1} shows that the energy level of 16 is the threshold for a Reilly inequality to hold. %The proof of Theorem \ref{thm:main1} is given in Section \ref{sec:MainResultProofs}.
{As an application of Theorem  \ref{thm:main1} we can show the convergence of the elastic flow for initial values below the energy level of 16.}

\begin{theorem}[Convergence of the Unpenalized Elastic Flow]%\leavevmode%
\label{thm:main2}
%\begin{enumerate}
% \item 
Let $f_0$ be a smooth immersed {closed} curve with $\mathcal{E}(f_0) \leq 16$, and denote its evolution by the {unpenalized} elastic flow by $(f_t)_{t \geq 0}$. Then $\mathcal{L}(f_t)$ is bounded on $[0,\infty)$ and $f_t$ converges to the Clifford elastica \eqref{eq:CliffordElastica} 
{in the sense of Remark \ref{rem:subconvergenceExplained1} (1).}
%\item 
%\end{enumerate}
\end{theorem}

 {From \eqref{eq:WillmoreElastic} it follows that elastic curves with elastic energy $16$ or less correspond to surfaces of revolution with Willmore energy $8 \pi$ or less. Hence the convergence result can be matched to the energy bound given in \cite{MR2119722}, where the authors show long time existence and convergence of the $L^2$-Gradient flow of the Willmore energy provided that the initial surface is an immersion of $\mathbb{S}^2$ (i.e. has genus 1) and has Willmore energy below $8\pi$.
The content of this paper differs from these results in two ways: Firstly, the Willmore flow of a surface of revolution and the $\mathbb{H}^2$-elastic flow of the profile curve differ by a factor, see \cite[Theorem 4.1]{DS18}. Secondly, since we consider the rotation of closed curves, the obtained surface of revolution is of different topology than $\mathbb{S}^2$, namely it has genus $0$.

{As already mentioned, the energy threshold in Theorem \ref{thm:main2} is optimal, which is discussed in the following theorem.}

\begin{theorem}[Nonconvergence of the Unpenalized Elastic Flow]
 \label{thm:nonconvergence}
For all $\varepsilon > 0 $ there exist a smooth initial curve $f_0$ with $16 < \E(f_0) < 16 + \varepsilon$ such that for its evolution $f_t$ by the {unpenalized} elastic flow we find that $\mathcal{L}(f_t)$ is unbounded. In particular the solution does not converge as $t \to \infty$.
\end{theorem}

The proofs are given in Section \ref{sec:MainResultProofs}.
% The main ingredient is a flow invariant identified in Section \ref{sec:Inv}. We show with the classification results from Section \ref{sec:explicit} and \ref{sec:closing}, that not all possible values of this flow invariant are attained by free elastica.
%In Section \ref{sec:Opt} we construct curves with energy arbitrarily close to 16, for which there exist no free elastica that have the same value of the invariant. One concludes that for these curves the flow cannot converge. By virtue of Theorem \ref{thm:LTE} the length has to be unbounded during the flow. The optimality statement for the Reilly-inequality (cf. \eqref{eq:2.2}) can be shown with a similar construction.
%Note however that the results differ in the way that the surfaces of revolution we construct are of genus $0$ and the Willmore Flow of surfaces of revolution differs from the hyperbolic elastic flow by a factor, see \cite{DS18}.}
%Similar to \cite{MR2119722} we use that the energy strictly below 16 ensures that the smooth curve is simple by an application of \cite[Theorem 6]{LiYau}, see Proposition \ref{prop:liyau}.}
As we mentioned before all circular evolutions converge, but their initial energy can be arbitrarily large, in particular larger than 16 (see \cite[Lemma 3.1]{DS18}), so evolutions of high initial energy do not necessarily have to be divergent. In Sections \ref{sec:Inv} and \ref{sec:Opt} we however identify a class of initial curves whose flow never converges, namely curves of \emph{vanishing total curvature}. The reason that these evolutions cannot converge is that the total curvature is a flow invariant and there is no free elastica of vanishing total curvature, hence there is no critical point available to converge to. We shall discuss this flow invariant in Section \ref{sec:Inv}. 

%The main preparation we need is an explicit parametrization of all elastica in $\mathbb{H}^2$, which we will provide in Section \ref{sec:explicit}. They will be given explicitly with multiple parameters. Since we are interested only in closed elastica, we derive parameter identites that ensure closedness in Section \ref{sec:closing}. Having these at hand we can prove the above theorems with the methods described above.  

%The rest of the article is organized as follows. In the next section, we explicitly calculate the parametrizations of constrained elastica %but postpone some technical proofs in the appendix.
%to classify the critical points. A helpful overview over all relevant parameters is given in Remark \ref{rem:helpTable}. We then use the findings of the subsequent section on closing conditions to show a more general version of the first part of Theorem \ref{thm:main1} in Section \ref{sec:Ratio}. %, from which the first part of Theorem \ref{thm:main2} follows.
%In Section \ref{sec:Inv}
% {we characterize the Euclidean total curvature of the classified elastica, and use the results to prove the {second part of Theorem \ref{thm:main1} as well as Theorem \ref{thm:nonconvergence}.}} in Section \ref{sec:MainResultProofs}.
\section{Explicit Parametrization of Elastica in the Hyperbolic Plane}
\label{sec:explicit}
In the following, we shall give an explicit parametrization of elastica in the hyperbolic plane, which we will then use throughout the rest of the paper. We will adapt many concepts from \cite{LangerSinger} most of which we will state for the reader's convenience. Here, our manifold of interest is the hyperbolic plane $\mathbb{H}^2 = \R \times (0,\infty)$ equipped with the usual metric tensor $g(x,y) = \frac{1}{y^2} \operatorname{id}$ (c.f. \cite[Subsection 2.1]{DS17}).
\subsection{The Elastica Equation}

Let $M$ be a smooth manifold. We denote by $\mathcal{V}(M)$ the set of all \emph{smooth vector fields} on $M$.
If $(M,g)$ is a Riemannian manifold with Levi-Civita connection~$\nabla$ and $c\colon  I \rightarrow M$ be an immersed curve  with velocity vector $c' \colon  I \rightarrow M$ then we denote by $T \colon  I \rightarrow M$ the \emph{unit tangential field} defined by % as follows: 
$T(t) := \sqrt{g_{c(t) } (c'(t) , c'(t) )}^{-1}  c'(t).$ 
We define the \emph{curvature vector field} $\overrightarrow{\kappa}[c] :=  \overrightarrow{\kappa} :I \rightarrow M$ {%to be
locally by} $\overrightarrow{\kappa}(t) := \nabla_T T$.

\begin{example}\label{ex:curvhyp}
Let $\gamma \in C^\infty((0,1); \mathbb{H}^2)$. Then $\overrightarrow{\kappa}[\gamma]$ is a vector field along $\gamma$. Therefore, if $\iota \colon  \mathbb{H}^2 \hookrightarrow \mathbb{R}^2$ denotes the canonical inclusion then $\overrightarrow{\kappa}[\gamma]$ can also be seen as a vector field in $\mathbb{R}^2$ along $\iota \circ \gamma = (\gamma_1, \gamma_2) $. The formula for $\overrightarrow{\kappa}$ reads
\begin{equation}\label{eq:curvhyper}
\overrightarrow{\kappa}[\gamma] = \begin{pmatrix}
\partial_s^2 \gamma_1 - \frac{2}{\gamma_2} \partial_s \gamma_1 \partial_s \gamma_2 \\ 
\partial_s^2 \gamma_2 + \frac{1}{\gamma_2} ( ( \partial_s \gamma_1 )^2 - ( \partial_s \gamma_2)^2) 
\end{pmatrix},
\end{equation}
where $\partial_s = \frac{\partial_x \gamma}{|\partial_x \gamma|_{\mathbb{H}^2} }= \frac{\gamma_2}{\sqrt{\gamma_1'^2 + \gamma_2'^2}} \partial_x \gamma$, see \cite[(12)]{DS17}.
\end{example}

%\begin{prop}[Normal Field Along a Curve in Hyperbolic Space] 
As $\mathbb{H}^2$ is diffeomorphic to the upper half plane with an orthogonality-preserving diffeomorphism, each smoothly immersed curve $c\colon I \rightarrow \mathbb{H}^2$ has a smooth normal vector field $N$ along $c$ such that 
\[ g_{c(t)} (T(t), N(t)) = 0 \quad\text{ and }\quad g_{c(t)} (N(t), N(t) ) =1 \text{ for each }t \in I.\]
The vector field $N$ is unique up to a sign. Note in particular that $\{T(t) ,N(t) \}$ forms an orthonormal basis of $T_{c(t)}\mathbb{H}^2$. We will now fix a choice of $N$ in the following
%\end{prop}

%\begin{proof}
%This follows immediately from the fact that $\mathbb{H}^2$ is diffeomorphic to the upper half plane with a diffeomorphism that preserves orthogonality. 
%\end{proof}

\begin{remark}\label{rem:4.4}
The normal field $N$ becomes unique when we prescribe that $ e^{i \frac{\pi}{2} } \iota \circ T = \iota \circ N $, where $\iota \colon \mathbb{H}^2 \hookrightarrow \mathbb{C}$ is the canonical inclusion. We will do this in the following and write $N = iT$ as shorthand notation, which makes actually sense when we look at the curve `with Euclidean eyes'. 
\end{remark}

%\begin{definition}[Scalar Curvature] \label{def:scalarcurve} 
For a smooth immersed curve $c\colon I \rightarrow \mathbb{H}^2$ one can not only define the curvature vector field $\overrightarrow{\kappa}$ but also the \emph{scalar curvature} $\kappa\colon I \rightarrow \mathbb{R}$  to be the unique function such that 
$\overrightarrow{\kappa} = \kappa N$, {see \cite[Section 4.4]{Willmore}}. Note that we sometimes write $\kappa[c]$ to emphasize the dependency of the curve $c$, see e.g. Proposition \ref{prop:charcurve}.
%\end{definition}

\begin{definition}[Elastic Energy]
Let $(M,g)$ be a Riemannian manifold, $\lambda \in \R$, and $c\colon I \rightarrow M$ be a smooth immersion. We define 
\begin{equation*}
\mathcal{E}_\lambda(c) := \int_{c(I)} (\kappa^2 + \lambda) \diff s  = \int_I (\kappa^2(t) + \lambda)\sqrt{g_{c(t)} (c'(t) ,c'(t) ) } \diff t.
\end{equation*} 
\end{definition}
\begin{definition}

We define 
\begin{equation*}
W^{2,2} ( \mathbb{S}^1, \mathbb{R}^2 ) := \{ \gamma \in W^{2,2}((0,1), \mathbb{R}^2) \colon \gamma(1) = \gamma(0) , \gamma'(1) = \gamma'(0) \}, 
\end{equation*}
where the point evaluations denote the evaluations of the representatives in $C^1([0,1])$. For each $z \in \mathbb{H}^2$ we define the set
\begin{equation*}
W^{2,2}_z ( \mathbb{S}^1, \mathbb{H}^2) := \{ \gamma \in W^{2,2}(\mathbb{S}^1, \mathbb{R}^2) \colon  \gamma(0) = z ,\gamma_2 > 0 \} .
\end{equation*}
and 
\begin{equation*}
W^{2,2}(\mathbb{S}^1, \mathbb{H}^2) := \bigcup_{z \in \mathbb{H}^2} W^{2,2}_z ( \mathbb{S}^1, \mathbb{H}^2). 
\end{equation*}
\end{definition}

\begin{remark}
Note that $\mathcal{E}_\lambda$ is well-defined also on $W^{2,2}( \mathbb{S}^1, \mathbb{H}^2)$ by using \eqref{eq:curvhyper} to make sense of $\kappa^2$.  However, note that the way we define it, there is no obvious metric on these sets without using a Nash embedding in the sense of \cite[Theorem 3]{Nash}.
 This definition of the Sobolev space might appear strange at first sight, but has the advantage that we can use the complex structure of $\R^2 \cong \mathbb C$. %, especially considering that they are defined differently in the literature. 

\end{remark}

The critical points of $\mathcal{E}_\lambda$ are called \emph{elastica} and satisfy the following Euler-Lagrange equation (see \cite[(1.3)]{LangerSinger}). 

\begin{definition}[Elastica in $\mathbb{H}^2$] \label{def:elastica}
A curve $\gamma \in C^\infty((0,L), \mathbb{H}^2)$ is called \emph{elastica} or \emph{elastic curve} in $\mathbb{H}^2$ if it is parametrized with hyperbolic arclength and satisfies 
\begin{equation}\label{eq:elasticaeq}
2 \kappa[\gamma]'' + \kappa[\gamma]^3 - (\lambda + 2)  \kappa[\gamma] = 0 
\end{equation}
for some $\lambda \in \mathbb{R}$. If $\lambda = 0 $ the curve is called \emph{free elastica}, otherwise it is called $\lambda$-\emph{constrained elastica} or just elastica.
\end{definition}

\begin{prop}[{\cite[page 6-7]{LangerSinger}}] 
Let $\gamma \in C^\infty((0,L), \mathbb{H}^2)$ be an elastic curve with curvature $\kappa = \kappa[\gamma]$ that is parametrized with hyperbolic arclength. Then 
%\begin{equation}
%2 \kappa'' + \kappa^3 - (\lambda + 2)  \kappa = 0. 
%\end{equation}
there exists a constant $C \in \mathbb{R}$ such that 
\begin{equation}\label{eq:onceint}
\kappa'^2 + \frac{1}{4} \kappa^4 - \frac{\lambda + 2}{2} \kappa^2 = C
\end{equation}
and $u := \kappa^2$ is a nonnegative solution of
\begin{equation}\label{eq:squaredelas}
u'^2 + u^3 - (2 \lambda + 4) u^2 - 4Cu  = 0.  
\end{equation}
\end{prop}

The elastica equation is solved explicitly in the following proposition, which is a major part of \cite{Steinberg} and has been obtained before in \cite{LangerSinger}.  A detailed proof is included in Appendix \ref{app:A} for the reader's convenience. 

\begin{prop}[Integration of the Elastica Equation {\cite{Steinberg,LangerSinger}}]\label{prop:integr} 
Let $\lambda \in \mathbb{R}$ be given. Then, every nonnegative solution $u = \kappa^2$ of \eqref{eq:squaredelas} is global and attains a global maximum $\kappa_0^2 := \sup_{x \in \mathbb{R}} u(x)$. Therefore, all nonnegative solutions of \eqref{eq:squaredelas} are translations of solutions with the following initial conditions $u(0) = \kappa_0^2 $ and $u'(0) = 0 $. Then, for $\kappa_0^2 < \lambda +2 $ there exist no elastica, and the other cases are exhaustively classified by the following four cases
\begin{enumerate}
\item (Circular Elastica) $\kappa_0^2 = \lambda + 2, \; C< 0$ and $u(s) = 2 + \lambda$. 
\item (Orbitlike Elastica) $ \kappa_0^2 \in ( \lambda + 2, 2 \lambda +4), \; C< 0$ and $u(s) = \kappa_0^2 \dn^2(rs,p),$ where $r = \frac{1}{2} \sqrt{\frac{2\lambda + 4}{2-p^2}}$ and $p \in (0,1)$ is such that $\kappa_0^2 = \frac{2\lambda + 4}{2-p^2}$, more explicitly $p^2 = \frac{2 \sqrt{(2+\lambda)^2 + 4C}}{2 + \lambda + \sqrt{(2+\lambda)^2 + 4C}}$.
\item (Asymptotically Geodesic Elastica) $\kappa_0^2 = 2\lambda + 4, \; C= 0$ and $u(s) = \kappa_0^2 \sech^2(rs)$, where $r = \frac{1}{2}\sqrt{2\lambda + 4}$
\item (Wavelike Elastica) $\kappa_0^2 > 2 \lambda + 4, \; C> 0 $ and $u(s) = \kappa_0^2 \cn(rs,p)$ where $r = \frac{1}{2} \sqrt{\frac{2\lambda + 4}{2p^2-1}}$ and $p \in ( \frac{1}{\sqrt{2}} , 1) $ is such that  $\kappa_0^2 = \frac{(2\lambda + 4) p^2}{2p^2-1}$, more explicitly $p^2 = \frac{2 + \lambda + \sqrt{(2+ \lambda)^2 + 4C}}{2 \sqrt{(2+\lambda)^2 + 4C}}$.
\end{enumerate}

%{
%\begin{equation*}
% u(s) = \begin{cases} 
% 2 + \lambda & \kappa_0^2 = \lambda + 2 \\
% \kappa_0^2 \dn^2(rs, p), \;  r= \frac{1}{2} \sqrt{\tfrac{2 \lambda + 4}{2-p^2}}, \; \textrm{ $p \in (0,1)$ s.t. } \kappa_0^2 = \frac{2 \lambda + 4}{2-p^2}&  
% \lambda+ 2<\kappa_0^2<2 \lambda + 4    \\
% \kappa_0^2 \sech^2(rs),\; r = \frac{1}{2}\sqrt{2 \lambda + 4} & \kappa_0^2 = 2\lambda + 4 \\
%\kappa_0^2 \cn^2(rs, p),\; r = \frac{1}{2}\sqrt{\frac{2\lambda + 4}{2p^2 - 1}}, \; 
%  \textrm{$p\in (\frac{1}{\sqrt{2}}, 1)$  s.t. } \kappa_0^2 = \frac{(2 \lambda + 4)p^2}{2 p^2 - 1}  &  \kappa_0^2 > 2\lambda + 4
%\end{cases}%
%\end{equation*}% 
%}
%where the first (\textit{circular}) and second (\textit{orbitlike}) cases apply when $C <0 $, the third (\textit{asymptotically geodesic}) case applies when $C = 0 $ and the last (\textit{wavelike}) case applies when $C > 0 $. Additionally,  the parameter $p$ is given by
%\begin{equation}\label{eq:modulus}
%p^2 = \frac{2 \sqrt{(2+\lambda)^2 + 4C}}{2 + \lambda + \sqrt{(2+\lambda)^2 + 4C}}
%\end{equation}
%in the orbitlike case and {in the wavelike case we have}
%\begin{equation*}
%p^2 = \frac{2+ \lambda + \sqrt{(2 + \lambda)^2 + 4C}}{2 \sqrt{(2+ \lambda)^2 + 4C }}.
%\end{equation*}
%%in the wavelike case. 
\end{prop}

%\begin{remark}
%From now on we will use the terms \emph{circular}, \emph{wavelike}, \emph{orbitlike} and \emph{asymptotically geodesic} to refer to the cases distinguished above. 
%\end{remark}

We want to derive an explicit parametrization for elastic curves. For this, we have to prescribe initial data. In Proposition \ref{prop:integr} we fixed the curvature and its derivative at $s=0$. Initial data for $\gamma(0)$ and $\gamma'(0)$ can be chosen in a computationally convenient way. This choice has to be made in a way that elastic curves with any initial data can be retrieved. One would hope that the retrieving process only involves isometries, since then the curvature changes only up to a sign.

In $\mathbb{R}^2$, Euclidean motions define isometries and { for each $p \in \mathbb{R}^2$ and $v \in T_p\mathbb{R}^2$ there exists} a Euclidean motion $\Phi$ such that $\Phi(p) = (0,0) $ and $\Phi(v) = (|v|,0)$. This means that each elastic curve with initial value $p$ and initial tangent vector $v$ is(Euclidean) isometric to an elastic curve starting at the origin with a horizontal tangent line. We shall prove a similar result for $\mathbb{H}^2$, inspired by \cite{EichmannMaster}. For this note that $\Phi : \mathbb{H}^2 \rightarrow \mathbb{H}^2$ is an isometry of $\mathbb{H}^2$ if and only if there exist $a,b,c,d \in \mathbb{R}$ such that $ad-bc = 1$ and 
\begin{equation*}
\iota \circ \Phi \circ \iota^{-1} (z) = \frac{az+b}{cz+d},
\end{equation*}
where $\iota : \mathbb{H}^2 \hookrightarrow \mathbb{C}$ denotes the canonical inclusion.

%\begin{prop}[{Isometries in $\mathbb{H}^2$, see e.g. \cite[Lemma 2.4]{EichmannPhD}}] 
%Let $\Phi\colon \mathbb{H}^2 \rightarrow \mathbb{H}^2$. Let $\iota \colon \mathbb{H}^2 \hookrightarrow \mathbb{C}$ denote the canonical inclusion. Then $\Phi$ is an isometry of $\mathbb{H}^2$  if and only if $\iota \circ \Phi \circ \iota^{-1}$ extends to an orientation preserving M\"obius transformation that fixes the real line, i.e. there are $a,b,c,d \in \mathbb{R}$ such that $ad-bc = 1 $ and 
%
%\end{prop}

\begin{lemma}[Reduction of the Initial Value Problem]\label{lem:inired}
Let $z \in \mathbb{H}^2$ and $v \in T_z \mathbb{H}^2$ such that $g_z(v,v) = 1$. Then for each $y > 0$ there exists an isometry $\Phi$ of $\mathbb{H}^2$ such that $\iota (\Phi(z))  = iy$ and $ d (\iota \circ \Phi)_z(v) = y $.
\end{lemma}
\begin{proof}We tacitly identify $\iota\circ \Phi \equiv \Phi$ and $ \iota (z) \equiv z $. We can without loss of generality assume that  $z = i r $ for some $r > 0 $ since we can compose with a translational M\"obius transformation that translates $z$ to the imaginary axis and leaves the differential invariant.
% Note that this is possible only since the differential of such a M\"obius transformation is the identity map. 
 Note that $\Phi(w) = \frac{aw+b}{cw+ d}$ for some $a,b,c,d \in \mathbb{R}$ and we identify $d \Phi_w$ with $\Phi'(w)$ via complex multiplication. We obtain that $\Phi$ is the desired M\"obius transformation if and only if 
\begin{enumerate}
\item $ad-bc = 1$.
\item $y = \Phi'(z) v = v \frac{ad-bc}{(cz+d)^2} = \frac{v}{(cz+d)^2}$.
\item $0 = \mathrm{Re} \; \Phi(z) = \frac{ac|z|^2 + bd+ (ad+bc)\mathrm{Re}(z)}{|cz+d|^2}$.
\item $y = \mathrm{Im} \;  \Phi(z) = \frac{ad-bc}{|cz+d|^2} \mathrm{Im}(z)  = \frac{ad-bc}{|cz+d|^2} r$
.
\end{enumerate} 
Note that condition (2) makes (4) redundant since  $g_z(v,v) = 1 $ implies that $|v|^2_{\mathbb{C}} = r^2 $. Indeed, if $(2)$ is satisfied then
\begin{equation*}
y = |y|_\mathbb{C} = \left\vert \frac{ad-bc}{(cz+d)^2} v \right\vert  = \frac{ad-bc}{|cp+d|^2} |v| = \frac{ad-bc}{|cz+d|^2}r, 
\end{equation*}
whence $(4)$ holds true.
Plugging $(1)$ into $(2)$ gives $\frac{1}{(cir+d)^2} = \frac{y}{v}$. 
Note that (2) can easily be solved for $c$ and $d$. Indeed, if $\sqrt{\cdot}$ denotes some branch of the complex root we obtain that $icr + d = \sqrt{\frac{y}{v}}$ and therefore $c=\frac{1}{r} \mathrm{ Im } \sqrt{\frac{y}{v}} $ and $d = \mathrm{Re} \sqrt{\frac{y}{v}}$.  Using that $\mathrm{Re}(z) = 0$ equations $(1)$ and $(3)$ yield the following linear system 
\begin{equation*}
\begin{cases} da-cb &  = 1 \\ r^2ca + db & = 0, \end{cases} 
\end{equation*} 
which has a unique solution once $c,d$ are known since 
\begin{equation*}
\det \begin{pmatrix}
d & -c \\ r^2 c & d 
\end{pmatrix} = d^2 + r^2 c^2 = \left\vert \sqrt{\frac{y}{v}} \right\vert^2 \neq 0. 
\end{equation*}
Finally we have found $a,b,c,d$ such that $(1),(2),(3),(4)$ are satisfied. The claim follows.
\end{proof}

\subsection{Killing Fields}

Let $(M,g)$ be a Riemannian manifold and $X \in \mathcal{V}(M)$. We define the \emph{flow map} $\phi_X$  of $X$ the map that associates to a pair $(t,p) \in \mathbb{R} \times M$ the value $c_p(t)\in M$ where $c_p$ is the unique maximal solution of 
\begin{equation}\label{eq:floo}
\begin{cases}
c'(s) = X(c(s) ) \\
c(0) = p.
\end{cases} 
\end{equation}  
Since it is unclear whether $c_p(t)$ exists for given $(t,p) \in \mathbb{R} \times M$, the domain of definition need not be $\mathbb{R} \times M$. 

%\begin{definition}[Integral Curve, Flow] 
%Let $(M,g)$ be a Riemannian manifold, $p \in M$ and $X \in \mathcal{V}(M)$. Then the solution of the Cauchy problem 
%\begin{equation*}
%\begin{cases}
%c'(t) = X(c(t) ) \\
%c(0) = p
%\end{cases} 
%\end{equation*}
%is called \emph{integral curve} (with respect to $X$ starting at $p$). Its maximal existence interval is called $(\varepsilon^{-}(p), \varepsilon^{+}(p))$. If $c_p \colon ( \varepsilon^{-}_p, \varepsilon^{+}_p) \rightarrow M $ is the integral curve with respect to $X$ starting at $p$ we can define a map 
%\begin{equation*}
%\phi_X\colon \bigcup_{p \in M}  (\varepsilon^{-}_p, \varepsilon^{+}_p) \times \{ p \} \rightarrow M , \; \phi_X(t,p) := c_p(t)
%\end{equation*}
%and call it the \textit{flow map} with respect to $X$. 
%\end{definition}
%\begin{definition}[Killing Field] 
A vector field $J \in \mathcal{V}(M)$ is called \emph{Killing field} for $M$ if for each $p \in M$ the $c_p$ is defined on the whole of $\mathbb{R}$ and $ \phi_t := \phi_J(t, \cdot)\colon M \rightarrow M $ is an isometry for each $t \in \mathbb{R}$. 
%\end{definition}

The reason that we introduce Killing fields is that they one can associate a Killing field $\widetilde{J}_\gamma$ to each given elastica $\gamma$. Since however Killing fields in $\mathbb{H}^2$ can also be characterized explicitly one obtains a representation of $\widetilde{J}_\gamma$ with three parameters. This can be used to perform an order reduction of the (fourth order) elastica equation. Details will be given in the following
\begin{lemma}[{Killing Fields for Elastica, \cite[Proposition 2.1]{LangerSinger}}] \label{thm:killingext}
Let $\gamma\colon I \rightarrow \mathbb{H}^2$ be an elastic curve parametrized by hyperbolic arclength. Define 
\begin{equation*}
J_\gamma := (\kappa^2 - \lambda) T + 2 \kappa' N .
\end{equation*} 
Then $J_\gamma$ has a unique extension to a Killing field in $\mathcal{V}(\mathbb{H}^2)$, which we will denote by $ \widetilde{J}_\gamma$. 
\end{lemma}
\begin{remark}\label{rem:nontrivialk}
{ Since
\begin{equation*}
 g_{\gamma(t)}( J_\gamma, J_\gamma ) = (\kappa^2- \lambda)^2 + 4\kappa'^2 = \kappa^4 - 2 \lambda \kappa^2 + \lambda^2 +4 \kappa'^2 = \lambda^2 + 4C + 4 \kappa^2  
\end{equation*}
where $C$ is the constant in \eqref{eq:onceint}, we see that $\widetilde{J}_\gamma \equiv 0$ implies that $\kappa \equiv const.$ and this case is already covered by Proposition \ref{prop:curvconst}. We infer that in the cases of elastica with nonconstant curvature the Killing field $\widetilde{J}_\gamma$ is not identically zero. }
\end{remark} 
%Most of the subsequent propositions are taken from \cite{EichmannPhD} and will be used extensively in the following.
\begin{prop}[{Killing Fields in $\mathbb{H}^2$, \cite[Example 2.10]{EichmannPhD}}]\label{prop:killinghyp}A vector field  $J \in \mathcal{V}(\mathbb{H}^2)$ is a Killing  field if and only if there are $a,b, c \in \mathbb{R}$ such that 
\begin{equation*}
J(x,y) = a \begin{pmatrix}
x^2 - y^2 \\ 2xy 
\end{pmatrix}  + b \begin{pmatrix}
x \\ y 
\end{pmatrix} + c \begin{pmatrix}
1 \\ 0 
\end{pmatrix}
\end{equation*} 
with respect to the Euclidean chart $\psi\colon \mathbb{H}^2 \rightarrow \mathbb{R}^2$, $\psi(x,y) := (x,y)^T$.
\end{prop}
For a given elastica $\gamma$, our goal is now to find the parameters $a,b,c$ that are associated to $\widetilde{J}_\gamma$ in the sense of Proposition \ref{prop:killinghyp}. The rest of this section will be dedicated to the following order reduction result, which is a slight refinement of \cite[Remark 4.6]{EichmannPhD}.

\begin{prop}[{Order Reduction}] \label{prop:ordred}
Let $\gamma \colon I \rightarrow \mathbb{H}^2$ be an elastic curve parametrized by hyperbolic arclength and $y > 0 $ be such that $\gamma(0) = (0,y)^T $, $\gamma'(0) = (y,0)^T$, $\kappa[\gamma](0) = \kappa_0$ and $\kappa[\gamma]'(0) = 0 $. Then either $\kappa \equiv const.$ or {$\kappa = \kappa[\gamma]$ satisfies
\begin{equation}\label{eq:diffeq}
(\kappa^2 - \lambda)\begin{pmatrix}
\gamma_1' \\ \gamma_2' 
\end{pmatrix} + 2 \kappa' \begin{pmatrix}
-\gamma_2'\\ \gamma_1' 
\end{pmatrix} = a \begin{pmatrix}
\gamma_1^2 - \gamma_2^2 \\ 2 \gamma_1 \gamma_2
\end{pmatrix}
+ c \begin{pmatrix}
1 \\ 0 
\end{pmatrix}
\end{equation} 
with constants $a,c \in \R$, }$a \neq 0 $ and 
\[- a y^2 + c = (\kappa_0^2 - \lambda) y \quad \text{ and }\quad ac = -\frac{1}{4} ( \lambda^2 + 4C).\] 
\end{prop}
\begin{remark}
Recall that the prescribed initial datum in Proposition \ref{prop:ordred} does not restrict the generality of the classification, see Lemma \ref{lem:inired}. Using M\"obius transformations might however change the parameters of the Killing field, hence the order reduction is exclusively applicable for elastica with the given initial data. 
\end{remark} 

For the proof of Proposition \ref{prop:ordred} it is crucial to examine the so-called \emph{characteristic integral curves} of an elastica. These are defined to be the solutions $c_z$ of \eqref{eq:floo}, where $X = \widetilde{J}_\gamma$ and $z \in \mathbb{H}^2$ is a point of maximum curvature of $\gamma$. 
First observe by \cite[Theorem 2.11, Remark 2.8]{EichmannPhD} that for each Killing field $J \in \mathcal{V}(M)$ and each $p \in M$ the solution $c_p$ of \eqref{eq:floo} is parametrized with constant velocity and has constant curvature. Hence, the following proposition provides therefore a description of all $c_p$:  

%\begin{prop}[{Integral Curves of Killing Fields, \cite[Theorem 2.11, Remark 2.8]{EichmannPhD}}]\label{prop:intcurve}
%Let $(M,g)$ be a Riemannian manifold, $p \in M$ and $J \in \mathcal{V}(M)$ a Killing field. Denote by $c_p$ the integral curve with respect to $J$ starting at $p$. Then $c_p$ is parametrized with constant velocity and $\kappa[c_p]$ is constant. 
%\end{prop}

\begin{prop}[{Curves of Constant Curvature in $\mathbb{H}^2$, \cite[Lemma 2.15] {EichmannPhD}}]\label{prop:curvconst}
Let $\gamma\colon I \rightarrow \mathbb{H}^2$ be a smooth immersed curve such that $\kappa[\gamma] \equiv const. $ Then one of the following holds true: 
\begin{enumerate}
\item If $|\kappa[\gamma]| > 1 $ then $\gamma$ is part of a (Euclidean) circle that does not intersect the $x$-axis.
\item If $|\kappa[\gamma]| = 1 $ then $\gamma$ is part of a (Euclidean) circle that touches the $x$-axis at exactly one point or part of a line parallel to the $x$-axis.
\item If $|\kappa[\gamma]| < 1 $ then $\gamma$  is part of a (Euclidean) circle intersecting the $x$-axis at exactly two points or part of a straight line that is not parallel to the $x$-axis. 
\end{enumerate}
\end{prop}

In particular we know that each characteristic integral curve $c_z$ must be of one of the three kinds.

To detect which of the three cases in the previous section applies to solutions $c_p$ one can use a useful criterion, going back to  \cite[p.81, Exercise 5b]{doCarmo} and \cite[Theorem 7.3]{EichmannMaster}. It says that for each Killing field $J$ on a connected and geodesically complete Riemannian Manifold $(M,g)$  that has a zero $q$ one has that $\phi_J$ preserves the geodesic distance to $q$, i.e. 
\begin{equation}\label{eq:killingzero}
\mathrm{dist}(q,p) = \mathrm{dist} ( q, \phi_J(t,p) )\quad  \forall t \in \mathbb{R} \; \forall p \in M .
\end{equation}
With all the information provided, we can compute the curvature of characteristic integral curves explicitly. One can view these characteristic integral curves as `external circle' for the given elastica.

%\begin{prop}[{Killing Fields with Zeroes, \cite[p.81, Exercise 5b]{doCarmo}, \cite[Theorem 7.3]{EichmannMaster}}]\label{prop:killingzero} 
%Let $(M,g)$ be a connected geodesically complete Riemannian manifold and $J \in \mathcal{V}(M)$ be a Killing field. If there exists a unique $q \in M$ such that $J(q ) = 0$ then 
%\begin{equation*}
%\mathrm{dist}(q,p) = \mathrm{dist} ( q, \phi_J(t,p) ) 
%\end{equation*}
%for all $p \in M$ and all $t \in \mathbb{R}$. Here, $\mathrm{dist}(\cdot,\cdot)$ denotes the geodesic distance between two points. 
%\end{prop}

%\begin{definition}[Characteristic Integral Curve]Let $\gamma\colon I \rightarrow \mathbb{H}^2$ be an elastic curve parametrized by hyperbolic arclength. If $p$ is a vertex of $\gamma$ of maximum squared curvature, i.e. $p = \gamma(t_0) $ for some $t_0$ such that $\kappa[\gamma] (t_0)^2 = \kappa_0^2, \kappa[\gamma]'(t_0) = 0 $, then we call $c_{p}:= \phi_{\widetilde{J}_\gamma} ( \cdot, p)\colon \mathbb{R} \rightarrow \mathbb{H}^2$  the \emph{characteristic integral curve} of $\gamma${ at $p$}. 
%\end{definition}

\begin{prop}[{\cite[Proposition 2.2] {LangerSinger}}]
Let\label{prop:charcurve} $\gamma \colon I \rightarrow \mathbb{H}^2$ be an elastic curve parametrized by hyperbolic arclength. If $z$ is a point of maximum curvature of $\gamma$ and $c_z$ is the characteristic integral curve of $\gamma$ at $z$ then for each $t_0 \in \mathbb{R}$ such that $\gamma(t_0) = z $, $c_z$ is tangential to $\gamma$ at $t = 0$ and 
\begin{equation*}
\kappa[c_z] \equiv \frac{2 \kappa[\gamma](t_0)}{(\kappa[\gamma](t_0)^2 - \lambda)},
\end{equation*}
 if we choose the normal field of $c_z$ such that the normal at $z$ coincides with the normal of $\gamma$. 
\end{prop}

%The following order reduction method is a slight refinement of \cite[Remark 4.6]{EichmannPhD} and the discussion afterwards. For the purpose of an explicit parametrization we need to analyze the relations of the parameters in detail.
  
%\begin{prop}[{Order Reduction}] \label{prop:ordred}
%Let $\gamma \colon I \rightarrow \mathbb{H}^2$ be an elastic curve parametrized by hyperbolic arclength and $y > 0 $ be such that $\gamma(0) = (0,y)^T $, $\gamma'(0) = (y,0)^T$, $\kappa[\gamma](0) = \kappa_0$ and $\kappa[\gamma]'(0) = 0 $. Then either $\kappa \equiv const.$ or {$\kappa = \kappa[\gamma]$ satisfies
%\begin{equation}\label{eq:diffeq}
%(\kappa^2 - \lambda)\begin{pmatrix}
%\gamma_1' \\ \gamma_2' 
%\end{pmatrix} + 2 \kappa' \begin{pmatrix}
%-\gamma_2'\\ \gamma_1' 
%\end{pmatrix} = a \begin{pmatrix}
%\gamma_1^2 - \gamma_2^2 \\ 2 \gamma_1 \gamma_2
%\end{pmatrix}
%+ c \begin{pmatrix}
%1 \\ 0 
%\end{pmatrix}
%\end{equation} 
%with constants $a,c \in \R$, }$a \neq 0 $ and 
%\[- a y^2 + c = (\kappa_0^2 - \lambda) y \quad \text{ and }\quad ac = -\frac{1}{4} ( \lambda^2 + 4C).\] 
%\end{prop}

\begin{proof}[{Proof of Proposition \ref{prop:ordred}}]
Let $\gamma$ be a non-circular elastic curve, i.e. $\kappa$ is nonconstant.
It follows from Proposition \ref{prop:killinghyp} and Lemma \ref{thm:killingext} that for some constants $a,b,c \in \mathbb{R}$ it holds
\begin{equation}\label{eq:diffeq'}
(\kappa^2 - \lambda)\begin{pmatrix}
\gamma_1' \\ \gamma_2' 
\end{pmatrix} + 2 \kappa' \begin{pmatrix}
-\gamma_2'\\ \gamma_1' 
\end{pmatrix} = a \begin{pmatrix}
\gamma_1^2 - \gamma_2^2 \\ 2 \gamma_1 \gamma_2
\end{pmatrix}+ b \begin{pmatrix}
\gamma_1 \\ \gamma_2 
\end{pmatrix}
+ c \begin{pmatrix}
1 \\ 0 
\end{pmatrix}.
\end{equation}
Evaluating \eqref{eq:diffeq'} at $t = 0 $ we find 
\begin{equation*}
(\kappa_0^2- \lambda)\begin{pmatrix}
 y \\ 0 
\end{pmatrix} = a \begin{pmatrix}
-y^2 \\ 0 
\end{pmatrix} + b \begin{pmatrix}
 0 \\ y 
\end{pmatrix} + c \begin{pmatrix}
1 \\ 0 
\end{pmatrix}.
\end{equation*}
The second line yields $b = 0 $ and the first line yields $ (\kappa_0^2 - \lambda) y = - ay^2 + c $.  Recall from Remark \ref{rem:nontrivialk} that $a$ and $c$ cannot be both zero at the same time.
Consider the characteristic integral curve of $\gamma$ at $z=(0,y)$, which we will call $c_0$ in the sequel. 
{The curvature of $c_0$ is by Proposition \ref{prop:charcurve} given by 
$\kappa [c_0] \equiv \frac{2 \kappa_0}{\kappa_0^2 - \lambda}$, %
in particular it is constant. {If $c_0$ is a line, then it must be parallel to the $x$-axis since $c_0$ is tangential at the vertex $\gamma(0)$ {by Proposition \ref{prop:charcurve}} and therefore  $c_0'(0) \parallel \gamma'(0) = (y,0)$.} Hence $c_{0,2}'(t) = 0$ for each $t$ and this implies, using $c_0'(t) = \widetilde{J}_\gamma(c_0(t))$ and Proposition \ref{prop:killinghyp}, that $a = 0 $. On the other hand, if $a = 0 $ one can deduce from the equation and Proposition \ref{prop:killinghyp} that $c_0'(t) = \widetilde{J}_\gamma(c_0(t)) = c (1,0)^T 
$ and therefore $c_0$ is } a line parallel to the $x$-axis. We infer that $a= 0 $ if and only if $c_0$ is a line parallel to the $x$-axis. In this case however, each integral curve to $\widetilde{J}_\gamma$ is a line parallel to the $x$-axis, as the Killing equation implies. Since lines parallel to the $x$-axis have curvature of $\pm 1$ we obtain 
\begin{equation}\label{eq:dingsdaa}
1 = |\kappa[c_0]|^2 = \frac{4 \kappa_0^2 }{(\kappa_0^2 - \lambda)^2}.
\end{equation}
Also note that 
%\begin{equation*}
$(\kappa_0^2- \lambda)^2 = \kappa_0^4 - 2 \lambda \kappa_0^2 + \lambda^2 = \lambda^2 + 4C + 4\kappa_0^2 $
%\end{equation*}
by the definition of $C$, see \eqref{eq:onceint}. But this implies together with \eqref{eq:dingsdaa} that 
\begin{equation*}
4 \kappa_0^2 = (\kappa_0^2 - \lambda)^2 = \lambda^2 + 4C+ 4 \kappa_0^2. 
\end{equation*} 
Therefore $\lambda^2 + 4C = 0$. In particular we obtain that $0 = ac = - \frac{1}{4}(\lambda^2 + 4C)$. 
In the remaining case $c_0$ is not a line and we find that $a \neq 0$. According to Proposition \ref{prop:curvconst} and since $c_0$ cannot be a line, $c_0$ must be part of a Euclidean circle through $(0,y)$,  which we can reparametrize the Euclidean way:
\begin{equation}
c_1(t) = \begin{pmatrix}
          0\\m
         \end{pmatrix} + r
\begin{pmatrix}
\cos(t) \\ \sin(t) 
\end{pmatrix}, \quad t \in \left( \frac{\pi}{2}- \varepsilon_1 , \frac{\pi}{2} + \varepsilon_2 \right),  
\end{equation}
where $m \in \mathbb{R}$ and $r > 0 $ are such that $m + r = y $. A short computation shows  
\begin{equation}\label{eq:kmax}
- \frac{m}{r} = \kappa[c_1]  = \frac{2 \kappa_0}{\kappa_0^2 - \lambda}.
\end{equation}
Additionally, there exists a diffeomorphism $\phi \in C^1(\mathbb{R}; \mathbb{R})$  such that $c_1(t) = c_0(\phi(t))$. 
From \cite[Theorem 2.11, Remark 2.8]{EichmannPhD} we infer $g(\widetilde{J}_\gamma(c_1(t)), \widetilde{J}_\gamma(c_1(t) ))  = \mathrm{const}$. 
Plugging in $t= \phi^{-1}(0)$ we obtain  from Lemma \ref{thm:killingext}
\begin{align}\label{eq:fw}
g(\widetilde{J}_\gamma(c_1(t)), \widetilde{J}_\gamma(c_1(t) ))& = g(\widetilde{J}_\gamma(c_0(0)), \widetilde{J}_\gamma(c_0(0) )) \nonumber\\  & = g({J}_\gamma(\gamma(0)), {J}_\gamma((\gamma(0) )) 
=(\kappa_0^2 - \lambda)^2.
\end{align}
We can compute this quantity in a different way, namely using Proposition \ref{prop:killinghyp}
\begin{align*}
g( \widetilde{J}_\gamma(x,y) , \widetilde{J}_\gamma(x,y) ) & = \frac{1}{y^2} \left( a(x^2 - y^2) + c)^2 + 4a^2 x^2 y^2 \right) \\ & = \frac{a^2(x^2+y^2)^2 + 2ac (x^2 - y^2 ) + c^2}{y^2}.
\end{align*}
Plugging in $x= r \cos(t) $ and $y = m + r \sin(t) $ and using $\cos^2(t)-\sin^2(t) = \cos(2t)$ and $ \sin^2(t) =  \frac{1- \cos(2t)}{2}$  we find together with \eqref{eq:fw} 
\begin{align*}
 (\kappa_0^2 - \lambda)^2  & =\frac{1}{ (m+r \sin(t) )^2 } \left[ a^2(r^2 + m^2)^2 - 2acm^2 + c^2 + 2a^2m^2r^2\right. \\ &  \qquad \quad + \left. \sin(t) ( 4a^2 (r^2 + m^2) mr - 2ac m r )+  \cos(2t) ( 2ac r^2 - 2 a^2 m^2 r^2 ) \right].
\end{align*}
Multiplying with the denominator and using once again that $ \sin^2(t) =  \frac{1- \cos(2t)}{2}$ we obtain 
\begin{align*}
&  a^2(r^2 + m^2)^2 - 2acm^2 + c^2 + 2a^2m^2r^2 +\sin(t) ( 4a^2 (r^2 + m^2) mr - 2ac m r ) \\ & \qquad \qquad + \cos(2t) ( 2ac r^2 - 2 a^2 m^2 r^2 )\\  & \quad  = (\kappa_0^2 - \lambda)^2 ( m^2 + \frac{r^2}{2} ) + 2m r (\kappa_0^2 - \lambda)^2\sin(t) - \cos(2t) ( \kappa_0^2 - \lambda)^2\frac{r^2}{2} .
\end{align*}
Because of the identity theorem for holomorphic functions, the above identity holds true for any $t \in \mathbb{C}$. Using linear independence of trigonometric polynomials to compare coefficients we find
\begin{equation}\label{eq:koeffvgl}
\begin{cases}
\frac{-(\kappa_0^2 - \lambda)^2}{2} = 2 ac - 2 a^2m^2 \\
m(\kappa_0^2 - \lambda)^2 = 2 a^2 m  (m^2 + r^2) - 2ac m \\
a^2 (r^2+m^2)^2 - 2acm^2 + c^2 + 2 a^2m^2 r^2 = ( \kappa_0^2 - \lambda)^2 ( m^2 + \frac{r^2}{2} ). 
\end{cases}
\end{equation}
In case that $m \neq 0 $, dividing the second equation by $m$ and summing with the first we find
\begin{equation}\label{eq:radius}
r^2 = \frac{(\kappa_0^2 - \lambda)^2}{4a^2}.
\end{equation}
Using $\frac{m^2}{r^2} = \frac{4 \kappa_0^2}{( \kappa_0^2 - \lambda)^2 }$ we obtain $m^2 = \frac{\kappa_0^2}{a^2}$. Plugging this into the third identity in \eqref{eq:koeffvgl} we find
\begin{align*}
0 = & \frac{1}{16} [ (\kappa_0^2 - \lambda)^2 + 4 \kappa_0^2 ]^2 + \frac{1}{2} \kappa_0^2 ( \kappa_0^2 - \lambda)^2 - ( \kappa_0^2 - \lambda)^2 \left( \frac{(\kappa_0^2 - \lambda)^2}{8} + \kappa_0^2 \right) \\
 & - 2 (ac) \kappa_0^2 + (ac)^2  .
\end{align*}
Completing the square and factoring out the brackets we obtain
%\begin{equation}
$(ac - \kappa_0^2)^2 = \frac{( \kappa_0^2 - \lambda)^4}{16}$,
%\end{equation}
and eventually 
\begin{equation*}
a c = \kappa_0^2 \pm \frac{(\kappa_0^2- \lambda)^2}{4}.
\end{equation*}
We will continue showing that the case `$-$' always applies. Suppose that `$+$' is true for some elastica $\gamma$. Then $ac> 0 $ which implies that $\widetilde{J}_\gamma$ has a zero on the $y$-axis since 
\begin{equation*}
g_{(0,s)}( \widetilde{J}_\gamma , \widetilde{J}_\gamma ) = \frac{a^2 s^4 - 2ac s^2 - c^2 }{s^2} = \frac{(as^2 -c)^2}{s^2}.
\end{equation*}
%Note that then
Thus, $\eqref{eq:killingzero}$ implies that the characteristic integral curve cannot reach the $x$-axis. In particular we obtain that $m> 0 $. Note also that
\begin{equation*}
ac = \kappa_0^2 + \frac{(\kappa_0^2 - \lambda)^2}{4} = (m^2 + r^2)a^2. 
\end{equation*}
Computing the absolute value of the Killing field at $\gamma(0) = c_1( \frac{\pi}{2})=(0,m+r)^T $ we obtain that 
\begin{equation*}
(\kappa_0^2 - \lambda)^2 = \frac{a^2(m+r) ^4 - 2 a c(m+r)^2 + c^2}{(m+r)^2} = \frac{1}{(m+r)^2} a^2 ( m^2 + r^2 - (m+r)^2)^2 .
\end{equation*}
Using $\eqref{eq:radius}$ we find 
\begin{equation*}
4r^2 = \frac{(m^2+ r^2 - (m+r)^2)^2}{(m+r)^2} = 4 m^2 r^2  \frac{1}{(m+r)^2} = 4r^2  \frac{1}{(1 + \frac{r}{m})^2}
\end{equation*}
which is impossible, hence `$-$' is always true and %Note that then 
%\begin{equation*}
$ac =  \kappa_0^2 - \frac{(\kappa_0^2- \lambda)^2}{4} = - \frac{1}{4} ( \lambda^2 + 4C)$,
%\end{equation*}
 where $C$ is the constant of \eqref{eq:onceint}.
 The last case to consider is $m = 0 $, but this would imply together with \eqref{eq:kmax} that $\kappa_0 = 0 $. Therefore, there exists some $t_0$ such that $\kappa(t_0) = \kappa'(t_0) = 0 $. Using that $\kappa'' = - \frac{1}{2} \kappa^3 - \frac{\lambda + 2}{2} \kappa $ we find that also  $\kappa''(t_0) = 0 $ and bootstrapping we find that every derivative of $\kappa$ attains the value $0$ at $t_0$. Since all the solutions for $u = \kappa^2$ extend to a holomorphic function on an open neighborhood of the real line, we infer that $\kappa \equiv 0 $, contradicting the non-circularity. 
\end{proof}

Having eliminated the parameter $b$ and expressed $a,c$ with quantities in Proposition \ref{prop:integr} we can now use the quantities to classify elastica by their Killing fields.

\begin{definition}[Classification of Elastica by their Killing Field]\label{def:classi} 
Let $\gamma\colon I \rightarrow \mathbb{H}^2$ be an elastic curve parametrized with hyperbolic arclength and the same initial data as in Proposition \ref{prop:ordred}. If the extended Killing field of $\gamma$ is given by
\begin{equation}\label{eq:kilextf}
\widetilde{J}_\gamma(x,y)  = a \begin{pmatrix}
x^2 - y^2 \\ 2xy 
\end{pmatrix} + c \begin{pmatrix}
1 \\ 0 
\end{pmatrix}
\end{equation}
 we say that $\gamma$ has a \emph{rotational Killing field} (or simply\emph{ $\gamma$ is rotational})  if $ac > 0 $, a \emph{translational Killing field} if $ac < 0 $ and a \emph{horocyclical Killing field} if $ac = 0 $. 
\end{definition}

\begin{remark} \label{rem:helpTable}
By now we have introduced some parameters that describe elastic curves, which we will use in the following. For the sake of the reader's convenience we include in Table \ref{table:parameters} the references that will be missing. We always consider $\lambda \in \mathbb{R}$, $\kappa_0 \in \mathbb{R}$, $y > 0 $ to be the `original' parameters from which we compute the following new parameters. We also include those which %have yet to be defined.
will be defined later.
\end{remark}

\begin{table}[ht]
\caption{Parameters describing elastica.}
\centering
{\footnotesize
\begin{tabular}{p{1.6cm}p{1.6cm}p{7.5cm}}
%\toprule
Parameter & Reference & Description 
\\ \toprule
$C$ & \eqref{eq:onceint} & Integration constant for the integrated elastica equation \\ 
\midrule
$r$ & Prop. \ref{prop:integr} & Periodicity determining parameter in solutions of elastica equation \\ 
\midrule
$p$ & Prop. \ref{prop:integr} & Shape determining parameter in solutions of elastica equation, so-called modulus.  \\ 
\midrule
$a,c$  &  \eqref{eq:kilextf} & Parameters determining the Killing field of an elastica \\  
\midrule
$n$ & Remark \ref{rem:n} & (Only for closed elastica) %Number of 
Periods of the curvature
\\ 
\midrule
$m$ & \eqref{eq:period} & (Only for closed rotational elastica)  The winding number of $\gamma$ %with respect to
w.r.t. the zero of the Killing field, see \eqref{eq:dingsk}, \eqref{eq:dingski} \\

\bottomrule

\end{tabular}
}
\label{table:parameters}

\end{table}

We conclude this section with some useful facts about elastica with rotational Killing field which will turn out to be the most relevant for the proof of \eqref{eq:2.1}. 

\begin{prop}[Similarity of Integral Curves]
Let $\gamma$ be as in Proposition \ref{prop:ordred}. 
If $\gamma$ has a rotational Killing field, then every characteristic integral curve $c$ of $\gamma$ has curvature $|\kappa[c]| > 1 $.
\end{prop}
\begin{proof}
If $ac > 0 $ then $\widetilde{J}_\gamma$ has a unique zero on the positive $y$-axis in $\mathbb{H}^2$. Indeed, $\widetilde{J}_\gamma(0, y) = (-ay^2 + c,0)$ according to Proposition \ref{prop:ordred}.  Equation \eqref{eq:killingzero} implies that every Killing vector field must be a line parallel to the $x$-axis or a circle that lies completely in $\mathbb{H}^2$.  We have discussed in Proposition \ref{prop:ordred} that a line parallel to the $x$-axis is impossible unless $a = 0$, so it has to be a circle, resulting in $|\kappa[c]| > 1$, see Proposition \ref{prop:curvconst}. 
\end{proof}

\begin{prop}\label{prop:rotkil}
Assume that $\gamma$ is an elastica with a rotational Killing field. Then $\lambda > -1 $. 
\end{prop}

\begin{proof}
If $\gamma$ has a rotational Killing field the integral curve $c$ starting at $\gamma(0)$ satisfies 
\begin{equation*}
 1 < |\kappa[c]| = \frac{2 |\kappa_0|}{\kappa_0^2 - \lambda} .
\end{equation*}
Therefore 
%\begin{equation}
$\lambda > \kappa_0^2 - 2 |\kappa_0|$, and
%\end{equation}
adding $1$ to both sides we get 
\begin{equation*}
\lambda + 1 > |\kappa_0|^2 - 2 |\kappa_0| + 1 = ( |\kappa_0|  - 1) ^2 \geq 0 . \qedhere
\end{equation*}
%The claim follows.
\end{proof}

\subsection{Explicit Parametrization}
So far, we have reduced the elastica equation to a first order system, see Proposition \ref{prop:ordred}. To get an explicit parametrization we exploit the structure of this system further: Remarkably, the system becomes separable if we rewrite it as an equation in $\mathbb{C}$. The reason is, that the Killing fields come from isometries in $\mathbb{H}^2$, all of which are also isometries in the Riemann sphere $\mathbb{C}\mathbb{P}^1$. This being a Riemann surface, the Killing fields should have some holomorphic structure. The elastica equation has been examined 
in a more general setting using this structure in \cite{Heller}, where the author provides explicit parametrizations of elastica in arbitrary space forms. Unfortunately these parametrizations are not very useful examining the limiting cases, as we will do in the later sections. A slight disadvantage of our approach is that we can only parametrize globally defined elastic curves. Since our main focus lies on closed elastic curves, this is not restrictive for our application. From now on, we will assume unless not explicitly stated otherwise, that $\gamma \in C^\infty(\mathbb{R}, \mathbb{H}^2)$ is a globally defined smooth immersed elastic curve. 

\begin{prop}[Nonvanishing of the Killing Field]\label{prop:nonvan}
Let $\gamma\colon  \mathbb{R} \rightarrow \mathbb{H}^2$ be a globally defined elastic curve. Then  $\widetilde{J}_\gamma$ 
has no zeroes on $\gamma(\mathbb{R})$, and also
\begin{equation*}
\theta(s) := \kappa^2 - \lambda + 2 i \kappa' 
\end{equation*}
satisfies $\theta(s) \neq 0$ for all $s \in\mathbb{R}$.  
\end{prop}

The proof is a laborious extension of the same conceptual flavor of the next proof. For readability we postpone the proof to Appendix \ref{appendix_theta}.

\begin{theorem}[An Explicit Parametrization of Elastica in $\mathbb{H}^2$]\label{thm:explpara} Let $\gamma\colon  \mathbb{R} \rightarrow \mathbb{H}^2$ be an elastic curve parametrized by hyperbolic arclength with curvature  $\kappa$ so that
\begin{equation}
\kappa'^2 + \frac{1}{4} \kappa^4 - \frac{\lambda + 2}{2} \kappa^2 = C , \;  \kappa(0)^2 = \kappa_0,\; \gamma(0) = (0,y)^T \;\text{ and }\;\gamma'(0) = (y,0)^T
\end{equation}
for $y > 0$. %We consider $\mathbb{H}^2 \subset \mathbb{C}$ and$\gamma\colon  I \rightarrow \mathbb{C}$. 
 \\Then either $\kappa \equiv const. $ or there exist $a,c \in \mathbb{R}$ with {$|a|+ |c| \neq 0$} determined uniquely by $ac = - \frac{1}{4} (\lambda^2 +4C)$, $-a y^2 + c = (\kappa_0^2 - \lambda) y$ such that (if $\mathbb{H}^2$ is identified with $\{ z \in \mathbb{C} : \mathrm{Im}(z) > 0 \}$) 
\begin{equation} \label{eq:reduced}
\gamma' = (a \gamma^2 + c ) \frac{1}{\kappa^2 - \lambda + 2 i \kappa'}.
\end{equation}
Moreover, there exists $z_1 \in i \mathbb{R}_{\neq 0 } $ such that $\gamma$ is parametrized by
\begin{equation}\label{eq:explisoloi}
\gamma(t) =  f  \left(  \int_0^t \frac{1}{\theta(s)} \diff s   + z_1 \right) \quad \textrm{where} \; \; \; \theta(s) = \kappa^2(s) -  \lambda + 2 i \kappa'(s),
\end{equation}
and the meromorphic function $f$ is given by
\begin{equation}\label{eq:Fallunterscheidung}
f(z) = 
\begin{cases} \sqrt{\frac{c}{a}} \tan (\sqrt{ac} z)   & a,c > 0 \\
-  \sqrt{\frac{c}{a}}  \cot(\sqrt{ac} z) & a, c < 0 \\ 
\sqrt{\frac{-c}{a}} \tanh (\sqrt{-ac} z) & ac < 0 \\
\frac{1}{az} & c = 0, a \neq 0 \\ 
cz & c \neq 0 , a = 0.
\end{cases}
\end{equation} 
\end{theorem}

\begin{proof}
Rewrite \eqref{eq:diffeq} as equation of complex numbers to obtain 
\begin{equation}\label{eq:wichtig}
(\kappa^2 - \lambda) T + 2 \kappa' i T = a \gamma^2 + c,
\end{equation}
where $a,c$ are as in the statement. With the relations in the statement it is an easy computation that $y$ determines $a,c$ uniquely. 
Now parametrization by arclength implies that $T = \gamma'$. This immediately implies \eqref{eq:reduced} using Proposition \ref{prop:nonvan}. Note that \eqref{eq:reduced} is a first order ODE with locally Lipschitz right hand side, so its maximal solution is unique when we specify $\gamma(0)= iy$. Let now $y > 0 $ be given. We verify now that \eqref{eq:explisoloi} indeed yields a solution of \eqref{eq:reduced} with $\gamma(0) = iy$. For this we need to distinguish several cases depending on the value of $a,c$, which depend on $y$ according to the paramter identities in the statement.  \textbf{Case 1 $a,c > 0$.} We take $f$ as given in \eqref{eq:Fallunterscheidung} and differentiate the expression \eqref{eq:explisoloi} yields for $\gamma$. Defining $z_0 := \sqrt{ac} z_1$ we compute

%Especially, if we can find a maximal solution for each $y$, the uniqueness theorem yields that we have found all possible solutions. Solving the ODE as if it were real yields an ansatz for  \eqref{eq:Fallunterscheidung}. To verify the ansatz in the case of $a,c >0$, define $z_0 := \sqrt{ac} z_1$ and compute
\begin{align*}
\gamma'(t) &=  \frac{\diff}{\diff t} \sqrt{\frac{c}{a}} \tan \left( \sqrt{ac} \int_0^t \frac{1}{(\kappa^2 - \lambda) + 2i \kappa' } \diff s  + z_0 \right) \\ &   = \frac{1}{(\kappa^2 - \lambda) + 2i \kappa'} \sqrt{ac} \sqrt{\frac{c}{a}} \left( 1 + \tan^2 \left( \sqrt{ac} \int_0^t \frac{1}{(\kappa^2 - \lambda) + 2i \kappa' } \diff s  + z_0 \right) \right) \\
&  =  \frac{c}{(\kappa^2 - \lambda) + 2 i \kappa'} \left( 1 + \frac{a}{c}\gamma^2(t) \right) = \frac{a \gamma(t)^2 + c}{(\kappa^2 - \lambda) + 2 i \kappa'}.
\end{align*}
Hence, $\gamma$ indeed solves the equation. It remains to show that $z_1 [ = \frac{z_0}{\sqrt{ac}}] \in \mathbb{C}$ can be chosen such that 
\begin{equation}\label{eq:AW}
i y = \gamma(0) = \sqrt{\frac{c}{a}} \tan(z_0 ) = \sqrt{\frac{c}{a}} \frac{\tanh(iz_0)}{i}.
\end{equation}

To find such $z_0$ we need to invert the $\tanh$ in the expression. For this we first observe that $y < \sqrt{\frac{c}{a}}$, which can easily be obtained by the parameter identities in the statement and the fact that by Proposition \ref{prop:integr} $\kappa_0^2- \lambda > 0$.  %For this observe that $(\kappa_0^2- \lambda) y = -ay^2 + c$ and $a,c > 0 $  imply that 
%\begin{equation}
%$\frac{c}{y} = a y + \kappa_0^2 - \lambda  > ay$ because by Proposition \ref{prop:integr} $\kappa_0^2 - \lambda \geq 2$ 
%\end{equation}  
%and therefore  $y^2 < \frac{c}{a}$. 
Choosing $ z_0 = - i \mathrm{Artanh} \big( - \frac{y}{\sqrt{\nicefrac{c}{a}}} \big)$ implies \eqref{eq:AW}. \textbf{Case 2}, $a,c <0$, works similarly. In this case one can observe that $y  > \sqrt{\frac{c}{a}}$.  %and recall that $y = \sqrt{\frac{c}{a}}$ is impossible since otherwise $\gamma(0)$ would be a zero of $\widetilde{J}_\gamma$, which contradicts Proposition \ref{prop:nonvan}.
 {\textbf{Case 3,} $ac< 0 $. Defining $z_0 := \sqrt{-ac} z_1$ one obtains} 
\begin{align*}
& \frac{\diff}{\diff t} \sqrt{\frac{-c}{a}} \tanh \left( \sqrt{-ac} \int_0^t \frac{1}{(\kappa^2 - \lambda) + 2i \kappa' } \diff s  + z_0 \right) \\ & \quad  = \frac{1}{(\kappa^2 - \lambda) + 2i \kappa'} \sqrt{-ac} \sqrt{\frac{-c}{a}} \left( 1 - \tanh^2 \left( \sqrt{-ac} \int_0^t \frac{1}{(\kappa^2 - \lambda) + 2i \kappa' } \diff s  + z_0 \right) \right) \\
& \quad =  \frac{c}{(\kappa^2 - \lambda) + 2 i \kappa'} \left( 1 + \frac{a}{c}\gamma^2(t) \right) = \frac{a \gamma(t)^2 + c}{(\kappa^2 - \lambda) + 2 i \kappa'}.
\end{align*}
Note that 
%\begin{equation}
$\sqrt{-\frac{c}{a}}\tanh(z_0) = i y$
%\end{equation}
can be solved using that $\tanh(z_0) = i \tan(-i z_0) $. In the end we obtain {$z_0 = i \arctan \frac{\sqrt{a}y}{\sqrt{-c}}$}. The other cases can be solved analogously. 
\end{proof}

\begin{remark}
%Theorem \ref{thm:explpara} gives only necessary and no sufficient representations for globally defined elastica.
 Note that not every curve given by \eqref{eq:explisoloi} for some $\kappa$ from Proposition \ref{prop:integr} is an elastica. The reason for that is that any such curves are not necessarily parametrized by hyperbolic arclength. {We shall see counterexamples in Appendix \ref{app:kleeblatt}.} This makes the analysis more complicated. 
\end{remark}

\section{Closing and Simplicity Conditions}
\label{sec:closing}
In this section we want to investigate whether the elastic curves parametrized in Theorem \ref{thm:explpara} are (smoothly) closed, i.e. whether there is some $L> 0 $  such that $\gamma(0) = \gamma(L)$ and all derivatives of $\gamma$ coincide at $0$ and $L$. Since we consider elastica parametrized by hyperbolic arclength the smallest such $L$ is given by the hyperbolic arclength of $\gamma$. The other property of our interest will be simplicity, i.e. whether the curve has no self-intersections in $(0,L)$. The following propositions will reveal why we are interested in these properties: They are related to the energy and to the number of periods the curvature completes in one period of the curve and will be useful. 

%\begin{remark}
%Let $\gamma$ be a closed elastica, and denote by $L > 0$
%the smallest number such that $\gamma(0) = \gamma(L)$ and all derivatives of $\gamma$ at $ t= L$ coincide with the corresponding derivatives at $ t= 0 $. This number is unique since a real-valued continuous function can at most have one minimal period. 
%\end{remark}
%\begin{remark}
%Note that elastica are by definition parametrized by hyperbolic arclength and therefore the period $L$ coincides with the hyperbolic length of $\gamma$ until it closes. We will therefore also denote this hyperbolic length by $L$.%from now on.
%%
%\end{remark}

\begin{prop}\label{prop:liyau}
Let $\gamma$ be a closed curve. If $\gamma$ is not simple then $\mathcal{E}(\gamma) \geq 16 $. 
\end{prop}

\begin{proof}
A well-known relation  (e.g. \cite[p.532]{LangerSinger2}) shows \eqref{eq:WillmoreElastic}.
Now $S(\gamma)$ is not embedded and therefore \cite[Theorem 6]{LiYau} implies that
$\int_{S(\gamma)} H^2\diff A \geq 8 \pi $. All in all we obtain  %\begin{equation}
$\mathcal{E}(\gamma) \geq  \frac{2}{\pi}  \cdot 8 \pi =  16.$% \qedhere
%\end{equation}
\end{proof}
Since \label{rem:n}the curvature of $\gamma$ also periodic, we may denote with $n$ the number of periods the curvature completes within $[0,L]$, i.e. $n$ is given by 
\begin{equation}\label{eq:enn}
n := \operatorname{sup}\left\{m \in \N: \kappa\left(s + \frac{L}{m}\right) = \kappa(s) \text{ for all } s \in \R\right\}.
\end{equation}
 Note that $n$ is finite if $\kappa$ is non-constant.
\begin{prop}[Closed Simple Elastica]\label{prop:periodorb}
Let $\gamma$ be a closed simple elastica. Then either $\kappa \equiv const.$ or $n \geq 2 $.
\end{prop}

\begin{proof}
This is a direct consequence of the four-vertex theorem in hyperbolic space, see \cite[Lemma 2.1]{Ghomi}, and Proposition \ref{prop:integr}. Indeed, if $n = 1$ then 
%the curvature $\kappa$ only one period  by Remark \ref{rem:n}. From Proposition \ref{prop:integr} can now 
 it can be inferred that in all cases of Proposition \ref{prop:integr} $\kappa$ attains at most two critical values in one period.
\end{proof}

The following closing conditions are closely related to the conditions in \cite{Steinberg}. We present a self-contained proof for the reader's convenience. First note that there are no closed asymptotically geodesic elastica as the curvature of such is not periodic (cf. Proposition \ref{prop:integr}). For wavelike and orbitlike elastica we can derive closing conditions also employing their Killing fields.
%\begin{lemma}
%There exist no asymptotically geodesic closed elastica.
%\end{lemma}
%\begin{proof}
% Immediate by Proposition \ref{prop:integr}.
%\end{proof}

\begin{prop}[Closing Condition for Rotational Elastica]\label{prop:closedrot} %Let $\gamma$ be an elastic curve with rotational Killing field, hyperbolic length $L$ and curvature $\kappa$. 
Let  $\gamma$ be a rotational elastica (cf. Definition \ref{def:classi}) with hyperbolic length $L$ and curvature $\kappa$.
Then $\gamma$ is orbitlike, i.e. $C< 0$. Moreover, $\gamma$ is closed if and only if 
\begin{equation}\label{eq:period}
\int_0^L  \sqrt{\frac{- \lambda^2 - 4C}{4}}\frac{\kappa^2 - \lambda }{\lambda^2 + 4C + 4 \kappa^2 } \diff s   =  \pi m 
\end{equation} 
for some $m \in \mathbb{Z}$ and 
\begin{equation}\label{eq:L}
L = 2 n \frac{K(p) }{r},
\end{equation}
where $p, r $ are given in Proposition \ref{prop:integr}{, $n \in \mathbb{N} $ is given in \eqref{eq:enn} and $K$ denotes the complete elliptic integral of first kind, see Appendix \ref{appendix:Jacobi}}.  Moreover, if $m = 0$ then $n =1$. In case that $|m|>1$ and $n >1 $,  $|m|$ and $n$ are relatively prime.
\end{prop}

\begin{proof}
Having a rotational Killing field implies that $ac > 0$ and therefore $-ac= \frac{1}{4} (\lambda^2 + 4C)  < 0$ by Theorem \ref{thm:explpara}. Thus, $C< 0 $ and therefore $\gamma$ is orbitlike. Since $\kappa(0) = \kappa(L) = \kappa_0$ the formula for $L$ follows from Proposition \ref{prop:integr} and Proposition \ref{prop:identities} (4). An easy computation shows that for two complex numbers $z, w \in \mathbb{C}$, $\tan(z) = \tan(w)$ holds if and only if $z = w + m \pi $ for some $m \in \mathbb{Z}$. The same periodicity holds true for the complex cotangent function.{ Therefore, in the relevant first two cases of \eqref{eq:Fallunterscheidung}, the function $f$ in Theorem \ref{thm:explpara} is $\frac{\pi}{\sqrt{ac}}$-periodic.
  Using Theorem \ref{thm:explpara}, we obtain another necessary condition, namely
%   \begin{align}\label{eq:aequiv}
%   & \gamma(0) = \gamma(L) \nonumber  \\
%   \Leftrightarrow\qquad  & f \left( \int_0^L \frac{1}{\theta(s)} ds + z_1 \right) = f(z_1) \nonumber \\ \Leftrightarrow \qquad  & \int_0^L \frac{1}{\theta(s)} ds  = m \frac{\pi}{\sqrt{ac}}
%   \end{align}
  \begin{equation}\label{eq:aequiv}
   \gamma(0) = \gamma(L)  \Leftrightarrow   f \left( \int_0^L \frac{1}{\theta(s)} ds + z_1 \right) = f(z_1)   \Leftrightarrow   \int_0^L \frac{1}{\theta(s)} ds  = m \frac{\pi}{\sqrt{ac}}
  \end{equation}
  for some $m \in \mathbb{N}$. Rearranging we obtain}
 \begin{align}
 m \pi  & =   \sqrt{ac} \int_0^L \frac{1}{( \kappa^2 - \lambda) + 2i \kappa'} \diff s  
% & = \sqrt{ac} \int_0^L \frac{(\kappa^2 - \lambda) - 2i \kappa'}{(\kappa^2 - \lambda)^2 + 4 (\kappa')^2} \\
  \;=\;  \sqrt{ac} \int_0^L \frac{(\kappa^2- \lambda)- 2i \kappa' }{4 \kappa'^2 + \kappa^4 - 2\lambda \kappa^2 + \lambda ^2} \diff s  \nonumber \\
   & = \sqrt{ac} \int_0^L \frac{(\kappa^2- \lambda)- 2i \kappa'}{\lambda^2 + 4C + 4 \kappa^2 } \diff s  \label{eq:simi}\\ 
   & =  \sqrt{ac} \left( \int_0^L \frac{\kappa^2 - \lambda}{\lambda^2 + 4C + 4 \kappa^2} \diff s  - 2 i \int_0^L \frac{\kappa'}{\lambda^2 + 4C + 4 \kappa^2} \diff s   \right). \nonumber
 \end{align}
 Substituting $u = \kappa(s)$ and using $\kappa(0) = \kappa(L)$ we find that the last integral is zero. This substitution is justified when we use that $\lambda^2 + 4C + 4\kappa^2>0 $ because of Proposition \ref{prop:nonvan}. We obtain that \eqref{eq:period} and \eqref{eq:L} are necessary for the closedness of $\gamma$, and their sufficiency follows from Theorem \ref{thm:explpara}. Note that in the case of $m = 0$, the $\frac{2K(p)}{r}$-periodicity of  $\kappa^2=\kappa_0^2 \operatorname{dn}(r\cdot, p)$ (see \eqref{eq:PeriodDN}) implies that 
 \begin{equation*}
 0 = \int_0^L \sqrt{\frac{- \lambda^2 - 4C}{4}}\frac{\kappa^2 - \lambda }{\lambda^2 + 4C + 4 \kappa^2 } \diff s   = n \int_0^\frac{2K(p)}{r} \sqrt{\frac{- \lambda^2 - 4C}{4}}\frac{\kappa^2 - \lambda }{\lambda^2 + 4C + 4 \kappa^2 } \diff s  ,
 \end{equation*}
 and therefore 
\begin{equation*}
\int_0^\frac{2K(p)}{r} \sqrt{\frac{- \lambda^2 - 4C}{4}}\frac{\kappa^2 - \lambda }{\lambda^2 + 4C + 4 \kappa^2 } \diff s  = 0 .
\end{equation*} 
Due to the sufficiency of \eqref{eq:period} and \eqref{eq:L}, $\gamma$ closes smoothly after $\frac{2K(p)}{r}$. Since $n$ has to be minimal, we infer that $n =1$. Now assume that $|m|$ and $n$ are larger than 1. If they had a common prime factor $q> 1 $ then
 \begin{equation*}
q  \int_0^\frac{2 \frac{n}{q}
K(p)}{r} \frac{\kappa^2- \lambda}{\lambda^2 + 4C + 4 \kappa^2} \diff s = \int_0^\frac{2 n K(p)}{r} \frac{\kappa^2- \lambda}{\lambda^2 + 4C + 4 \kappa^2} \diff s  = \pi m 
 \end{equation*}
 and dividing by $q$ we obtain
 \begin{equation*}
  \int_0^\frac{2 \frac{n}{q}
K(p)}{r} \frac{\kappa^2- \lambda}{\lambda^2 + 4C + 4 \kappa^2} \diff s = \pi \frac{m}{q}.
 \end{equation*}
 Since $\frac{m}{q}$ and $\frac{n}{q}$ are integers, the sufficiency of \eqref{eq:period} and \eqref{eq:L} again implies that $\gamma$ already closes up smoothly after $\frac{n}{q}$ periods, contradicting the minimality of $n$. It follows that  $\mathrm{gcd}(|m|,n) = 1$. 
\end{proof}

\begin{prop}[Closing and Non-Simplicity of Wavelike Elastica] \label{prop:waveclose}\label{cor:noclwv}
Let $\gamma$ be a wavelike elastic curve and $p,r$ be as in Proposition \ref{prop:integr}. Then $\gamma$ is closed if and only if 
\begin{equation}\label{eq:clwv}
L = \frac{4K(p)}{r} \quad \text{ and }\quad \int_0^L \frac{\kappa^2- \lambda }{\lambda^2 + 4C + 4\kappa^2} \diff s  = 0.
\end{equation}%
Additionally, $\gamma$ is not simple. Furthermore, $\gamma$ can not be free, i.e. $\lambda = 0$ is not allowed.  
\end{prop} 

\begin{figure}[ht]
    \centering
    \begin{subfigure}[b]{0.48\textwidth}
        \includegraphics[width=\textwidth]{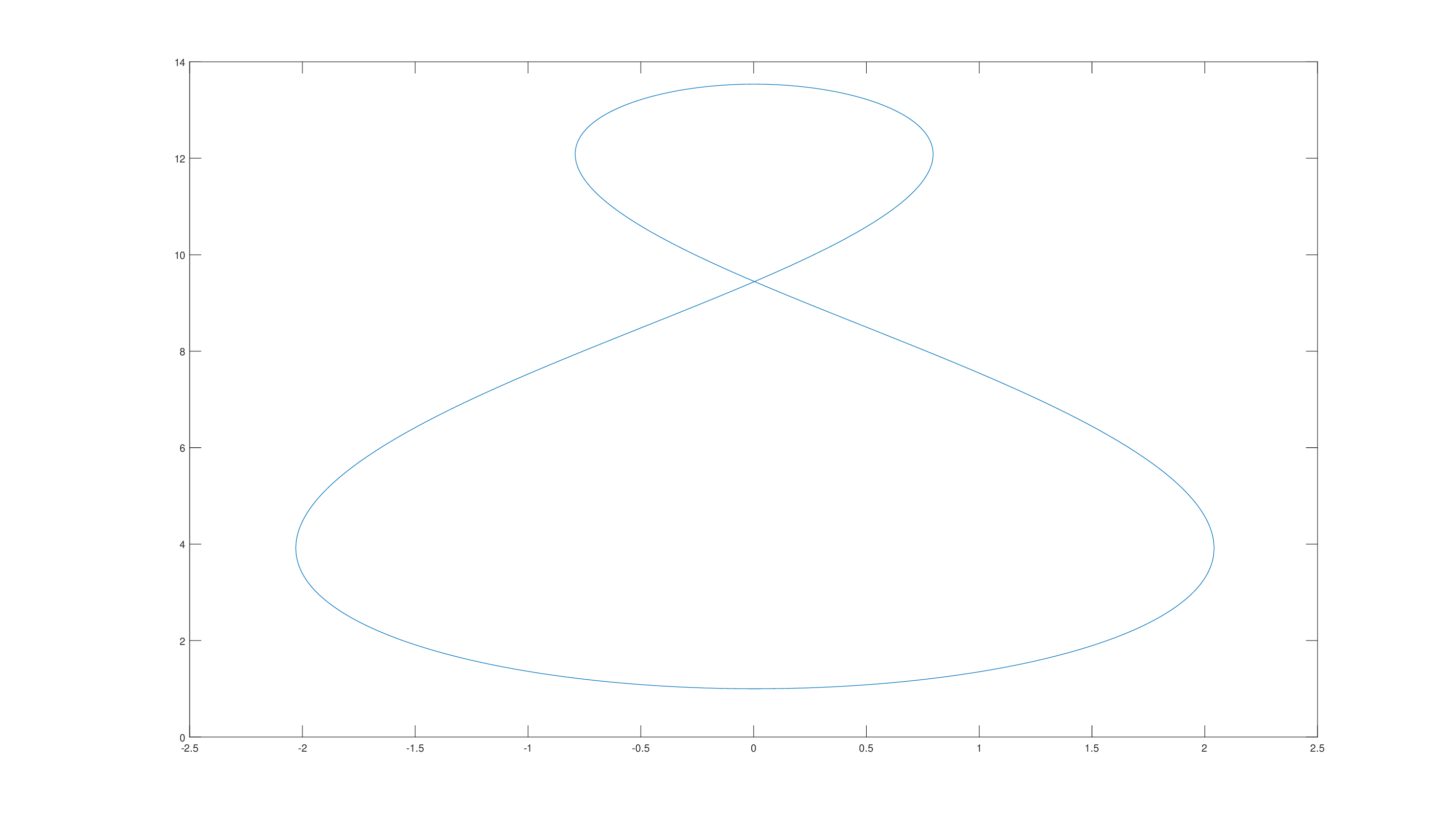}
        \caption{$\lambda=0.6, C= 0.36$}
%        \label{fig:gull}
    \end{subfigure}
    ~
    \begin{subfigure}[b]{0.48\textwidth}
        \includegraphics[width=\textwidth]{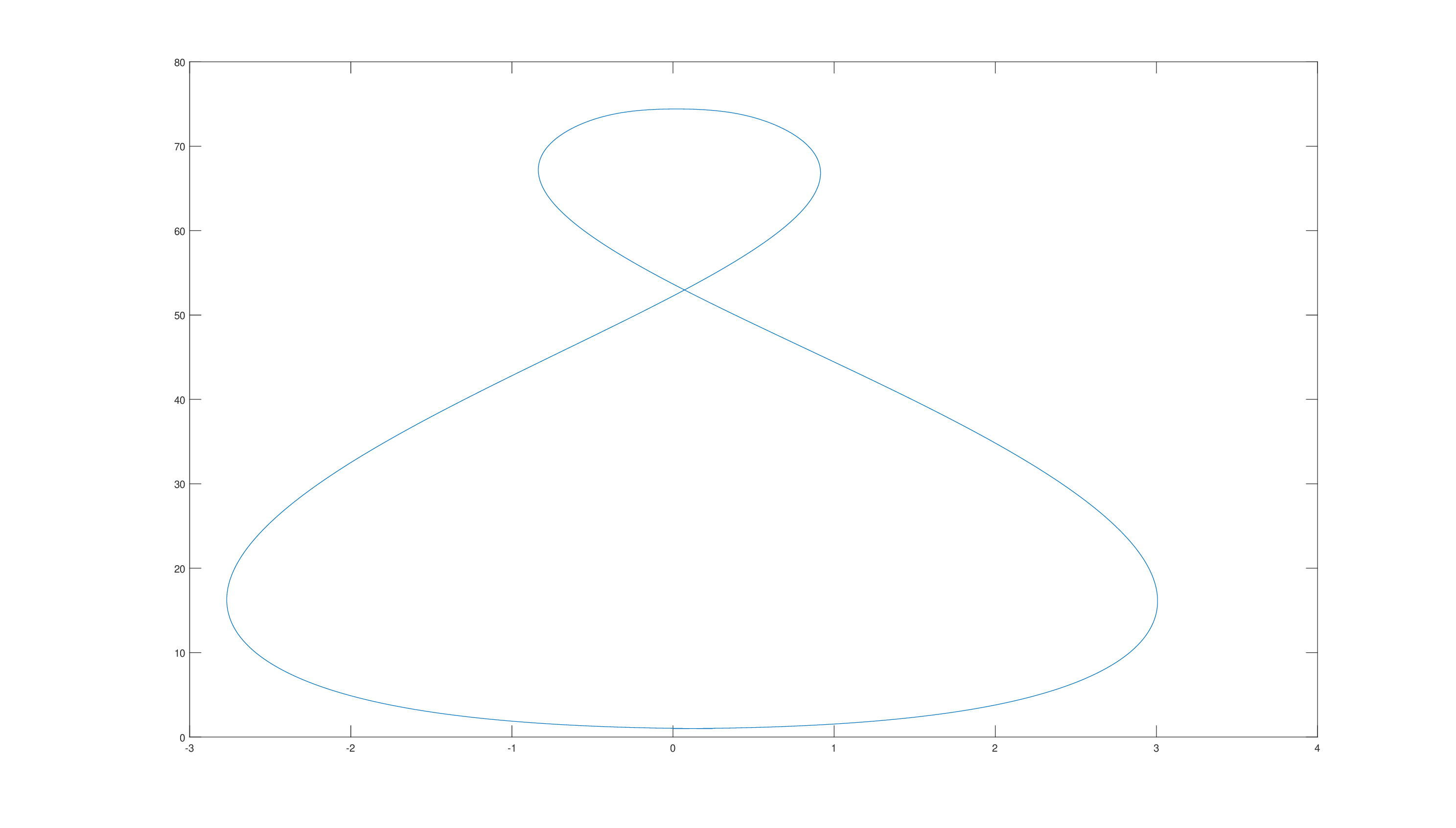}
        \caption{$\lambda =0.01, C= 0.00248$}
    \end{subfigure}
    \caption{Closed Wavelike Elastica. Note that both curves are scaled such that they pass through $(0,1)$, but for the smaller $\lambda$ the curve passes through $(0,\approx75)$ compared to $(0,\approx13.5)$ for the larger $\lambda$.}
    %\label{fig:OrbitlikeFreeElastica}
\end{figure}

\begin{proof}
If $\gamma$ is wavelike, then $C > 0 $ and therefore the Killing field of $\gamma$ is translational, see Proposition \ref{prop:ordred} and Definition \ref{def:classi}. We claim that $\kappa(s) = \kappa_0 \cn(rs, p) $. Indeed, from Proposition \ref{prop:integr} it follows that $\kappa(s) = \pm \kappa_0 \cn(rs,p) $, where `$+$' or `$-$' could be chosen differently for each $s$. However, any nonconsistent choice of sign would contradict the smoothness of $\gamma$. It follows now from the periodicity of $\cn$ that $L = \frac{4n K(p)}{r}$. As an intermediate claim, we show that $n = 1$. Since $\tanh$ has no real period we find with Theorem \ref{thm:explpara}, proceeding as in \eqref{eq:aequiv} and \eqref{eq:simi} that $\gamma$ is not closed unless 
\begin{equation}\label{eq:ueb1}
0 = \sqrt{ac} \int_0^L \frac{1}{(\kappa^2 -\lambda) + 2i \kappa' } \diff s  = \sqrt{ac }\int_0^L \frac{\kappa^2 - \lambda}{\lambda^2 + 4C + 4 \kappa^2} \diff s , 
\end{equation}
 as a computation similar to \eqref{eq:simi} reveals. Note now that $\kappa^2 = \kappa_0^2 \cn^2(rs, p)$ is $2 \frac{K(p)}{r}$-periodic. Therefore 
 \begin{equation} \label{eq:ueb2}
 0 = \int_0^L \frac{\kappa^2 - \lambda}{\lambda^2 + 4C + 4 \kappa^2} \diff s   = n \int_0^{\frac{4K(p)}{r}} \frac{\kappa^2 - \lambda}{\lambda^2 + 4 C + 4 \kappa^2  }\diff s,
\end{equation} 
 which implies using Theorem \ref{thm:explpara} that $\gamma$ already closes up smoothly at $L = \frac{4 K(p)}{r}$. The intermediate claim follows.  Furthermore,
 \begin{equation*}
 0 = \int_0^L \frac{\kappa^2 - \lambda}{\lambda^2 + 4C + 4 \kappa^2} \diff s   = 2 \int_0^{\frac{2K(p)}{r}} \frac{\kappa^2 - \lambda}{\lambda^2 + 4 C + 4 \kappa^2 }\diff s.
\end{equation*}  
We conclude that $\gamma ( \frac{2 K(p)}{r} ) = \gamma(0)$ and therefore $\gamma$ has a self-intersection. Hence, $\gamma$ is not simple. To show that $\gamma$ is not free we look at \eqref{eq:clwv}  with $\lambda = 0$. In order for this equation to be satisfied, one would need $\kappa \equiv 0$, which is not possible. 
\end{proof}

%\begin{cor}
%There are no closed free wavelike elastica. 
%\end{cor}
%
%\begin{proof}
%The only solution to \eqref{eq:clwv} with $\lambda = 0 $ and $C > 0 $ is $\kappa \equiv 0$. 
%\end{proof}

\begin{prop}[Closing and Simplicity for Translational and Horocyclical Orbitlike Elastica]\label{prop:orbclose}
Let $\gamma$ be a non-rotational orbitlike elastica. Then $\gamma$ is closed if and only if $n= 1 $ and 
\begin{equation*}
\int_0^L \frac{\kappa^2- \lambda}{\lambda^2+ 4C + 4 \kappa^2 } \diff s = 0 .
\end{equation*}
Additionally, $\gamma$ is not simple. 
\end{prop}

\begin{proof} The claim that $n = 1$ can be shown following the lines of the corresponding part of the proof of Proposition \ref{prop:waveclose}, more precisely one proceeds similar to \eqref{eq:ueb1} and \eqref{eq:ueb2} with the minor difference 
of the periodicity of $\mathrm{cn}$ instead of $\mathrm{dn}$, see Proposition \ref{prop:ellipt}.
%that now because of the periodicity of $\mathrm{dn}$, see Proposition \ref{prop:ellipt}, differs from the periodicity of $\mathrm{cn}$.
If $\gamma$ were simple, Proposition \ref{prop:periodorb} would imply that $n \geq 2$ which is a contradiction.
\end{proof}

 \begin{prop}[Closed orbitlike elastica] \label{prop:simcloorbit}
Let $\gamma$ be an orbitlike elastica with parameter $\lambda < \frac{64}{\pi^2} - 2\approx4.48$. If $\gamma$ is closed, then $\gamma$ is rotational and satisfies $m \neq 0$. 
 \end{prop}

 \begin{figure}[ht]
    \centering
    \begin{subfigure}[b]{0.48\textwidth}
        \includegraphics[width=\textwidth]{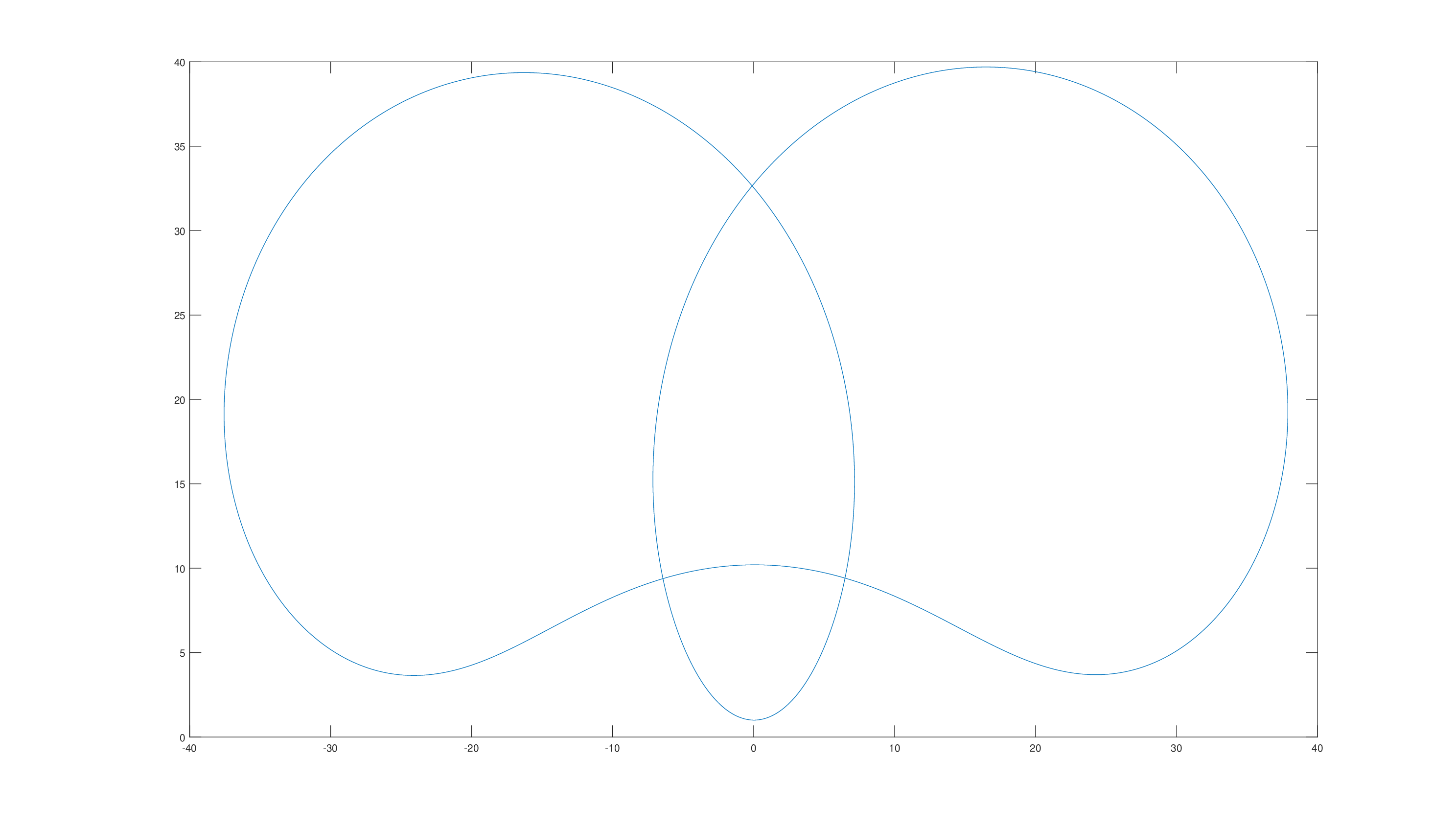}
        \caption{$\lambda = 0, C= -0.394317, n = 3$}
%        \label{fig:gull}
    \end{subfigure}
    ~
    \begin{subfigure}[b]{0.48\textwidth}
        \includegraphics[width=\textwidth]{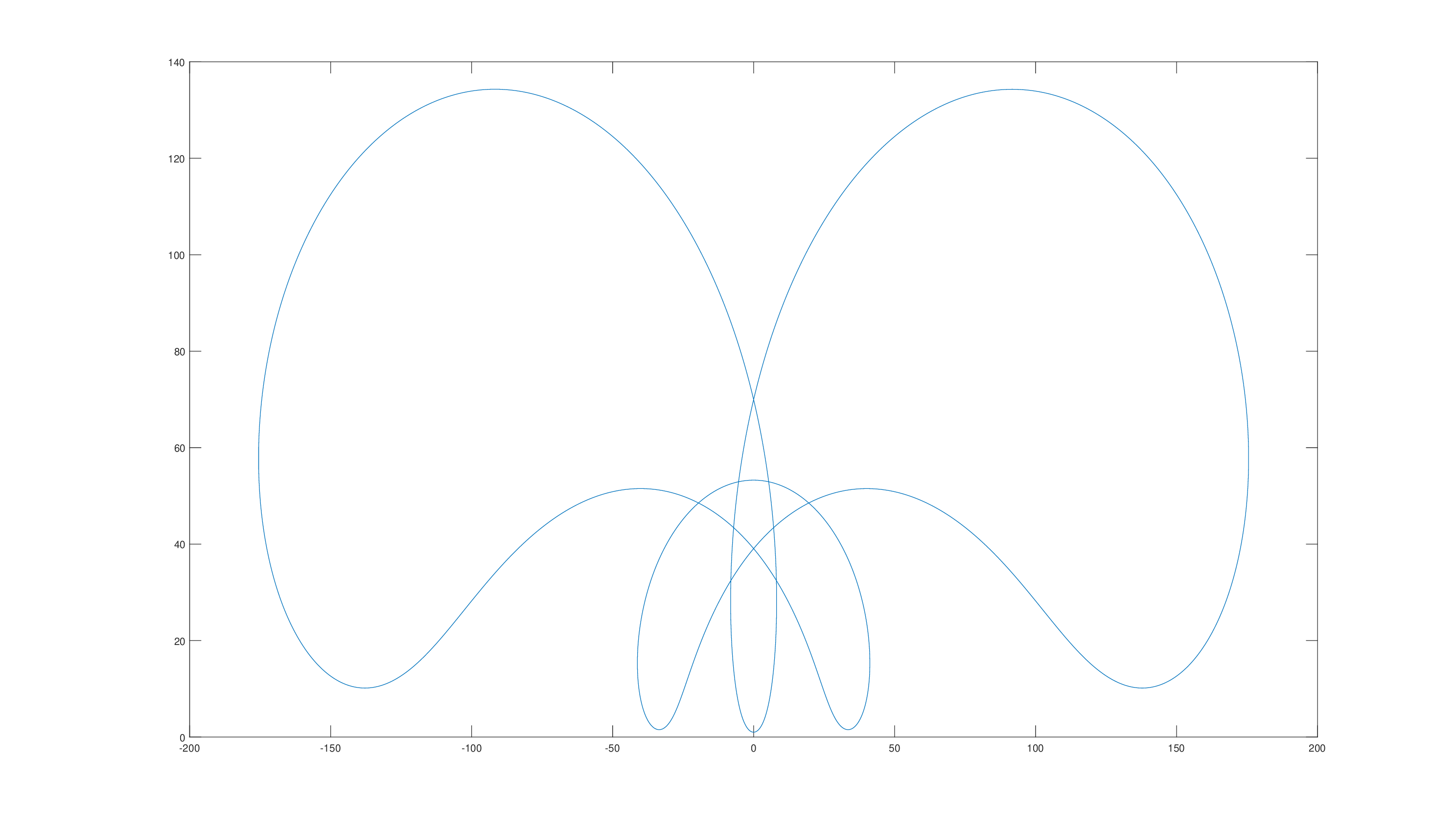}
        \caption{$\lambda =0, C= -0.0635, n = 5$}
    \end{subfigure}
    
    \begin{subfigure}[b]{0.48\textwidth}
        \includegraphics[width=\textwidth]{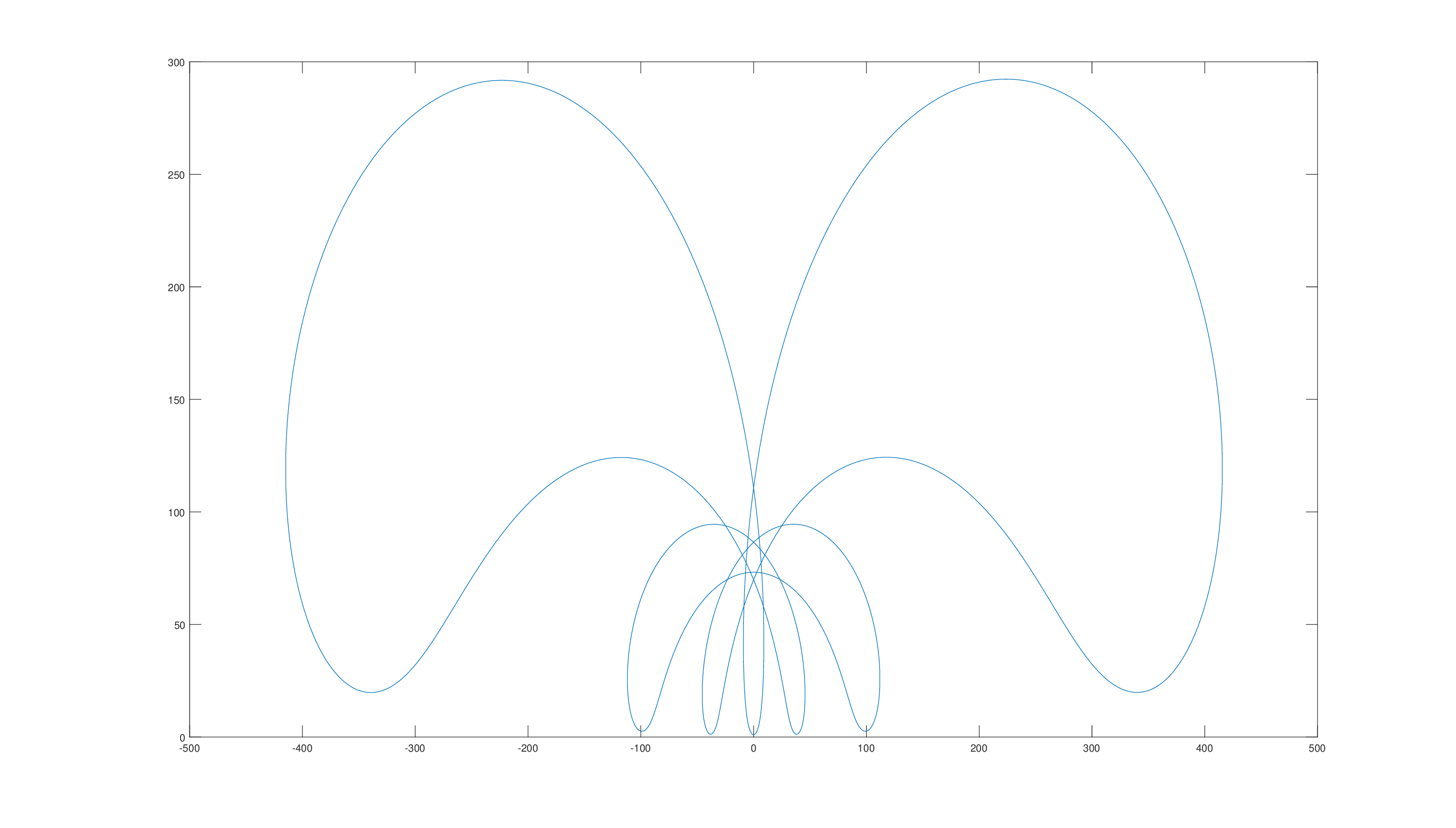}
        \caption{$\lambda = 0, C= -0.0225, n=7$}
    \end{subfigure}
    ~
    \begin{subfigure}[b]{0.48\textwidth}
        \includegraphics[width=\textwidth]{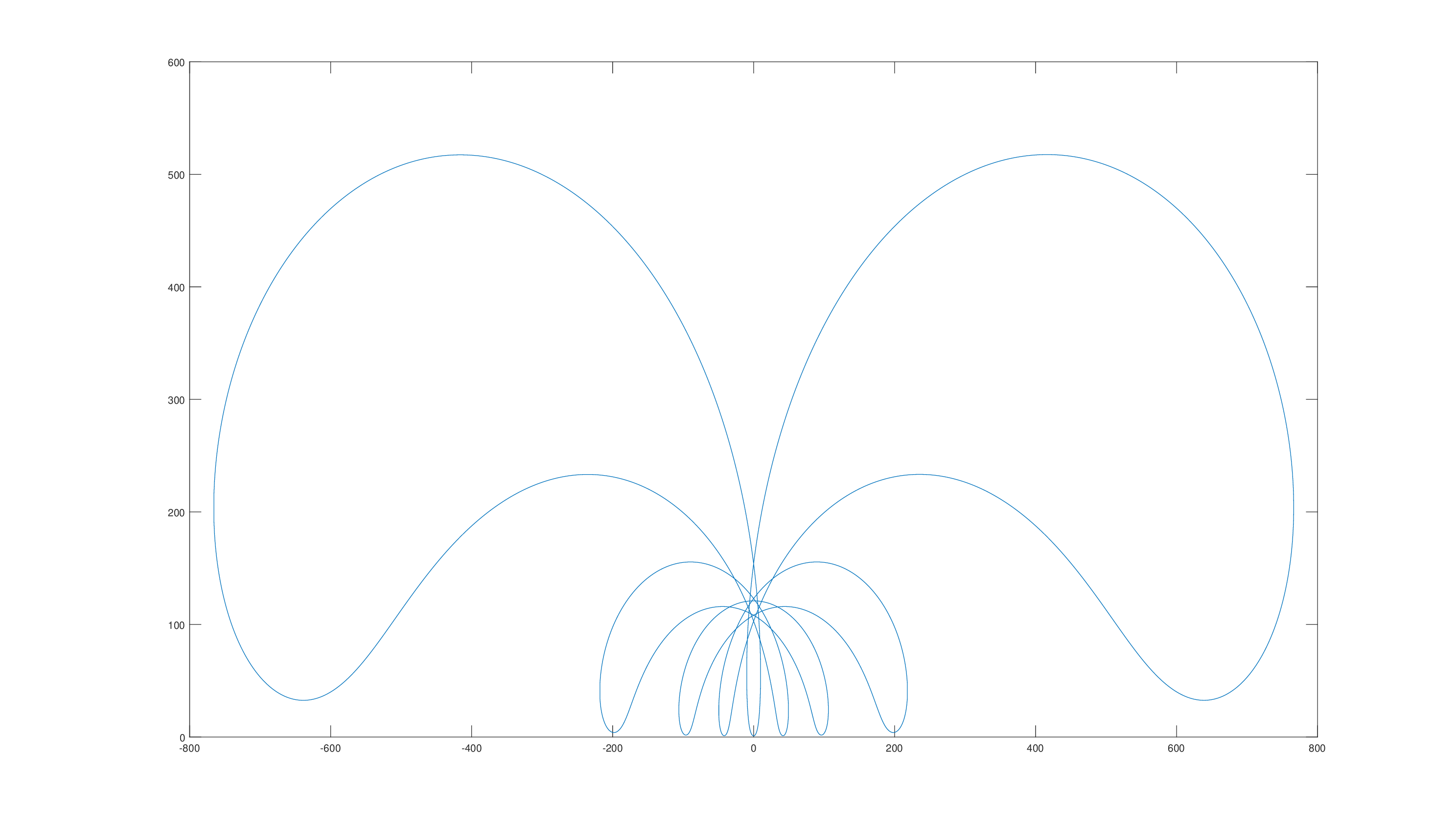}
        \caption{$\lambda = 0 , C= -0.01086, n = 9$}
    \end{subfigure}  
    \caption{Closed Free Orbitlike Elastica}\label{fig:OrbitlikeFreeElastica}
\end{figure}
% Ende von Section 4. Figure 1: "Closed Free Orbitlike Elastica"
% Figure 1a) "lambda = 0, C= -0.394371,n = 3".
% Figure 1b) "lambda =0 , C= -0.0635, n = 5".
% Figure 1c) "lambda = 0, C= -0.0225, n=7".
% Figure 1d) "lambda = 0 , C= -0.01086, n = 9".
% 

 \begin{proof}
 If $\gamma$ is not rotational or $m = 0$, then either Proposition \ref{prop:orbclose} or Proposition \ref{prop:closedrot} imply that $n=1$. Therefore $\gamma$ can not be simple since otherwise Proposition \ref{prop:periodorb} would be violated. Notice that $\mathcal{E}(\gamma) \geq 16 $ according to Proposition \ref{prop:liyau}. Then, by Proposition \ref{prop:integr} and \eqref{eq:intdnsquared} we find
 \begin{align*}
 16 \leq \mathcal{E}(\gamma) &= \int_0^{\frac{2K(p)}{r}} \kappa_0^2 \dn^2(rs, p) \diff s = \frac{\kappa_0^2}{r} \int_0^{2 K(p)} \dn^2(s,p) \diff s  \\ 
 & = 4 \sqrt{2 \lambda + 4} \frac{E(p)}{\sqrt{2-p^2}}\leq 4 \sqrt{2\lambda + 4} \frac{\pi}{2\sqrt{2}} 
 \end{align*}
 where we used Proposition \ref{prop:ellipt} in the last step. Solving the inequality for $\lambda$ we obtain that $\lambda \geq \frac{64}{\pi^2}- 2 $.
 \end{proof}

 \begin{remark}\label{rem:strres}
A close examination of the proof of Proposition \ref{prop:simcloorbit} reveals that  we have actually shown a stronger result: Orbitlike elastica with $n =1$ exist only for $\lambda \geq \frac{64}{\pi^2} - 2 $.
 \end{remark}

 \section{The Ratio of Energy and Length}
 \label{sec:Ratio}
 In this section we show the announced Reilly type inequality \eqref{eq:2.1}. This inequality will account for the fact that for each evolution $(f_t)_{t \geq 0 }$ by elastic flow with $\mathcal{E}(f_0) < 16$, the hyperbolic length $(\mathcal{L}(f_t))_{t \geq 0 }$ is uniformly bounded in time. This already implies that $(f_t)_{t \geq 0 }$ is a convergent evolution as we shall see in Section \ref{sec:MainResultProofs}.

 \begin{lemma}[Energy of Simple Orbitlike Elastica]\label{energ:orbitl}
Let $\gamma$ be a noncircular simple closed orbitlike elastica. Then 
\begin{equation}
\mathcal{E}(\gamma) \geq 8 \sqrt{2 \lambda + 4  }.
\end{equation}
\end{lemma}

\begin{proof}
Note that $n \geq 2$ by Proposition \ref{prop:periodorb}. Furthermore, using Proposition \ref{prop:integr} and \eqref{eq:intdnsquared},
\begin{align*}
\int_\gamma \kappa^2 \diff s  %& = \int_0^L \kappa_0^2 \dn^2 (rs, p) \diff s  \\ 
 &  =  \kappa_0^2 \int_0^{\frac{2n K(p)}{r}} \dn^2(rs,p) \diff s   = \frac{\kappa_0^2 n}{r} \int_0^{2 K(p)} \dn^2(z,p) \diff z \\ & = 4 \kappa_0 n E(p) = 4 \sqrt{2\lambda + 4}n \frac{E(p)}{\sqrt{2-p^2}}  \geq 8\sqrt{2\lambda + 4} \frac{E(p)}{\sqrt{2-p^2}} \\ &\geq 8 \sqrt{2\lambda + 4},
\end{align*}
where we used the inequality $\frac{E(p)}{\sqrt{2-p^2}} \geq 1 $ from Proposition \ref{prop:ellipt}. 
\end{proof}
{With all the preparations, we are finally ready to prove the first part of Theorem \ref{thm:main1}. As a first step, we examine the Reilly quotient for elastica in the subsequent lemma and then show in Theorem \ref{thm:reilly} that the infimum of the Reilly quotients coincides with the infimum of the Reilly quotients for elastica.}
For a similar inequality for non-closed curves in hyperbolic space c.f. \cite[Section 5]{Eichmann}.

\begin{lemma}[A Reilly-Type Inequality for Small-Energy-Elastica]\label{e/lelastica}Let $\varepsilon >0 $ be arbitrary. Then there is $c = c(\varepsilon) > 0 $ such that
\begin{equation*}
\frac{\mathcal{E}(\gamma)}{\mathcal{L}(\gamma)} \geq c(\varepsilon)\; \;  \textrm{for each closed elastica $\gamma$ that satisfies} \;  \mathcal{E}(\gamma) \leq 16 - \varepsilon. 
 \end{equation*}
 \end{lemma}

\begin{proof}
Fix $\varepsilon> 0$. Note that according to Proposition \ref{prop:waveclose} any wavelike elastica is nonsimple and therefore $\mathcal{E}(\gamma) \leq 16 - \varepsilon <16$ can only hold for simple orbitlike or for circular elastica, see Proposition \ref{prop:liyau}. In the case of a circular elastica, i.e. $\kappa \equiv const$ we obtain that $\gamma$ is a circle in $\mathbb{H}^2$ since according to Proposition \ref{prop:curvconst} this is the only closed curve with constant curvature. However then $|\kappa[\gamma]| >1 $ and this %means 
implies
\begin{equation*}
\mathcal{E}(\gamma) = \int_\gamma |\kappa|^2 \diff s  \geq \int_\gamma 1\;  \diff s  = \mathcal{L}(\gamma) .
\end{equation*} 
In the case of a simple orbitlike elastica with $n \geq 2 $ (see Proposition \ref{prop:periodorb}) we obtain using Proposition \ref{prop:integr} and \eqref{eq:intdnsquared}
\begin{align}\label{eq:elquot}
\frac{\mathcal{E(\gamma)}}{\mathcal{L}(\gamma) } & = \frac{1}{\frac{2nK(p)}{r}} \int_0^\frac{2nK(p)}{r} \kappa_0^2 \dn^2(rs, p ) \diff s  \nonumber \\ & = \kappa_0^2 \frac{E(p) }{K(p) } = (2 \lambda + 4 )\frac{E(p)}{(2-p^2) K(p)} \geq \frac{1}{K(p)},
\end{align}
where we used in the last step that $\lambda \geq - 1 $. This is true since according to Proposition \ref{prop:orbclose} $\gamma$ has to be rotational which 
allows us to apply the estimate in Proposition \ref{prop:rotkil}.
%applicable.

Now assume that the statement is false. Then there exists a sequence $(\lambda_l,C_l)_{l \in \mathbb{N}} $ such that for each $l \in \mathbb{N} $ there is an orbitlike elastica $\gamma_l$ with parameters $\lambda_l, C_l$ satisfying  $\mathcal{E}(\gamma_l) \leq 16 - \varepsilon$ and $ \frac{\mathcal{E}(\gamma_l)}{\mathcal{L}(\gamma_l)} \leq \frac{1}{l}$. According to \eqref{eq:elquot}, $\frac{1}{K(p_l)} \rightarrow 0 $ and hence because of Proposition \ref{prop:identities} (5), $p_l \rightarrow 1 $ as $l \rightarrow \infty$. Recall from Proposition \ref{prop:integr} that
\begin{equation}\label{eq:newmod}
p_l^2 = \frac{2 \sqrt{(2+\lambda_l)^2 + 4 C_l}}{2 + \lambda_l + \sqrt{(2 + \lambda_l)^2 + 4C_l  }}.
\end{equation}
Lemma \ref{energ:orbitl} implies that 
%\begin{equation}
$16- \varepsilon \geq 8 \sqrt{2\lambda_l + 4 }$,
%\end{equation}
and hence $(\lambda_l)_{l \in \mathbb{N}}$ defines a bounded sequence. Since 
\begin{equation*}
4 |C_l| = \kappa_{0,l}^4 - (2 \lambda_l +4 ) \kappa_{0,l}^2 = \frac{(2\lambda_l+ 4)^2}{(2-p_l^2)^2} - \frac{(2\lambda_l + 4)^2}{(2- p_l^2)} \leq (2\lambda_l+4)^2 ,
\end{equation*}
the sequence $(C_l)_{l \in \mathbb{N}}$ is bounded as well. Hence, there is a subsequence $(k_l)\subset \mathbb{N}$ and $(\lambda, C)$  such that $\lambda_{k_l} \rightarrow \lambda$ and $C_{k_l} \rightarrow C$. Passing to the limit in \eqref{eq:newmod} we obtain 
%\begin{equation}
%1  =  \frac{2 \sqrt{(2+ \lambda)^2 + 4C}}{2 + %\lambda + \sqrt{(2+ \lambda)^2 + 4C}}.
%\end{equation}
%Therefore 
%\begin{equation}
$2 + \lambda + \sqrt{(2+ \lambda)^2 + 4C} = 2 \sqrt{(2+ \lambda)^2 + 4C}$,
%\end{equation}
which implies that $C = 0$. As we discussed above, $\gamma_l$ is rotational and hence $\lambda_l^2 + 4C_l < 0 $. Consequently,
%\begin{equation}
$0 \geq \lim_{l \rightarrow \infty} ( \lambda_{k_l} ^2 + 4C_{k_l} ) = \lambda ^2 $
%\end{equation}
and thus $\lambda = 0 $. But then Lemma \ref{energ:orbitl} yields the contradiction
\begin{equation*}
16 - \varepsilon \geq \mathcal{E}(\gamma_{k_l} ) \geq 8 \sqrt{2 \lambda_{k_l} + 4} \rightarrow 16 \quad (l \rightarrow \infty).\qedhere 
\end{equation*}
%a contradiction. 
\end{proof}

The following theorem is a precise formulation of Theorem \ref{thm:main1} in the space $W^{2,2}(\mathbb{S}^1, \mathbb{H}^2)$.

\begin{theorem}[A Reilly-Type Inequality] \label{thm:reilly}For each $\varepsilon > 0 $ there exists $c= c(\varepsilon) > 0$ such that 
\begin{equation*}
\frac{\mathcal{E}(\gamma)}{\mathcal{L}(\gamma)} \geq c(\varepsilon) \qquad \forall \gamma \in W^{2,2} (\mathbb{S}^1, \mathbb{H}^2) \; \textrm{ immersed s.t.\; }  \mathcal{E}(\gamma) \leq 16-\varepsilon. 
\end{equation*}
Furthermore, $c(\varepsilon)$ can be chosen as %to be the same constant as 
in Lemma \ref{e/lelastica}.
\end{theorem}

\begin{proof}
Let $\varepsilon > 0 $ be arbitrary and fix $\overline{\gamma} \in W^{2,2}(\mathbb{S}^1, \mathbb{H}^2) $ immersed such that $\mathcal{E}(\overline{\gamma}) \leq 16 - \varepsilon$. Then define 
\begin{equation*}
\mathcal{A}(\overline{\gamma}) := \{ \gamma \in  W^{2,2}_{\overline{\gamma}(0)}(\mathbb{S}^1, \mathbb{H}^2) \; \textrm{immersed} : \mathcal{L}(\gamma) = \mathcal{L}(\overline{\gamma}) \}.
\end{equation*}
For simplicity of notation we define $L := \mathcal{L}(\overline{\gamma})$. 
We claim that $\inf_{\gamma \in \mathcal{A}(\overline{\gamma}) } \frac{\mathcal{E}(\gamma)}{\mathcal{L}(\gamma)}$ is attained by an elastica. Clearly, since $\overline \gamma \in \mathcal{A}(\overline \gamma) \neq \emptyset$ we have $\inf_{\mathcal A(\overline \gamma)} \mathcal{E}(\cdot) \leq \mathcal{E}(\overline \gamma) \leq 16 -\varepsilon$. Note that since the length is kept fixed $L$ is constant and it suffices to show that $\inf_{\mathcal{A}(\overline{\gamma})} \mathcal{E}(\gamma)$ is attained by an elastica. To prove this, we proceed in three steps.

\textbf{Step 1:} Applying the direct method of the caluculus of variations we prove that the infimum is attained.
To show compactness we first observe that  for each $\gamma \in \mathcal{A}(\overline{\gamma})$ it holds that 

\begin{equation}\label{eq:dingskirchen}
\overline{\gamma}_{2}(0) e^{-L} \leq \gamma_2(t) \leq \overline{\gamma}_{2}(0)e^{L} \quad \forall  t \in  [0,1] .
\end{equation}
Indeed, fix $t \in [0,1]$ and compute 
\begin{equation*}%\label{eq:boundsecnd}
|\log(\gamma_2(t)) - \log ( \gamma_2(0)) |  \leq \int_0^t \frac{|\gamma_2'(t)|}{\gamma_2(t)} \diff t  \leq \int_0^1 \frac{\sqrt{\gamma_1'^2 + \gamma_2'^2}}{\gamma_2} \diff t  = \mathcal{L}(\gamma) = L.
\end{equation*}
The inequality \eqref{eq:dingskirchen} follows using that $\gamma_2(0) = \overline{\gamma}_{2}(0)$. 
Also observe that for each $\gamma \in \mathcal{A}(\overline{\gamma})$ we can choose $\phi_\gamma \in W^{2,2} ((0,L), (0,1))$ increasing and bijective such that $\mu := \gamma \circ \phi_\gamma \in W^{2,2}((0,L), \mathbb{H}^2) $ and  $\frac{\sqrt{\mu_1'^2 + \mu_2'^2}}{\mu_2} = 1 $ on $(0,L)$, i.e $\mu$ is the reparametrization of $\gamma$ by arclength.
Then,
\begin{equation}\label{eq:length}
\int_0^L \mu_1'^2 + \mu_2'^2 \diff s = \int_0^L \mu_2^2 \diff s\leq \overline{\gamma}_2(0)^2 e^{2L} L ,
\end{equation}
as \eqref{eq:dingskirchen} holds also for $\mu$.
 Using Example \ref{ex:curvhyp} and expanding and rearranging the squares we obtain 
\begin{align}\label{eq:energumf}
&\int_\gamma |\kappa[\gamma]|^2 \diff s  
 = \int_\mu |\kappa[\mu]|^2 \diff s \nonumber \\ 
 &\quad = \int_0^L \frac{1}{\mu_2^2} \left( \mu_1'' - \frac{2}{\mu_2} \mu_1' \mu_2' \right)^2 + \frac{1}{\mu_2^2}\left( \mu_2'' + \frac{1}{\mu_2} (\mu_1'^2- \mu_2'^2)  \right)^2 \diff s  \nonumber \\
 &\quad  =  \int_0^L \frac{\mu_1''^2 + \mu_2''^2}{\mu_2^2}  - \frac{4}{\mu_2^3}\mu_1'' \mu_1' \mu_2' + \frac{4}{\mu_2^4} \mu_1'^2 \mu_2'^2 + \frac{2 \mu_2''}{\mu_2^3} (  \mu_1'^2 - \mu_2'^2 ) + \frac{(\mu_1'^2 - \mu_2'^2)^2}{\mu_2^4} \diff s  \nonumber  \\ 
 &\quad  =  \int_0^L \frac{\mu_1''^2 + \mu_2''^2}{\mu_2^2}  + \frac{(\mu_1'^2 + \mu_2'^2)^2}{\mu_2^4} - \frac{4 \mu_1'' \mu_1' \mu_2'}{\mu_2^3} + \frac{2\mu_2'' ( \mu_1'^2- \mu_2'^2) }{\mu_2^3}  \diff s \nonumber  \\
 &\quad  = \int_0^L \frac{\mu_1''^2 + \mu_2''^2}{\mu_2^2} \diff s   + L - \int_0^L \frac{4 \mu_1'' \mu_1' \mu_2'}{\mu_2^3} + \frac{2\mu_2'' ( \mu_1'^2- \mu_2'^2) }{\mu_2^3} \diff s  . \end{align}
Using the Peter-Paul inequality we find {
\begin{align*}
\int_\gamma |\kappa[\gamma]|^2 \diff s  & \geq \int_0^L \frac{\mu_1''^2 + \mu_2''^2}{\mu_2^2} \diff s   + L- \delta \int_0^L \frac{\mu_1''^2 + \mu_2''^2}{\mu_2^2}\diff s \\
&\quad - \frac{1}{\delta} \int_0^L \frac{4}{\mu_2^4} \mu_1'^2 \mu_2'^2 + \frac{(\mu_1'^2 - \mu_2'^2)^2}{ \mu_2^4}\diff s \\ 
 &
 \geq (1- \delta) \int_0^L \frac{\mu_1''^2 + \mu_2''^2}{\mu_2^2} \diff s  + \left( 1 - \frac{1}{\delta} \right) L
 \end{align*}}
for each $\delta > 0 $. Choosing $\delta = \frac{1}{2}$ we obtain 
\begin{equation}\label{eq:coerc}
\int_\gamma |\kappa[\gamma]|^2 \diff s  \geq  \frac{1}{2}\int_0^L \frac{\mu_1''^2 + \mu_2''^2}{\mu_2^2} \diff s   - L  \geq \frac{1}{2\overline{\gamma}_2(0) e^{2L}} \int_0^L (\mu_1''^2+ \mu_2''^2) \diff s  - L . 
\end{equation}
With these preliminary results we can now prove the subconvergence of the minimizing sequence. Let $(\gamma_n)_{n \in \mathbb{N}} \subset \mathcal{A}(\overline{\gamma})$ be a minimizing sequence for $\mathcal{E(\gamma)}$. Define $\mu_n = \gamma_n \circ \phi_{\gamma_n}$. Notice that $\mu_n \in W^{2,2} ( (0,L) , \mathbb{R}^2) $ for each $n \in \mathbb{N}$. Also note that $(\mu_n)$ is bounded in $W^{2,2}((0,L), \mathbb{R}^2)$ because of \eqref{eq:coerc}, \eqref{eq:length} and $\mu_n(0) = \gamma_{0,2}(0)$. Therefore $(\mu_n)_{n \in \mathbb{N}}$ possesses a weakly $W^{2,2}$ convergent subsequence, which we denote for the sake of simplicity by $(\mu_n)$ again. Let $\nu \in W^{2,2}((0,L), \mathbb{R}^2)$ denote the weak limit. Define $\gamma^*(t) := \nu(Lt)$ for $t \in [0,1]$. We show now that $\gamma^* \in \mathcal{A}(\overline{\gamma})$ and $\mathcal{E}(\gamma^*) \leq \liminf_{n \rightarrow \infty } \mathcal{E}(\gamma_n) = \inf_{\mathcal{A}(\overline{\gamma})} \mathcal{E}(\gamma)$. First note that $\mu_n$ converges to $\nu $ uniformly and together with \eqref{eq:dingskirchen} we infer that $\nu_2(t) \in [\overline{\gamma}_{2}(0) e^{-L} , \overline{\gamma}_{2}(0) e^L ] $ for each $t \in [0,L]$. Furthermore, since $\nu_n' \rightarrow \nu'$ uniformly in $[0,L]$ we find 
\begin{align*}
\mathcal{L}(\gamma^*) & = \int_0^1 \frac{\sqrt{(\gamma_1^*)'^2 + (\gamma_2^*)'^2 }}{\gamma_2^*} \diff t  = L \int_0^1 \frac{\sqrt{\nu_1'^2(Lt) + \nu_2'^2(Lt) }}{\nu_2 (Lt)  } \diff t  = \int_0^L \frac{\sqrt{\nu_1'^2 + \nu_2'^2}}{\nu_2} \diff t  \\ & = \lim_{n \rightarrow \infty} \int_0^L \frac{\sqrt{\mu_{n,1}'^2 +\mu_{n,2}'^2}}{\mu_{n,2}} \diff t   = L,
\end{align*} 
and we conclude that $\gamma^* \in \mathcal{A}(\overline \gamma)$. We proceed with the lower semicontinuity of $\mathcal{E}$. Since $\mathcal{E}(\gamma^*) = \mathcal{E}(\nu)$ % is invariant with respect to reparametrizations 
 we find (see \eqref{eq:energumf}) 
\begin{equation}\label{eq:6.35}
\int_{\gamma^*} |\kappa[\gamma^*]|^2 \diff s  
%\nonumber &  = \int_\nu |\kappa[\nu]|^2 \diff s  \\ & 
 = \int_0^L \frac{\nu_1''^2 + \nu_2''^2}{\nu_2^2} \diff t   + L - \int_0^L \frac{4 \nu_1'' \nu_1' \nu_2'}{\nu_2^3} + \frac{2\nu_2'' ( \nu_1'^2- \nu_2'^2) }{\nu_2^3} \diff t  .
 \end{equation}%
 We will show that this expression falls below $\liminf_{n \rightarrow \infty} \mathcal{E}(\gamma_n)$. 
Since $\mu_{n,1}'' \rightharpoonup \nu_1'' $ in $L^2(0,L)$ and  $\mu_{n,1}' \rightarrow \nu_1',  \mu_{n,2}' \rightarrow \nu_2' , \mu_{n,2} \rightarrow \nu_2$ in $C[0,L]$ and $\mu_{n,2}$ is uniformly bounded from below we find
\begin{align*}
 & \lim_{n \rightarrow \infty}  \left\vert \int_0^L \frac{4 \mu_{n,1}'' \mu_{n,1}' \mu_{n,2}' }{\mu_{n,2}^3} \diff t - \int_0^L \frac{4 \nu_1'' \nu_1' \nu_2'}{\nu_2^3}\diff t \right\vert \\
  &  \quad \leq \liminf_{n \rightarrow \infty} \left\vert \int_0^L  \tfrac{4 \mu_{n,1}'' \mu_{n,1}' \mu_{n,2}' }{\mu_{n,2}^3} - \frac{4 \mu_{n,1}'' \nu_{1}' \nu_{2}' }{\nu_{2}^3} \diff t \right\vert
  + \left\vert \int_0^L \tfrac{4 \mu_{n,1}'' \nu_{1}' \nu_{2}' }{\nu_{2}^3} - \int_0^L  \tfrac{4 \nu_1'' \nu_1' \nu_2'}{\nu_2^3}\diff t\right\vert \\
   & \quad \leq 4 \sup_{ [0,L] } \left\vert \tfrac{ \mu_{n,1}' \mu_{n,2}' }{\mu_{n,2}^3} - \tfrac{ \nu_{1}' \nu_{2}' }{\nu_{2}^3} \right\vert
   \int_0^L |\mu_{n,1}''|\diff t  +  \left\vert \int_0^L \tfrac{4 \mu_{n,1}'' \nu_{1}' \nu_{2}' }{\nu_{2}^3}\diff t - \int_0^L  \tfrac{4 \nu_1'' \nu_1' \nu_2'}{\nu_2^3} \diff t\right\vert \\ 
 & \quad  \leq  4 \sup_{ [0,L] } \left\vert \tfrac{ \mu_{n,1}' \mu_{n,2}' }{\mu_{n,2}^3} - \tfrac{ \nu_{1}' \nu_{2}' }{\nu_{2}^3} \right\vert  \Big(\int_0^L |\mu_{n,1}''|^2 \diff t \Big)^\frac{1}{2} L^\frac{1}{2} +  \Big\vert \int_0^L \tfrac{4 \mu_{n,1}'' \nu_{1}' \nu_{2}' }{\nu_{2}^3} -   \tfrac{4 \nu_1'' \nu_1' \nu_2'}{\nu_2^3} \diff t\Big\vert
  \rightarrow 0 
\end{align*}    
where we used in the last step that weakly convergent sequences in Hilbert spaces are bounded. The last summand in \eqref{eq:6.35} can be treated similarly. For the first summand observe that 
\begin{equation*}%\label{eq:weakc}
\int_0^L \tfrac{\nu_1''^2}{\nu_2^2}\diff t = \lim_{n \rightarrow \infty} \int_0^L \tfrac{\nu_1'' \mu_{n,1}''}{\nu_2^2} \diff t= \lim_{n \rightarrow \infty} \int_0^L \tfrac{\nu_1'' \mu_{n,1}'' }{\nu_2} \left( \tfrac{1}{\mu_{n,2}} - \tfrac{1}{\nu_2} \right)\diff t + \int_0^L\tfrac{\nu_1'' \mu_{n,1}''}{\mu_{n,2}\nu_2 }\diff t.
\end{equation*}
Note that 
\begin{align*}
& \limsup_{n \rightarrow \infty} \left\vert \int_0^L \frac{\nu_1'' \mu_{n,1}'' }{\nu_2} \left( \frac{1}{\mu_{n,2}} - \frac{1}{\nu_2} \right)\diff t \right\vert  \leq \limsup_{n \rightarrow \infty} \sup_{[0,L]} \left\vert  \frac{1}{\mu_{n,2}} - \frac{1}{\nu_2} \right\vert \int_0^L \frac{|\nu_1''| |\mu_{n,1}''|}{\nu_2 }\diff t  \\ 
& \quad \leq \limsup_{n \rightarrow \infty} \sup_{[0,L]} \left\vert  \frac{1}{\mu_{n,2}} - \frac{1}{\nu_2} \right\vert \frac{1}{\overline{\gamma}_{2}(0)}e^L \int_0^L |\nu_1''| |\mu_{n,1}''|^2 \diff t\\  & \quad  = \limsup_{n \rightarrow \infty} \sup_{[0,L]} \left\vert  \frac{1}{\mu_{n,2}} - \frac{1}{\nu_2} \right\vert \frac{1}{\overline{\gamma}_{2}(0)} e^L \left( \int_0^L |\nu_1''|^2 \diff t\right)^\frac{1}{2}  \left( \int_0^L |\mu_{n,1}''|^2 \diff t \right)^\frac{1}{2} = 0 ,
\end{align*}
where we used again that weakly convergent sequences are bounded. % Using \eqref{eq:weakc}
Together with the Cauchy-Schwarz inequality we find  
\begin{equation*}
\int_0^L \frac{\nu_1''^2}{\nu_2^2}\diff t  = \lim_{n \rightarrow \infty} \int_0^L\frac{\nu_1'' \mu_{n,1}''}{\mu_{n,2}\nu_2 \diff t} \leq \liminf_{n \rightarrow \infty} \left( \int_0^L \frac{\nu_1''^2}{\nu_2^2}\diff t \right)^\frac{1}{2} \left( \int_0^L \frac{\mu_{n,1}''^2}{\mu_{n,2}^2}\diff t \right)^\frac{1}{2}, 
\end{equation*}
and this implies 
\begin{equation*}
\int_0^L \frac{\nu_1''^2}{\nu_2^2}\diff t \leq \liminf_{n \rightarrow \infty}  \int_0^L \frac{\mu_{n,1}''^2}{\mu_{n,2}^2}\diff t.
\end{equation*}
Similarly, we can show that 
\begin{equation*}
\int_0^L \frac{\nu_2''^2}{\nu_2^2} \diff t\leq \liminf_{n \rightarrow \infty}  \int_0^L \frac{\mu_{n,2}''^2}{\mu_{n,2}^2}\diff t
\end{equation*}
and come to the conclusion
\begin{equation} 
  \int_{\gamma^*} |\kappa[\gamma^*]|^2 \diff s   \leq \liminf_{n \rightarrow \infty} \int_{\mu_n} |\kappa[\mu_n]|^2 \diff s  = \liminf_{n \rightarrow \infty} \int_{\gamma_n} |\kappa[\gamma_n]|^2  \diff s  .
\end{equation} 
 Thus $\gamma^* \in \mathcal{A}(\overline \gamma)$ with $\mathcal{E}(\gamma^*) = \inf_{\gamma \in \mathcal{A}(\overline{\gamma})} \mathcal{E}(\gamma)$, i.e. the minimum is attained {by a curve $\gamma^* \in W^{2,2}(\mathbb{S}^1,\mathbb{H}^2)$ that satisfies $\E(\gamma^*) \leq 16- \varepsilon$}. 
%We first conclude that $\mathcal{E}(\gamma^*) \leq \mathcal{E}(\overline{\gamma})< 16 - \varepsilon$.  Therefore $\gamma^* \in \mathcal{A}(\overline{\gamma})$ and this finally shows that the infimum is attained. 

\textbf{Step 2:} Any minimizer in $\mathcal{A}(\overline{\gamma})$ is a critical point of $\mathcal{E}+ \lambda \mathcal{ L }$ for some $\lambda \in \mathbb{R}$. For this we use  \cite[Proposition 43.21]{Zeidler} with $X: = W^{2,2}(\mathbb{S}^1, \mathbb{R}^2)$, $Y := \mathbb{R}$ and
\begin{equation*}
M := \{\gamma \in W^{2,2} (\mathbb{S}^1,\mathbb{R}^2 ) \; \textrm{immersed} \; | \; \mathcal{E}(\gamma) \leq 16 - \varepsilon , \frac{1}{2}e^{-L} \overline{\gamma}_{2}(0) < \gamma_2 < \frac{3}{2}e^L \overline{\gamma}_{2}(0) \}
\end{equation*}   
{as well as $\mathcal{F} (\gamma) := \frac{1}{L} \mathcal{E}(\gamma) $ and $\mathcal{G}(\gamma) = \mathcal{L}(\gamma) - L$. Note that $M \subset W^{2,2}(\mathbb{S}^1, \mathbb{H}^2)$. To infer from \cite[Proposition 43.21]{Zeidler} that any critical point  $\gamma^*$ satisfies
\begin{equation} \label{eq:criti}
\mathcal{F}'(\gamma^*) + \lambda \mathcal{G}'(\gamma^*) = 0 , 
\end{equation}
for some $\lambda \in \mathbb{R}$ one has to show that $\mathcal{G}$ is a Frech\'et differentiable submersion  {(i.e. $\mathcal{G}'(\gamma)$ is surjective for all $\gamma \in M$)} and $\mathcal{F}$ is Frech\'et differentiable {on $M$}. }%, then we can % which would prove the claim.
 We only sketch the proof of the submersion property: It is standard to show that each critical point of $\mathcal{G}$ in $M$ satisfies $\kappa \equiv 0$. However, in $M$ there {exist} no closed curves with $\kappa \equiv 0$ since geodesics in $\mathbb{H}^2$ are never closed. This implies that $G$ is a submersion on $M$.

\textbf{Step 3:} It follows for instance from \cite[Section 5]{Eichmann} (with the same function spaces used in Step 2) that all solutions of $\eqref{eq:criti}$ are smooth and their arclength reparametrizations satisfy the {(possibly constained)} elastica equation.\\
To conclude the proof we use Lemma \ref{e/lelastica} to obtain that
\begin{equation*}
\frac{\mathcal{E}(\overline{\gamma})}{\mathcal{L}(\overline{\gamma})} \geq \frac{\mathcal{E}(\gamma^*)}{\mathcal{L}(\gamma^*) } \geq \inf_{ \sigma \; \textrm{elastica}, \; \mathcal{E}(\sigma) \leq 16- \varepsilon} \frac{\mathcal{E}(\sigma)}{\mathcal{L}(\sigma)} \geq c(\varepsilon).
\end{equation*}
since $\gamma^*$ is an {elastica} satisfying $\mathcal{E}(\gamma^*)  \leq 16- \varepsilon $. Since $\overline{\gamma}$ was arbitrary, the claim follows. 
\end{proof}

\section{A Flow Invariant} \label{sec:Inv}
In this section we describe the possible limit behavior of the flow by computing the Euclidean total curvature  
\begin{equation*}
T[\gamma] := \frac{1}{2\pi} \int_\gamma \kappa_{\mathbb{R}^2} \diff s  
\end{equation*}
for closed elastic curves $\gamma$, more precisely for $\iota\circ \gamma$, where $\iota\colon \mathbb{H}^2 \rightarrow \mathbb{C}$ is the canonical embedding into the upper half plane. Notice once more that $\kappa_{\mathbb{R}^2}$ denotes the $\mathbb{R}^2$-curvature of $\gamma$ and not the hyperbolic curvature of $\gamma$ and $\diff s $ denotes (only in this section) the Euclidean arclength parameter. The explicit parametrization given in Theorem \ref{thm:explpara} allows us to look at the curves `with Euclidean eyes' and therefore to compute $T[\gamma]$. 

The total curvature is such an important quantity since -- as it will turn out -- for subconvergent evolutions by elastic flow \eqref{eq:Flow} in $\mathbb{H}^2$, the initial curve and the limit curve will have the same total curvature  (at least, provided that the initial curve is smooth, see \cite[Theorem 1.1]{DS17}). Therefore, the total curvature allows us to classify subconvergent evolutions by elastic flow and to exclude their existence for certain initial data. Indeed, we will show in this section that there cannot be a subconvergent evolution with initial data of vanishing total curvature.

\begin{definition}[Regular Homotopy] Let $\mathrm{Imm}(\mathbb{S}^1, \mathbb{R}^2)$ be the set of all immersed curves in  $C^1(\mathbb{S}^1, \mathbb{R}^2)$. Together with the relative topology of $C^1(\mathbb{S}^1, \mathbb{R}^2) $, it becomes a topological space. We say that two curves $\gamma_1, \gamma_2 \in \mathrm{Imm}(\mathbb{S}^1, \mathbb{R}^2)$ are \textit{regularly homotopic}, if they lie in the same path-component of $\mathrm{Imm}(\mathbb{S}^1, \mathbb{R}^2)$. A path in $\mathrm{Imm}(\mathbb{S}^1, \mathbb{R}^2)$ is called a \textit{regular homotopy}. 
\end{definition}

\begin{remark}
Let $\gamma_0 \in C^{\infty}(\mathbb{S}^1, \mathbb{H}^2)$ be immersed and let $(\gamma_t)_{t \geq 0 } $ be the evolution of $\gamma_0$ under the elastic flow (see Theorem \ref{thm:LTE}) in $\mathbb{H}^2$. Note that for each $t \geq 0$, the canonical Euclidean inclusions of $\gamma_0$ and $\gamma_t$ are regularly homotopic in $\mathrm{Imm}(\mathbb{S}^1, \mathbb{R}^2)$, since $\mathbb{H}^2$ is diffeomorphic to $\mathbb{R}^2$. Here we also used that the flow is sufficiently smooth, see \cite[Theorem 1.1]{DS18}.
\end{remark}

\begin{prop}[Whitney-Graustein Theorem] \label{prop:whitney}
Fix $c \in \mathrm{Imm}(\mathbb{S}^1, \mathbb{R}^2)$. Then $T[c]$ is an integer. Additionally, two curve $c_1,c_2 \in \mathrm{Imm}(\mathbb{S}^1, \mathbb{R}^2)$ are regularly homotopic if and only if $T[c_1] = T[c_2]$. Additionally, if $\gamma_l \rightarrow \gamma$ in $C^1(\mathbb{S}^1, \mathbb{R}^2)$ then there is $N \in \mathbb{N}$ such that $T[\gamma_l] = T[\gamma]$ for all $l \geq N$. 
\end{prop}

\begin{proof}
The fact that $T[c]$ is an integer and that $T$ is continuous with respect to the $C^1(\mathbb{S}^1, \mathbb{R}^2)$-topology follows for instance from \cite[Theorem 6.11]{Abbena}. The remaining direction is known as the Whitney-Graustein Theorem and proved in \cite{Whitney}.
\end{proof}

\begin{remark}
The previous proposition actually shows that $T$ defines a flow invariant for all flows that define regular homotopies in $\mathrm{Imm}(\mathbb{S}^1, \mathbb{R}^2)$. 
\end{remark}

\begin{prop}[Total Curvature of Elastica in $\mathbb{H}^2$]\label{prop:RTC}
Let $\gamma$ be the canonical embedding of a closed hyperbolic elastic curve into $\mathbb{C}$, that is parametrized by hyperbolic arclength. Then
\begin{equation*}
T[\gamma] = \frac{1}{2\pi}\mathrm{Im} \left(  \int_\gamma \frac{2az}{az^2+ c}  \diff z - \int_\theta \frac{1}{z} \diff z       \right),
\end{equation*}
where $\diff z$ denotes a complex line integral and 
%\begin{equation}
 $\theta(s) := (\kappa^2(s) - \lambda) +2 i \kappa'(s)$ for $s \in [0,L]$.
\end{prop}

\begin{proof}
Recall that for a smooth plane curve $c \colon (0,T) \rightarrow \mathbb{C} $ the normal is given by $N = \frac{i c'}{|c'|}$ and therefore

$\kappa_{\mathbb{R}^2}[c] =  \frac{\langle c'', N \rangle}{|c'|^2}  = \mathrm{Re} \left(  i \frac{\overline {c''}  c'}{|c'|^3 } \right)$.
%\end{equation}
 Recall from Theorem \ref{thm:explpara}  that 
%\begin{equation*}
$\gamma' %= \frac{a\gamma^2 + c}{(\kappa^2 - \lambda) + 2i \kappa'}
= \frac{a\gamma^2 + c}{\theta}$, thus 
\begin{align*}
\gamma'' & = \frac{2 a \gamma \gamma' }{\theta} - \frac{\theta'(a\gamma^2 + c)}{\theta^2} =  \frac{2 a \gamma \gamma'}{\theta } - \gamma' \frac{\theta'}{\theta}.
\end{align*}
Plugging into the formula for $\kappa_{\mathbb{R}^2}$ we find that on $[0,L]$ 
\begin{equation*}
\kappa_{\mathbb{R}^2}[\gamma] = \mathrm{Re} \left( \frac{i}{|\gamma'|} \left( \overline{  \frac{2a \gamma }{\theta} - \frac{\theta'}{\theta}} \right) \right)  = \frac{1}{|\gamma'|} \mathrm{Im} \left( \frac{2a \gamma}{\theta} - \frac{\theta'}{\theta} \right) .
\end{equation*}
We obtain
\begin{align*}
\int_\gamma \kappa_{\mathbb{R}^2} \diff s   & = \int_0^L \frac{1}{|\gamma'|} \mathrm{Im} \left( \frac{2a \gamma}{\theta} - \frac{\theta'}{\theta} \right) |\gamma'| \diff t  %\\ & 
= \mathrm{Im} \left(  \int_0^L \frac{2 a \gamma}{\theta} \diff t  - \int_0^L \frac{\theta'}{\theta} \diff t  \right).
%\\ & = \mathrm{Im} \left(  \int_0^L 2a \sqrt{\frac{|c|}{|a|}} \frac{ f( \mu(t) + z_0 )}{\theta(t) } \diff t  - \int_\theta \frac{1}{z} \diff z \right) 
%\\ & = \mathrm{Im} \left( \pm \int_0^L \sqrt{|ac|} f( \mu(t) + z_0) \diff t  - \int_\theta \frac{1}{z} \diff z \right) \\ & = \mathrm{Im} \left( \pm \int_0^L \mu'(t) f(\mu(t) + z_0) \diff t - \int_\theta \frac{1}{z} \diff z \right).  
\end{align*}
The differential equation in Theorem \ref{thm:explpara} reads $\theta(s) \gamma'(s) = a \gamma(s)^2 + c$ and Proposition \ref{prop:nonvan} implies that $a \gamma(s)^2 + c \neq 0 $ for all $s$. Therefore 
\begin{equation*}
\int_0^L \frac{2 a \gamma}{\theta} \diff t  = \int_0^L \frac{2 a \gamma(t)\gamma'(t) }{a \gamma(t)^2 + c }  \diff t  = \int_\gamma \frac{2a z}{az^2 + c} \diff z . \qedhere
\end{equation*}
\end{proof}

\begin{cor}[Total Curvature for Wavelike Elastica] Let $\gamma$ be a closed wavelike elastica. Then $T[\gamma] = 0 $. 
\end{cor}
\begin{proof}
Using the notation from Proposition \ref{prop:RTC} we first show that 
$\int_\theta \frac{1}{z} \diff z = 0$.
%\end{equation}
Recall that $z \mapsto \frac{1}{z}$ has a complex antiderivative on the simply-connected domain $\mathbb{C}\setminus \mathbb{R}_{\leq 0 } $. We shall show that $\theta([0,L]) \subset \mathbb{C}\setminus \mathbb{R}_{\leq 0 }$. Indeed, if $\mathrm{Im}( \theta(s)) = 0$ then $\kappa'(s) = 0 $. Using that $\kappa(s) = \kappa_0 \cn(rs,p) $ for some $r,p$ this happens only {if} $\kappa(s) = \pm \kappa_0$.  
Furthermore, we find using Proposition \ref{prop:integr} that
$\mathrm{Re} (\theta(s)) =  \kappa(s)^2 - \lambda = \kappa_0^2 - \lambda \geq 2 >0$, 
%\end{equation}
which implies that $\theta(s) \not \in \mathbb{C}\setminus \mathbb{R}_{\leq 0}$.  
It remains to show that 
%\begin{equation*}
$\int_\gamma \frac{2az}{az^2 + c} \diff z = 0$,  
%\end{equation*} 
but this is clear since $\gamma$ lies entirely in the upper half plane and the roots of the integrand are both on the real axis, remember $ac<0$ since $C> 0$ and $ac = - \frac{1}{4}(\lambda^2 + 4C)$, see Proposition \ref{prop:ordred} and Proposition \ref{prop:integr}. The claim follows using Cauchy's Integral Theorem.
\end{proof}

\begin{cor}[Total Curvature for Orbitlike Elastica]\label{cor:totrwealk}
Let $\gamma$ be an orbitlike rotational closed elastica. Let $m \in \mathbb{Z}$ be the integer in Proposition \ref{prop:closedrot}. Then
\begin{equation*}
T[\gamma] = \begin{cases} m & \textrm{if } \kappa_0^2 < 4 + \lambda \\ m \pm n  & \textrm{if } \kappa_0^2 > 4+ \lambda, \end{cases} 
\end{equation*} 
and if $\kappa_0^2 = 4 + \lambda $, then there exists no closed orbitlike elastica. 
\end{cor}

\begin{figure}
    \centering
    \captionsetup[subfigure]{justification=centering}
    \begin{subfigure}[b]{0.48\textwidth}
        \includegraphics[width=\textwidth]{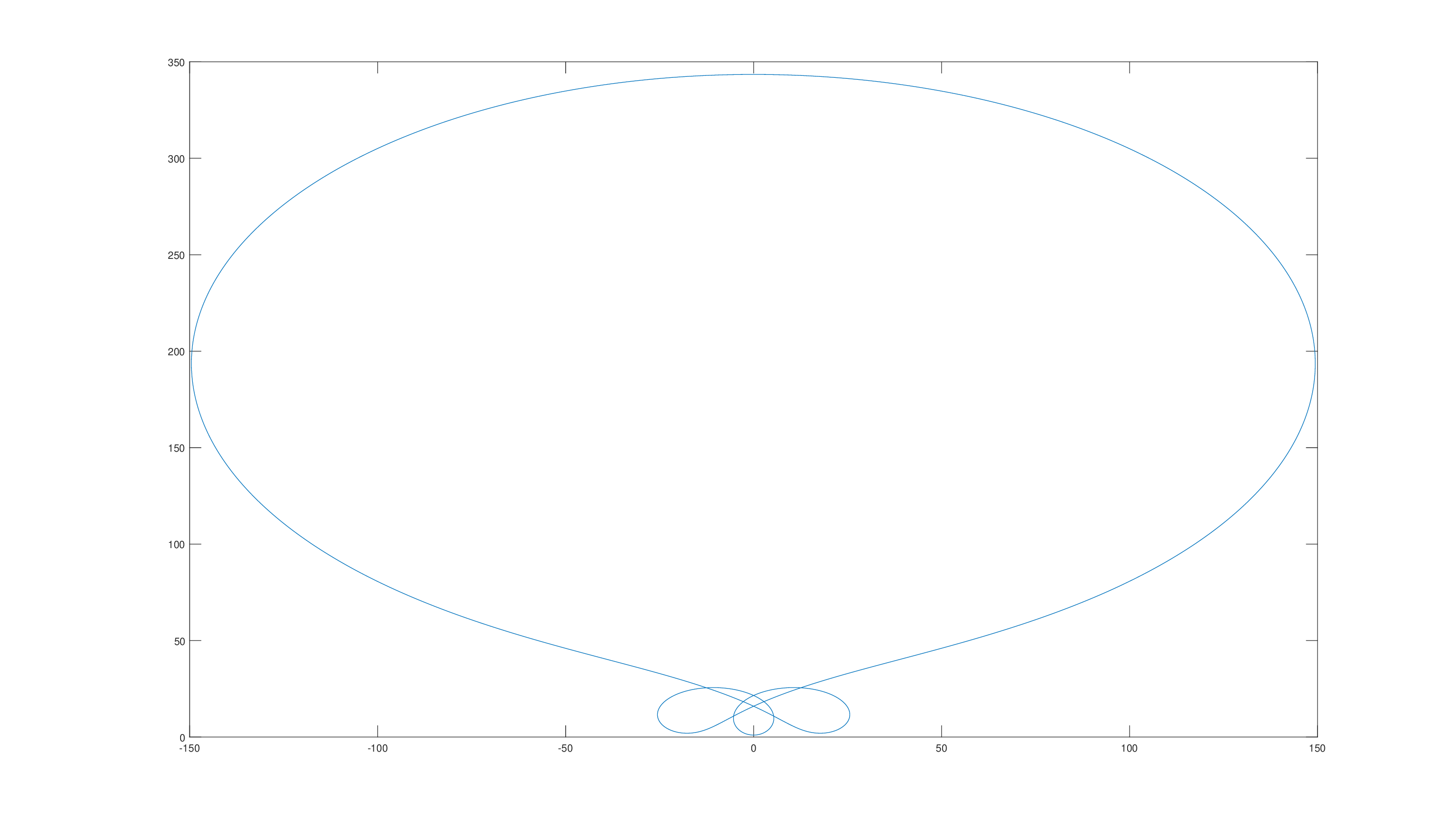}
        \caption{$T =m = 4$; $\lambda=0.39,$\newline$ C\approx - 0.54$}%        \label{fig:gull}
    \end{subfigure}~
    \begin{subfigure}[b]{0.45\textwidth}  \includegraphics[trim={630 45 580 600},clip,width=\textwidth]{3a}\caption{$T =m = 4$; $\lambda=0.39,$\newline$ C\approx - 0.54$ (detail)}%
    \end{subfigure}
    
    \begin{subfigure}[b]{0.45\textwidth}
        \includegraphics[width=\textwidth]{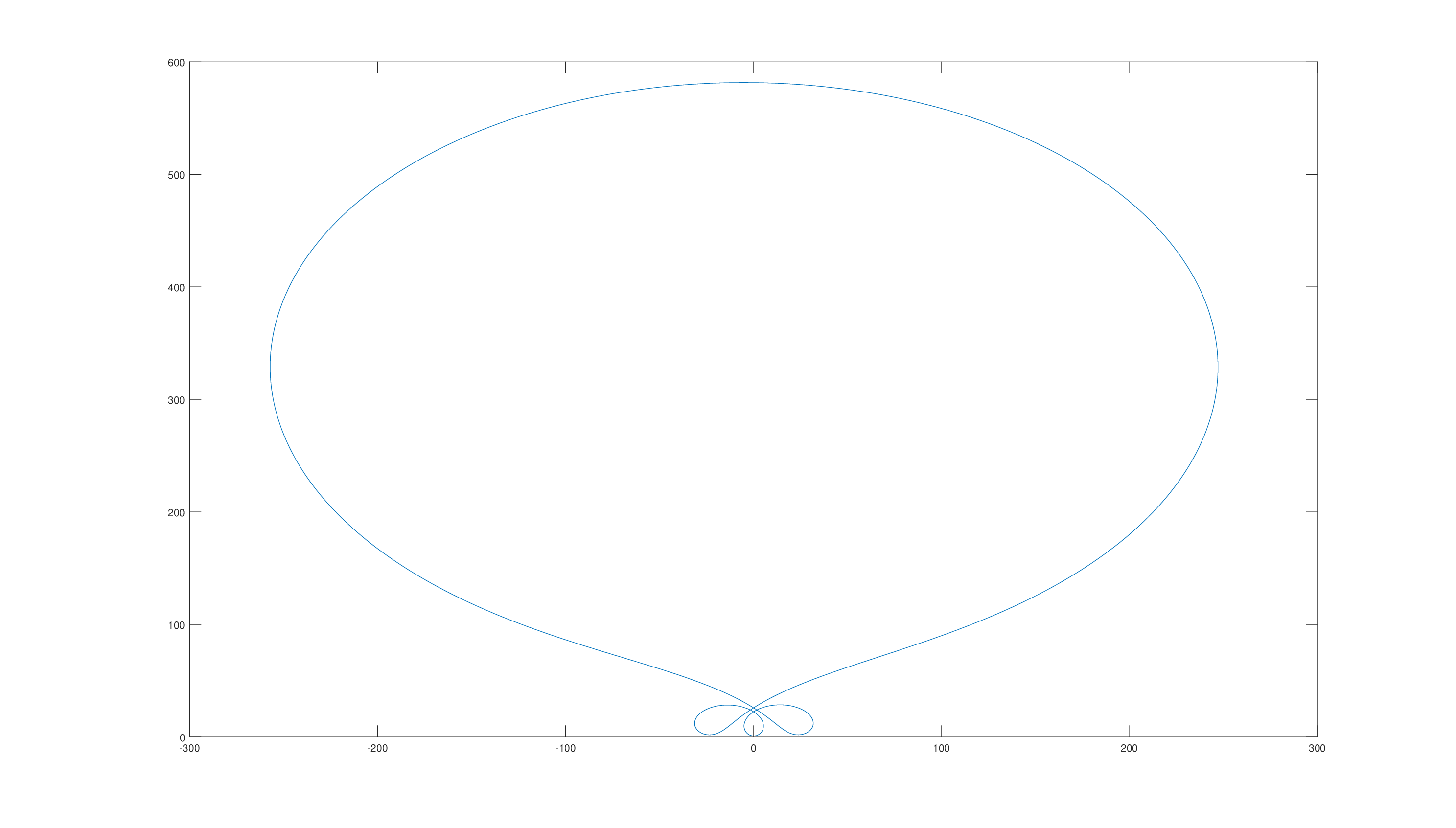}
       \caption{$T = m + n = 1 + 3 = 4$; \newline $\lambda=0.3 , C \approx -0.38$}
    \end{subfigure}~
 \begin{subfigure}[b]{0.45\textwidth}\includegraphics[trim={680 70 630 680},clip,width=\textwidth]{3b}\caption{$T = m + n = 1 + 3 = 4$; \newline  $\lambda=0.3 , C \approx -0.38$ (detail)} \end{subfigure}
    \caption{Examples for the Formula of the Total Curvature from Corollary \ref{cor:totrwealk} }%
\end{figure}

\begin{proof}
We show first that 
\begin{equation*}
\int_\theta \frac{1}{z} \diff z =  \begin{cases} 0 & \text{if }\kappa_0^2 < 4 + \lambda \\ \pm n & \text{if }\kappa_0^2 > 4 + \lambda. \end{cases}
\end{equation*}
Recall that a parametrization of $\theta$ is given by $\theta(s) = \kappa^2(s) - \lambda + 2i \kappa'$. We compute using the elastica equation \eqref{eq:elasticaeq} 
\begin{align*}
\theta(s) & = \kappa^2(s) - \lambda + 2i \kappa'(s)  = \frac{\kappa^3(s)- \lambda \kappa(s)}{\kappa(s)} + 2i \kappa'(s) \\
& = \frac{2 \kappa(s)- 2\kappa''(s)}{\kappa(s)}  + 2 i \kappa'(s) = 2 \left( 1 - \frac{\kappa''(s)}{\kappa(s)} \right) + 2 i \kappa'(s) .
\end{align*}
Now $\kappa(s) = \pm \kappa_0 \dn(rs,p)$, where the choice of sign has to be consistent again because of smoothness. We only treat the case `$+$' here but the other case can be shown similarly. The first and second derivatives can be simplified as follows using $\kappa_0 = 2 r$ according to the second case in Proposition \ref{prop:integr}, and Proposition \ref{prop:identities}:
\begin{align*}
\kappa'(s) & = - \kappa_0 r p^2 \sn(rs,p) \cn(rs,p) = - 2r^2p^2 \cos( \am(rs,p)) \sin(\am(rs,p) ) \\ &  = -  r^2 p^2 \sin(2 \am(rs,p) ) ,\\
\kappa''(s)&  = - \kappa_0 r^2 p^2 ( \cn^2(rs,p) \dn(rs,p) - \sn^2(rs,p) \dn(rs,p) ) \\ &  = - \kappa_0 r^2 p^2 \dn(rs,p) \cos(2 \am(rs,p)) = - \kappa(s) r^2 p^2 \cos(2 \am(rs,p) ) . 
\end{align*}
All in all %\begin{equation}
$\theta(s) = 2  \left( 1 + r^2 p^2 \cos(2 \am(rs,p) ) \right) - 2 i r^2 p^2 \sin( 2 \am (rs, p) )$    
%\end{equation}
for each $s \in [0,L] = \left[ 0 , \frac{2n K(p)}{r} \right]$. Since  $ s \mapsto 2 \am(rs, p) $ is strictly monotone, and $2n \pi = 2 \am (2n K(p), p) $ we can instead integrate over the following reparametrization:
\begin{equation*}
\widetilde{\theta}(\ell) := 2 ( 1+ r^2 p^2 \cos(\ell) ) -2i r^2 p^2 \sin(\ell) \quad  \ell \in [0, 2n \pi].
\end{equation*}
It becomes obvious that $\theta$ is an $n$-fold cover of $\partial B_{2r^2p^2}(2)$. Therefore 
\begin{equation*}
\int_\theta \frac{1}{z} \diff z = \begin{cases} 0  & 0 \not \in B_{2r^2p^2}(2) \\ \pm n & 0 \in B_{2r^2p^2}(2)    \end{cases}  .
\end{equation*} 
We write $\pm$ since it %does not matter
is not important for our result in which direction the circle is parametrized. Indeed, if we had treated the `$-$' case in detail, the circle would be parametrized in the opposite direction. 
 Lemma \ref{lem:8.13} shows that $r^2p^2 < 1$ if and only if $\kappa_0^2 < 4+ \lambda$. Also, Remark \ref{rem:paramrestr} shows that $r^2p^2 = 1 $ and $\kappa_0^2 = 4 +\lambda$ can never occur, so the classification is indeed complete.\\
For the rest note that $z \mapsto \frac{2az}{az^2 + c}$ is a logarithmic derivative and therefore all the residues coincide with the orders of the roots of $z \mapsto az^2 + c$. However, since $ac > 0 $, all poles have order 1. Therefore 
\begin{equation*}
\int_\gamma \frac{2az}{az^2 + c} \diff z = 2 \pi i \left( \omega\left(\gamma, \sqrt{\frac{-c}{a}} \right)  + \omega\left(\gamma , - \sqrt{\frac{-c}{a}} \right) \right) 
\end{equation*}
where $\omega(\gamma, \cdot) $ denotes the winding number of $\gamma$ and $\sqrt{\cdot}$ denotes one branch of the complex square root. Note that exactly one of $\sqrt{\frac{-c}{a}}$ and $- \sqrt{\frac{-c}{a}}$ lies in $\mathbb{H}^2$. Therefore one of these winding number is zero. Let us assume that $ \omega(\gamma , - \sqrt{\frac{-c}{a}} ) = 0$. We look to determine $ \omega(\gamma , - \sqrt{\frac{c}{a}} )$.  On the one hand 
\begin{align}\label{eq:dingsk}
\int_\gamma \frac{1}{az^2 + c} \diff z & = \int_0^L \frac{\gamma'(s)}{a \gamma^2(s) + c} = \int_0^L \frac{1}{\theta(s) } \diff s  = \frac{\pi m}{\sqrt{ac}}
\end{align}
where we used \eqref{eq:simi}. On the other hand 
\begin{equation}\label{eq:dingski}
\int_\gamma \frac{1}{az^2 + c} \diff z = 2 \pi i\cdot \mathrm{Res}\Big( \frac{1}{az^2 + c} , \sqrt{\frac{-c}{a}} \Big) \omega\Big( \gamma, \sqrt{\frac{-c}{a}}\Big)  = \frac{\pi}{\sqrt{ac}}\omega\Big( \gamma, \sqrt{\frac{-c}{a}}\Big)
\end{equation}
where we used that the residue is $\frac{1}{2i\sqrt{ac}}$. If follows from the last two equations that $ \omega(\gamma , - \sqrt{\frac{c}{a}} ) = m $. The case of $\omega( \gamma, \sqrt{\frac{-c}{a}} ) = 0 $ can be checked similarly.   
\end{proof}

\begin{cor}\label{cor:tutfrer}
There is no closed free elastic curve such that $T[\gamma] = 0$.  Moreover, for each $\lambda < \frac{64}{\pi^2}- 2$, each $\lambda$-constrained elastic curve that satisfies $T[\gamma] = 0 $ is wavelike.
\end{cor}

\begin{proof}
Showing the second part of the statement implies the first part using Corollary \ref{cor:noclwv}. Closed curves of constant curvature do certainly not satisfy $T[\gamma] = 0 $ since they are (possibly multi-fold) circles. Assume that there is a closed free orbitlike elastica such that $T[\gamma]= 0$. Let $\kappa_0$ and $\lambda$ be the parameters for this elastica. Note that Proposition \ref{prop:simcloorbit} implies that $\gamma$ is rotational. If $\kappa_0^2 < 4 + \lambda$ then $T[\gamma] = m $, where $m$ is given in Proposition \ref{prop:closedrot}. However Proposition \ref{prop:simcloorbit} implies that $m \neq 0 $ if $\lambda < \frac{64}{\pi^2} - 2$
, a contradiction. If $\kappa_0^2 < \lambda + 4$ then $T[\gamma] = m \pm n $. Unless $m = n = 1$ or $m = -1$ and $n = 1$ this cannot equal zero since $m,n$ would be relatively prime according to Proposition \ref{prop:closedrot}. However $n = 1$ is not possible for the considered values of $\lambda$, see Remark \ref{rem:strres}. 
\end{proof}

The following Corollary gives a sufficient condition for the initial value  ensuring the non-convergence of the elastic flow. A natural question is then to find the minimal energy level on which such phenomena occur. In Corollary \ref{cor:8.4} we present smooth curves $\gamma_\varepsilon$ with energy below $16+\varepsilon$ satisfying $T[\gamma_\varepsilon] = 0$, $\varepsilon > 0$.

\begin{cor}[A Class of Bad Initial Data]\label{rem:unblength}
 Let $\gamma_0$ be a smoothly closed curve such that $T[\gamma_0] = 0 $. Let $(\gamma_t)_{ t \geq 0 } $ be the time evolution of the elastic flow with initial value $\gamma_0$. Then $(\mathcal{L}(\gamma_t))_{ t \geq 0 }$ is unbounded. In particular $(\gamma_t)_{t \geq 0 }$ is a nonconvergent evolution. 
 \end{cor}

 \begin{proof}
 Assume on the contrary that $\mathcal{L}(\gamma_t)$ is bounded. Then there is a free elastic curve $\gamma_\infty$ and $t_n \rightarrow \infty$ such that the constant-hyperbolic-speed reparametrizations of $(a_n (\gamma_{t_n}- (p_n,0) ) )_{n \in \mathbb{N}}$ converge to $\gamma_\infty$ in $W^{m,2} ( \mathbb{S}^1, \mathbb{R}^2) $ for each $m \in \mathbb{N}$ and appropriately chosen $a_n, p_n$, see Theorem \ref{thm:LTE}.
% Then since $\mathcal{L}(a_n ( \gamma_{t_n} - (p_n ,0) ) ) = \mathcal{L}(\gamma_{t_n} )  $ is bounded, arguments similar to \eqref{eq:boundsecnd} imply, that also $a_n ( \gamma_{t_n} - (p_n,0) ) $ converges to $\gamma_\infty$ also in $W^{m,2} (\mathbb{S}^1, \mathbb{R}^2 ) $ for each $m \in \mathbb{N}$.
 Therefore Proposition \ref{prop:whitney} yields that  
\begin{equation}\label{eq:totalCurvLimit}
T[\gamma_\infty] = \lim_{n \rightarrow \infty} T[a_n (\gamma_{t_n} - (p_n, 0) ) ]  = \lim_{n \rightarrow \infty} T[\gamma_{t_n} ] = 0 .
\end{equation}
The existence of such $\gamma_\infty$ however would contradict Corollary \ref{cor:tutfrer}.
 \end{proof}

\section{Optimality Discussion}\label{sec:Opt}

\subsection{Optimality of the Energy Bound}
So far, we have shown that the length along the elastic flow remains bounded, provided that the initial datum $\gamma_0\in C^\infty(\mathbb{S}^1, \mathbb{H}^2)$ has small elastic energy, more precisely $\mathcal{E}(\gamma_0) < 16$, see Theorem \ref{thm:reilly}. Additionally we have constructed a class of initial data for which the length along the flow is unbounded, namely the class of curves of vanishing Euclidean total curvature. To investigate optimality of the bound of $16$, we look for curves of small energy with vanishing total curvature. 

\begin{definition}
For each $\lambda > 0 $ we call a curve $\gamma$ a \emph{$\lambda$-figure-eight}, when $\gamma$ is a $\lambda$-constrained wavelike elastic curve of vanishing total curvature.
\end{definition}

\begin{prop}\label{prop:figeight}
For each $\lambda  \in (0,  \frac{64}{\pi^2}- 2)  $ there exists a $\lambda$-figure eight.
\end{prop}

\begin{proof}
Fix some $\lambda_0 \in (0,  \frac{64}{\pi^2}- 2)$. Take an arbitrary curve  $\sigma \in C^\infty(\mathbb{S}^1, \mathbb{H}^2) $ such that $T[\sigma] = 0 $ and consider the %penalized elastic 
flow for  $\mathcal{E}_{\lambda_0}$ with initial datum $\sigma$. Applying \cite[Theorem 1.1]{DS17} we find that the flow exists and subconverges to an elastic curve $\gamma$
that satisfies \eqref{eq:elasticaeq} with $\lambda = \lambda_0$. This elastic curve has to satisfy $T[\gamma] = 0$ (see \eqref{eq:totalCurvLimit}). We now claim that $\gamma$ cannot be circular or orbitlike, since circular and orbitlike elastic curves with $\lambda < \frac{64}{\pi^2}-2$ have nonvanishing total curvature. Indeed, for circular elastica one can easily compute the total curvature of an $k$-fold cover of a circle, which is exactly $k$, so nonzero. Now suppose $\gamma$ is an orbitlike elastica. Since $\lambda < \frac{64}{\pi^2}-2$, $\gamma$ is rotational with $m\neq 0$ by Proposition \ref{prop:simcloorbit}. Then there are two cases to distinguish: if $\kappa_0^2 < 4 + \lambda$ then Corollary \ref{cor:totrwealk} yields the contradiction $0=T[\gamma] =m$.
%HALLO MARIUS, PASST DAS SO? GEHT SCHNELLER...
%However then $m = 0$ by construction of $\gamma$. But then Proposition \ref{prop:closedrot} implies that $n =1 $, which is a contradiction to Remark \ref{rem:strres}, since $\lambda < \frac{64}{\pi^2}-2$. 
If $\kappa_0^2 > 4 + \lambda$, then according to Corollary \ref{cor:totrwealk},  $T[\gamma] = m \pm n $, which can be zero only in the case $m = n = 1$ since $m,n$ are relatively prime otherwise, see Proposition \ref{prop:closedrot}. However $n= 1$ is 
a contradiction to Remark \ref{rem:strres}. 
Hence, %From this can be inferred that 
$\gamma$ must be wavelike which completes the proof. 
\end{proof}

We now derive a modified closing condition for wavelike elastic curves that is more stable to compute for small $\lambda$. This has the advantage that the new condition eliminates parameters that can hypothetically become large for small $\lambda$ and therefore lead to numerical difficulties. 

\begin{prop}[A modified closing condition]
Let $\gamma$ be a wavelike elastic curve. If $\gamma$ is closed then 
\begin{equation}\label{eq:8.1}
0 = \int_0^{2\pi} \frac{\cos^2(\theta) - \frac{\lambda}{\kappa_0^2}}{\left( 1 - \frac{4 \kappa_0^2}{(\kappa_0^2 - \lambda)^2} \sin^2(\theta) \right) \sqrt{1- p^2 \sin^2(\theta)}} d \theta. 
\end{equation}
\end{prop}  

\begin{proof}
If $\gamma$ is closed, we find using Proposition \ref{prop:waveclose} that 
\begin{align*}
0 & = \int_0^\frac{4 K(p)}{r}  \frac{\kappa^2 - \lambda }{ \lambda^2 + 4C + 4 \kappa^2}\diff s 
%\\ &
=  \int_0^{\frac{4 K(p)}{r}} \frac{\kappa_0^2 \cn^2(rs, p) - \lambda}{(\kappa_0^2- \lambda)^2 - 4 \kappa_0^2 + 4 \kappa_0^2 \cn(rs, p) }\diff s  \\
%&= \frac{1}{r} \int_0^{4 K(p)} \frac{\kappa_0^2 \cn^2(s, p) - \lambda}{(\kappa_0^2- \lambda)^2 - 4 \kappa_0^2 + 4 \kappa_0^2 \cn^2(s, p) }\diff s  \\ 
%& = \frac{1}{r} \int_0^{4 K(p)} \frac{\kappa_0^2 \cn^2(s, p) - \lambda}{(\kappa_0^2- \lambda)^2 - 4 \kappa_0^2 \sn^2(s,p) }\diff s \\ 
 & = \frac{\kappa_0^2}{r ( \kappa_0^2 - \lambda)^2 } \int_0^{4K(p)} \frac{\cn^2(s,p)- \frac{\lambda}{\kappa_0^2}}{1 - \frac{4 \kappa_0^2}{(\kappa_0^2 - \lambda)^2} \sn^2(s, p) } \diff s \\ 
 & = \frac{\kappa_0^2}{r ( \kappa_0^2 - \lambda)^2 } \int_0^{2\pi} \frac{\cos^2 \theta - \frac{\lambda}{\kappa_0^2}}{\left( 1 - \frac{4 \kappa_0^2 }{(\kappa_0^2- \lambda)^2} \sin^2 \theta \right) \sqrt{1- p^2 \sin^2\theta} } \diff \theta
\end{align*}
where we used the substitution $ \theta = \am(s,p)$ or equivalently $ s = \int_0^\theta \frac{1}{\sqrt{1- p^2\sin^2\beta}} d\beta $ in the last step. Dividing by the prefactors proves the claim. 
\end{proof}

\begin{cor}[Energy of $\lambda$-Figure-Eights]\label{cor:8.4}
For each $\varepsilon > 0 $ there exists a smooth curve $\gamma_\varepsilon$ such that $16 \leq \mathcal{E}(\gamma_\varepsilon) \leq 16+ \varepsilon$ and $T[\gamma_\varepsilon] = 0 $. 
\end{cor}  

\begin{proof}
Let $(\lambda_n)_{n \in \mathbb{N}}$ be a sequence of positive numbers smaller than $1$ and converging to zero. Denote by $\gamma_n$ a $\lambda_n$-figure eight constructed in Proposition   \ref{prop:figeight} and let $C_n, p_n, r_n$ be its canonical parameters. We show that $p_n \rightarrow 1$. Indeed, assume that there is a subsequence, which we will denote again by $(p_n)$ which converges to some other $\widetilde{p} \in [ \frac{1}{\sqrt{2}},1)$. We first show that $(\kappa_0^{(n)})_{n \in \mathbb{N}}$ (which denotes the maximum curvature of $\gamma_n$) is bounded. Indeed, if there were a subsequence (again denoted by $(\kappa_0^{n})_{n \in \mathbb{N}}$) that converges to $\infty$, then  $\frac{4(\kappa_0^{(n)})^2 }{( (\kappa_0^{(n)})^2 - \lambda_n)^2} $ would converge to zero. We can plug all the asymptotics in \eqref{eq:8.1} to obtain the contradiction
\begin{align*}
0 & = \lim_{n \rightarrow \infty} \int_0^{2\pi} \frac{\cos^2(\theta) - \frac{\lambda_n}{(\kappa_0^{(n)})^2}}{\left( 1 - \frac{4 (\kappa_0^{(n)})^2}{((\kappa_0^{(n)})^2 - \lambda_n)^2} \sin^2(\theta) \right) \sqrt{1- p_n^2 \sin^2(\theta)}} d \theta \\
 &=  \int_0^{2\pi} \frac{\cos^2(\theta)}{\sqrt{1- \widetilde{p}^2 \sin^2 (\theta)}} d\theta >0, 
\end{align*}
because the denominator can be uniformly bounded and the convergence of all quantities is uniform. Therefore $\kappa_0^{(n)}$ remains bounded. In particular, since $(\kappa_0^{(n)})^2 =  p_n^2 \frac{2\lambda_n + 4}{2 p_n^2 - 1 } $, it must hold that $\widetilde{p} \neq \frac{1}{\sqrt{2}} $. We can also show by a similar contradiction argument, again using \eqref{eq:8.1}, that given $\widetilde{p} \neq 1 $, $\frac{4 (\kappa_0^{(n)})^2}{( (\kappa_0^{(n)})^2 - \lambda_n )^2}$ must tend to $1$ as $n \rightarrow \infty$. Thus %Now observe that  
\begin{align*}
0 \leq 4 C_n & \leq \lambda_n^2 + 4 C_n = \left( (\kappa_0^{(n)})^2 - \lambda_n \right)^2 - 4 ( \kappa_0^{(n)} ) ^2  \\ & = \left( (\kappa_0^{(n)})^2 - \lambda_n \right)^2 \left(1 - \frac{4 ( \kappa_0^{(n)} ) ^2}{\left( (\kappa_0^{(n)})^2 - \lambda_n \right)^2} \right) \rightarrow 0,
\end{align*}
%as a product of a bounded sequence and a sequence converging to zero. 
showing  $C_n \rightarrow 0 $.  We obtain with Proposition \ref{prop:integr}
\begin{align*}
\widetilde{p}^2 = \lim_{n \rightarrow \infty} p_n^2 =  \lim_{n \rightarrow \infty} \frac{2 + \lambda_n + \sqrt{(\lambda_n + 2)^2 + 4C_n}}{2 \sqrt{(2 + \lambda_n)^2 + 4 C_n }} = 1,
\end{align*}
a contradiction. Therefore $p_n \rightarrow 1$ as $n \rightarrow \infty$. Now observe that 
\begin{align*}
\mathcal{E}(\gamma_n) & = \int_0^{ \frac{4 K(p_n)}{r_n} } (\kappa_0^{(n)})^2 \cn^2(r_n s, p_n)  = \frac{(\kappa_0^{(n)})^2}{r_n} \int_0^{4K(p_n)} \cn^2(s, p_n)\diff s  \\
& = 2 p_n^2 \sqrt{\frac{2\lambda_n + 4}{2p_n^2- 1}} \int_0^{4 K(p_n) }  \cn^2(s,p_n)\diff s   \\ & = 2 p_n^2 \sqrt{\frac{2\lambda_n + 4}{2p_n^2- 1}} \int_0^{4 K(p_n)} \left[ \left( 1- \frac{1}{p_n^2} \right) + \frac{1}{p_n^2} \dn^2(s,p) \right] \diff s  
\\ & = 8 p_n^2 \sqrt{\frac{2 \lambda_n + 4}{2 p_n^2 - 1}} \left( \left( 1- \frac{1}{p_n^2} \right) K(p_n) + \frac{1}{p_n^2} E(p_n) \right)  \\ 
 & = 8 \sqrt{\frac{2 \lambda_n + 4}{2 p_n -1}} \left( (p_n^2- 1) K(p_n) + E(p_n) \right)  \rightarrow 16 \quad (n \rightarrow \infty) .
\end{align*}
In particular, since $\mathcal{E}(\gamma_n ) \geq 16 $ (see Proposition \ref{prop:liyau} and Proposition \ref{prop:waveclose}) we find that for each $\varepsilon > 0 $ there has to be $n \in \mathbb{N}$ such that $16 \leq \mathcal{E}(\gamma_n) \leq 16 +\varepsilon $. The claim follows.
\end{proof}
\subsection{Behavior at the Critical Energy Level}
We have discussed what happens if we start the flow with curves of energy below $16$ and we have also identified phenomena that occur for curves of energy just slightly above $16$. The only energy level that remains to be understood is the energy level of exactly $16$. Here we distinguish two cases:  If the elastic flow $(f_t)_{t \geq 0 }$ does not start at an elastic curve, the energy will instantaneously decrease from the energy level of $16$ to an energy level below, as in this case
\begin{equation}
\frac{d}{dt} \mathcal{E}(f_t) = -||\nabla_{L^2} \mathcal{E}(f_t)||_{L^2}^2 < 0 .
\end{equation}
This being so, we can bound $(\mathcal{L}(f_t))_{t \geq 0}$ by restarting the flow at a positive time where we reach an energy level below $16$. If the flow starts at an elastic curve of energy $16$, the flow will not change the curve at all. Hence $(\mathcal{L}(f_t))$ remains bounded in any case, which is - as we discussed - sufficient for the convergence. In this section we rule out the latter case by showing that there exists no closed free elastica of energy equal to $16$. We show even more: The only closed free elastica of energy less or equal to $16$ is -- up to reparametrization and isometries -- the Clifford elastica. This leaves it as the only possible limit curve for evolutions with small energy. 
 \begin{prop}\label{prop:simfreel}
Let $\gamma \in C^\infty(\mathbb{S}^1, \mathbb{H}^2)$ be a free elastica such that $\mathcal{E}(\gamma) \leq 16$. Then $\kappa[\gamma] \equiv const.$ 
 \end{prop}
 
\begin{proof}
{Let us distinguish two cases.} If $\mathcal{E}(\gamma) < 16$ then $\gamma$ has to be simple, see Proposition  \ref{prop:liyau}. From {Hopf's Umlaufsatz (see e.g. \cite[Theorem 2.2.10]{BaerDiffgeo}) it} can be inferred that $T[\gamma] \in \{-1 , 1\}$. Also, Proposition \ref{prop:periodorb} implies that $n \geq 2 $ or $\kappa \equiv const$. For a contradiction suppose that $n \geq 2$. Note that $\gamma$ is orbitlike, see Corollary \ref{cor:noclwv}. Additionally, 
\begin{equation*}
\kappa_0^2 = \frac{2 \lambda + 4}{2-p^2} = \frac{4}{2-p^2} < 4 
\end{equation*}
and Corollary \ref{cor:totrwealk} implies that 
$m =  T[\gamma] = \pm 1$. Notice that $\gamma$ is rotational because of Proposition \ref{prop:simcloorbit}. By Proposition \ref{prop:closedrot} and $4C = \kappa_0^4 - 4 \kappa_0^2$ (see \eqref{eq:onceint}) we obtain
\begin{align*}
\pi & = |\pi m | = \left\vert \int_0^\frac{2nK(p)}{r} \sqrt{-C} \frac{\kappa^2}{4C +  4 \kappa^2} \diff s \right\vert
\\ & = \frac{\sqrt{\kappa_0^2 - \frac{1}{4}\kappa_0^4 }}{r} \int_0^{2nK(p)} \frac{\dn^2(s,p)}{\kappa_0^2 - 4 +4 \dn^2(s,p)  } \diff s
\\ & = 2n \sqrt{1- \frac{\kappa_0^2}{4}} \frac{|\kappa_0|}{r} \int_0^{K(p)} \frac{\dn^2(s,p)}{\frac{4}{2-p^2} - 4p^2 \sn^2(s,p) } \diff s 
\\ & = n \sqrt{1 - \frac{1}{2-p^2}} (2-p^2) \int_0^{K(p)} \frac{\dn^2(s,p)}{1- p^2(2-p^2)\sn^2(s,p) } \diff s 
 \\ & = \frac{n}{2} \left( 2 \sqrt{1-p^2} \sqrt{2-p^2} \int_0^{K(p)} \frac{\dn^2(s,p)}{1 - p^2(2-p^2)\sn^2(s,p)} \diff s\right) .
\end{align*}
According to \cite[Proof of Proposition 5.3, p.21]{LangerSinger} the expression in parentheses is always {strictly} larger than $\pi$. Since $n \geq 2$ this leads to {the desired} contradiction.  
{ Now suppose that $\mathcal{E}(\gamma) = 16$. Again because of Corollary \ref{cor:noclwv} and Proposition \ref{prop:integr}, $\gamma$ is either orbitlike or circular. Suppose now that $\gamma$ is orbitlike. Similar to the Proof of Lemma \ref{energ:orbitl} one computes using $\lambda = 0$ that %\begin{equation*}
$16 = \mathcal{E}(\gamma) = 8 n \frac{E(p)}{\sqrt{2-p^2}}$, %\end{equation*}
in particular 
\begin{equation*}
2 \frac{\sqrt{2-p^2}}{E(p)} = n \in \mathbb{N}.
\end{equation*}
However, according to Proposition \ref{prop:ellipt}, the {number} on the left hand side is stricly between $\frac{4\sqrt{2}}{\pi} \approx 1.80063$ and $2$, and hence cannot be natural. We conclude that $\gamma$ has to be circular, i.e. $\kappa[\gamma] \equiv const$.}
\end{proof}

\begin{cor}\label{cor:only_the_lonely}
 Let $\gamma$ be a closed free elastica with $\mathcal{E}(\gamma) \leq 16$. Then $\gamma$ is the Clifford elastica \eqref{eq:CliffordElastica} up to translation, rescaling and reparametrization.
\end{cor}

\begin{proof}
 Since $\kappa[\gamma] \equiv const.$ by Proposition \ref{prop:simfreel}, it follows that $\kappa[\gamma] \equiv \sqrt{2}$ by Definition \ref{def:elastica}. Denote the Clifford elastica \eqref{eq:CliffordElastica} by $\tau$, then one finds $\kappa[\tau] \equiv \sqrt{2}$, thus 
 {$\gamma \equiv \tau$} up to isometries of $\mathbb{H}^2$ and reparametrization. Note that inversions are not needed, since, by  Proposition \ref{prop:curvconst}, $\gamma$ is given as a Euclidean circle in $\mathbb{H}^2$, which can be mapped to $\tau$ using translations and rescalings only.
\end{proof}

\section{Proof of the Main Results}
\label{sec:MainResultProofs}
In this section we show the proofs of the main results. We start with the fundamental result of \cite{DS17} that settles question of long time existence and identifies the uniform-in-time boundedness of the hyperbolic length as sufficient for the convergence 
% We start with {a slight variation of} the main result from \cite{DS17}  on the elastic flow of curves in hyperbolic space.
 \begin{theorem}[{Slight variation of \cite[Theorem 1.1 (i)]{DS17}}]
  \label{thm:LTE}
  Let $f_0\colon  \Sph \to \Hyp$ be a smooth immersion and $\lambda \geq 0$. Then there exists a unique, smooth, global solution $f\colon \Sph \times [0,\infty) \to \mathbb{H}^2$ to the initial value problem 
  \begin{equation}
 \left\{ \begin{array}{rll}\partial_t f &= - \nabla_{\!L^2}\E_\lambda(f) , & \mbox{ on }\mathbb{S}^1 \times (0,\infty), \\
f(\cdot, 0) &= f_0, & \mbox{ on } \mathbb{S}^1.% \, ,
\end{array}\right. 
 \end{equation}
 Moreover, if the length $\mathcal{L}(f(\cdot, t))$ of $f$ is uniformly bounded on $[0,\infty)$, then the solution subconverges smoothly after appropriate scaling, translation in the $x$-direction and reparametrization to an elastic curve, {which is a free elastica in the case of $\lambda = 0$ (see Definition \ref{def:elastica})}.
 \end{theorem}
  \begin{remark}
 \vspace{-.01cm}
 \label{rem:subconvergenceExplained1}\leavevmode
\begin{enumerate}
 \item {The precise formulation of the subconvergence result is as follows: Denote the constant speed reparametrization of $f$ by $\tilde f$, then there exists smooth functions $p \colon [0,\infty) \to \R$, $a \colon [0,\infty) \to \R_{>0}$ such that the isometric image $\hat f(t,\cdot) := a(t)(\tilde f(t,\cdot) - (p(t), 0)^T)$ of $f$ subconverges smoothly to an elastic curve, i.e. for any $t_n \to \infty$ there exist some subsequence $t_{n_k}$ and some elastica $f_\infty$ with $\|\hat f(t_{n_k}, \cdot) - f_\infty\|_{W^{m,2}} \to 0$ for all $m \in \N$ (c.f. \cite[p. 22]{DS17}).}
 \item Note that scaling and translation in the $x$-direction are isometries in $\mathbb{H}^2$. {Hence $\widehat{f}(t,\cdot)$ is an isometric image of $f(t, \cdot)$.}
 \item The uniform bound of the length is immediate if $\lambda > 0$, as this implies
  \[
   \mathcal{L}(f(\cdot, t)) \leq \frac{1}{\lambda}\E_\lambda(f(\cdot, t)) \leq \frac{1}{\lambda}\E_\lambda(f(\cdot, 0)) < \infty,
  \]
  since the energy is monotonically decreasing during the flow. This observation was used in {\cite[Theorem 1.1 (i)]{DS17}}, which states the above subconvergence result only for $\lambda > 0$, but the proof of {\cite[Theorem 1.1 (i)]{DS17}} shows that any bound on the length is sufficient for the subconvergence.
\end{enumerate}
\end{remark}
With a \emph{Lojasiewicz-Simon gradient inequality} we can actually improve the subconvergence to convergence: 
\begin{remark}\label{rem:convergence}
If the elastic flow $f$ subconverges to an elastic curve $f_\infty$ {in the sense of  Remark \ref{rem:subconvergenceExplained1} (1)}, then it converges smoothly to $f_\infty$.\end{remark}
Since a proof of this result is beyond the scope of this article we only give a sketch here and refer the reader to \cite{Loja} for details
%We start with a sketch of the proof of Remark \ref{rem:convergence}.
\label{proof:convergence}
\begin{proof}[{Sketch of Proof of Remark \ref{rem:convergence}}] The convergence is usually shown with a \L{}ojasievicz-Simon inequality  (c.f. \cite{chill2009} and  \cite[Theorem 1.2]{Loja}). It is enough to show convergence in $L^2$, as a subsequence argument proves convergence in all higher Sobolev norms. By \cite[Corollary 3.11]{CHILL2003572} (see also \cite[p. 355]{chill2009}) it is sufficient for the 
 \L{}ojasievicz-Simon inequality to hold if one shows that there exists a neighborhood $U \subset H^{4,\bot}$ of the sublimit $f_\infty$ (where $H^{4,\bot}$ is defined analogously to \cite{Loja})  such that $\E\colon  U \to \R$ and $\nabla \E \colon  U \to L^{2,\bot}$ are analytic and the Frech\'e{}t derivative $(\nabla \E)'(f_\infty)$ is Fredholm of index zero. Identifying the tangent space of $\Hyp \subset \R^2$ with $\R^2$ and choosing $U$ small enough such that $u+\phi$ is still immersed and the second component satisfies $(u+\phi)_2 > 0$ for all $\phi \in U$ we find similar to \cite[Theorem 3.5]{Loja} that the two mappings are analytic (the existence of such an $U$ is guaranteed by Sobolev embeddings). Moreover, since for any normal vector field $N$ along $f_\infty$ we have
 \[
  \nabla \E_\lambda(f_\infty+u N) = \left( \frac{ \left((f_\infty+u N)_2\right)^4}{|\partial_x (f_\infty+ u N)|^4} \partial_x^4 u + \text{lower order terms}\right)N
 \]
by \cite[p. 11]{DS17}, one finds that 
\[
  (\nabla \E_\lambda)'(f_\infty)(u N) =  \frac{ \left((f_\infty)_2\right)^4}{|\partial_x f_\infty|^4} (\partial_x^4 u)N + B(u)N.  \]
  By the Sobolev embedding theorem we see that $B \colon  H^{4,\bot} \to L^{2,\bot}$ is a compact mapping, thus $ (\nabla \E_\lambda)'(f_\infty) \colon  H^{4,\bot} \to L ^{2,\bot}$ is Fredholm of index zero. This shows that a \L{}ojasievicz-Simon inequality holds on $U$, from which one can deduce the claim similarly to \cite[Theorem 1.2]{Loja}.
 \end{proof}
 
 We now show Theorem \ref{thm:main1}.
 
 \begin{proof}[Proof of Theorem \ref{thm:main1}]
Equation \eqref{eq:2.1} follows {immediately} from Theorem \ref{thm:reilly}. The second part, i.e. \eqref{eq:2.2}, can be inferred from Corollary \ref{cor:8.4} as follows:  
%This corollary implies that for each $\varepsilon > 0$
%\changedA{
%\begin{equation*}
%\inf \left\lbrace  \frac{\mathcal{E}(\gamma)}{\mathcal{L}(\gamma) } : \gamma \in C^\infty(\mathbb{S}^1, \mathbb{H}^2),\, \mathcal{E(\gamma)} \leq 16 + \varepsilon \right\rbrace = 0 . 
%\end{equation*}
%To show this,

Let $\delta > 0$ and consider a smooth curve $\gamma_0$ such that $16 < \mathcal{E}(\gamma_0) \leq 16 + {\delta}$ and $T[\gamma_0] = 0 $, whose existence is provided by Corollary \ref{cor:8.4}. From Theorem \ref{thm:LTE} we obtain the evolution $(\gamma_t)_{t \geq 0 }$ of $\gamma_0$ by the elastic flow with $\lambda = 0$. Then $\mathcal{E}(\gamma_t) \leq  16 + {\delta}$ but according to Corollary \ref{rem:unblength} we have $\mathcal{L}(\gamma_t) \rightarrow \infty $, at least up to a subsequence. This subsequence produces arbitrarily small values of $\frac{\mathcal{E}}{\mathcal{L}}$.
\end{proof}

\begin{proof}[Proof of Theorem \ref{thm:main2}]
Let $f_0$ be a smooth immersion with $\mathcal{E} (f_0) \leq 16$. First, we assume that $\delta := 16-\mathcal{E}(f_0) > 0$. Since \begin{equation}   \label{eq:FlowDecreasing}                                                                                       
\frac{\diff}{\diff t} \mathcal{E}(f_t) = \langle \nabla \E(f_t), \partial_t f_t\rangle _{L^2} =  -\|\nabla \E(f_t)\|_{L^2}^2 \leq 0                                                                                                        \end{equation} we find that $\E(f_t) \leq \E(f_0)  \leq 16 - \delta$, thus 
$\mathcal{L}(f_t) \leq c_\delta \mathcal{E}(f_t) \leq c_\delta \mathcal{E}(f_0)$ for all $t$ by \eqref{eq:2.1}. %Theorem \ref{thm:reilly}. 
Hence, by % the first part follows from 
Theorem \ref{thm:LTE} and Remark \ref{rem:convergence}, the flow converges in the sense of  Remark \ref{rem:subconvergenceExplained1} (1) to some free elastica with energy {below} $16$.
In Corollary {\ref{cor:only_the_lonely}} we show that the only free elastica with energy below 16 is the Clifford Elastica, which finishes the proof in this case. If $\mathcal{E}(f_0) = 16$, then 
%either $f_0$ is already elastic and the claim is trivially satisfied or 
{$f_0$ is not elastic by Corollary \ref{cor:only_the_lonely}, thus}
$\mathcal{E}(f_t) < 16$ for all $t > 0$ by \eqref{eq:FlowDecreasing}, from which we can deduce the claim as above.
\end{proof}

Similarly to the proof of \eqref{eq:2.2} we show Theorem \ref{thm:nonconvergence}.

\begin{proof}[{Proof of Theorem \ref{thm:nonconvergence}}]
Theorem \ref{thm:nonconvergence} is immediate from Corollary \ref{rem:unblength} and Corollary \ref{cor:8.4}.
\end{proof}

\appendix

\section{Minor Proofs}

\subsection{Proof of Proposition \ref{prop:integr}}\label{app:A}
\begin{proof}[Proof of Proposition \ref{prop:integr}]
Remember that $u = \kappa^2 \geq 0 $. Therefore we aim to classify nonnegative solutions of 
%\begin{equation*}
$u'^2 + u^3 - (2 \lambda + 4 ) u^2 - 4C u = 0 
$, see \eqref{eq:squaredelas}. %\end{equation*}
This equation is of the form 
%\begin{equation*}
$u'^2 = P(u)
$ %\end{equation*}
for the polynomial $P$ given by 
%\begin{equation*}
$P(x) = - (x-\alpha) (x- \beta)(x- \gamma),
$ %\end{equation*}
where
\begin{equation*}
\{\alpha,\beta,\gamma\} = \{ 0 , (\lambda + 2) + \sqrt{(\lambda+2)^2 + 4C}, (\lambda + 2) - \sqrt{(\lambda + 2)^2 + 4C } \}.
\end{equation*}
Note that $\alpha, \beta, \gamma$ have to be real-valued since otherwise $P(u) $ can only have one real  root, which is zero. However then $P_{\mid_{(0,\infty)}}$ is negative, which contradicts the existence of positive real-valued solutions of $u'^2 = P(u)$. Note also that one root of $P$ has to be strictly positive for the very same reason. 
From now on we adhere to the convention $\alpha \leq \beta \leq \gamma$ as in \cite{Davis}. Note that $\alpha \neq \gamma$ because otherwise $\alpha = \beta = \gamma = 0$ and the equation reads $(u')^2 = -u^3 $.  This however has no nonnegative solution except for the trivial one.  Observe also that $ \beta \leq u(s) \leq \gamma$ for all $s$, since otherwise nonnegativity is violated again (since $\alpha \leq 0 $). In particular we find $\beta \neq \gamma$. We distinguish between two cases: $\alpha \neq \beta$ and $\alpha = \beta$. Note that 
\begin{equation*}
\beta = \alpha \; \Leftrightarrow  \;  0 = \lambda + 2 - \sqrt{(\lambda + 2)^2 +4C} \; \Rightarrow C = 0 .
\end{equation*} 
 Conversely, note that if there exists a solution $u$ with $C= 0$ then $\lambda + 2 \geq 0 $ since otherwise all roots are nonpositive and $u'^2 = P(u)$ cannot be true. Therefore 
%\begin{equation}
$\alpha = \lambda + 2 - \sqrt{(\lambda+ 2)^2+ 4C} = 0 = \beta. 
$ %\end{equation} 
 As a conclusion, $\alpha = \beta$ holds if and only if $C = 0$. 

\textbf{Case 1:} $\alpha \neq \beta$ or $C \neq 0 $. %Observe that then
In this case we find $\alpha < \beta < \gamma$. 
 We substitute $v= - u(2 \cdot) $ to obtain 
%\begin{equation}
$v'^2  = 4 (v+ \alpha) (v+ \beta) (v + \gamma)$. 
%\end{equation} 
We infer from \cite[p.157, Eq.(10,11)]{Davis} that the general solution is given by
\begin{equation*}
v(x) = - ( \gamma + (\beta - \gamma) \sn^2(\theta (x-x_0) , p))\end{equation*}
where $x_0 \in \mathbb{R}$ is some constant and $\theta^2 =- \alpha + \gamma$ as well as $p^2 =   \frac{\gamma- \beta}{\gamma - \alpha}$. Using that $u'(0) = 0 $ we can choose $x_0 = 0 $.   Resubstituting we obtain 
%\begin{equation*}
$u(x)= \gamma (1 - q^2 \sn^2(rs, p))$, 
%\end{equation*}
where $q^2 = \frac{\gamma- \beta}{\gamma}$ and $ r^2= \frac{1}{4} (\gamma- \alpha)$. Notice that any such solution $u$ is global and  attains its global maximum $ \kappa_0^2= \gamma$. \\
First, for the wavelike case, $C>0$, $\alpha, \beta$ and $\gamma$ have to be ordered in the following way: $\alpha = \lambda + 2 - \sqrt{(\lambda + 2)^2 + 4 C }$, $\beta = 0 $ and $\gamma = \lambda + 2 + \sqrt{(\lambda + 2)^2 + 4C}$. Hence $ q = 1 $ and $u(s) = \gamma \cn^2(rs, p) = \kappa_0^2 \cn^2(rs,p)$.
Note that  
\begin{equation*}
 p^2 = \frac{\gamma- \beta}{\gamma- \alpha } =  \frac{\lambda + 2 + \sqrt{(\lambda+ 2)^2 + 4C}}{2\sqrt{(\lambda + 2)^2 + 4C }}
\end{equation*}  
and therefore $p \in ( \frac{1}{\sqrt{2}}, 1 ) $. Moreover,
\begin{equation*}
\kappa_0^2 = \lambda + 2 + \sqrt{(\lambda + 2)^2+ 4C }=2p^2 \sqrt{(\lambda + 2)^2 + 4C} = 2p^2 ( \kappa_0^2 - ( \lambda + 2) ).
\end{equation*}
Solving for $\kappa_0^2$ we obtain $\kappa_0^2 = \frac{p^2(2\lambda + 4)}{2p^2 - 1}$. 
Rewriting 
%\begin{equation*}
$\kappa_0^2 = \frac{1}{2} \left( 1 + \frac{1}{2p^2- 1}\right) (2 \lambda + 4) $
%\end{equation*}
we infer that $\kappa_0^2 > 2 \lambda + 4$. 
 Additionally, 
\begin{equation*}
r^2 = \frac{1}{4} ( \gamma - \alpha ) = \frac{1}{4} \left( 2 \sqrt{(\lambda+ 2)^2 + 4C} \right) = \frac{1}{4} \frac{\kappa_0^2}{p^2}.
\end{equation*}
 In the orbitlike case, $C< 0$, we set $\alpha = 0 $ , $\beta = \lambda + 2 -  \sqrt{(\lambda + 2)^2 + 4C}$ and   $\gamma = \lambda + 2 + \sqrt{(\lambda + 2)^2 + 4C}$. %Note that then
Whence $q^2 = p^2 $ and thus  $u(x) = \gamma \dn^2(rs, p) $, where 
 \begin{equation*}
 p^2  = \frac{2 \sqrt{(\lambda + 2)^2 + 4C}}{\lambda + 2 + \sqrt{(\lambda + 2)^2 + 4C}}. 
 \end{equation*}
 Note that 
 \begin{equation*}
 \kappa_0^2 = \gamma = \lambda +2 + \sqrt{(\lambda + 2)^2 + 4C} = \lambda + 2 + \frac{p^2 \kappa_0^2}{2}.
 \end{equation*}
 Solving for $\kappa_0^2$ we obtain $\kappa_0^2 = \frac{2 \lambda + 4}{2-p^2}$. In particular we infer that $\lambda + 2 < \kappa_0^2 < 2 \lambda + 4$.

 \textbf{Case 2:} $\alpha = \beta$ or $C= 0 $.  If $C = 0 $, the differential equation reads 
 \begin{equation}
 u'^2 + u^3 - (2 \lambda + 4 ) u^2 = 0. 
 \end{equation}
 We infer that either $u \equiv 0 $ or $u \equiv 2 \lambda + 4 $ or there is $x_0 \in \mathbb{R}$ such that $0 < x_0 < 2 \lambda + 4 $. In the last case note that in a neighborhood of $x_0$ we find %\begin{equation}
 $\frac{u'^2}{u^2} = 2\lambda + 4  - u 
$. Whence, % \end{equation}
% and therefore, 
substituting $ v= \log(u) $ yields 
% \begin{equation}
$ v'^2 = 2 \lambda + 4 - e^v .
$ % \end{equation}
 Defining $\theta := e^{\frac{-v}{2}} $ we obtain 
% \begin{equation}
$ 4 \theta'^2 = (2 \lambda + 4) \theta^2- 1 .
$% \end{equation}

Substituting  $\widetilde \theta = \sqrt{2 \lambda + 4} \theta ( \frac{2}{\sqrt{2\lambda + 4}} \cdot ) $ we obtain $1 + \widetilde{\theta}'^2 = \theta^2$. Setting $w := \mathrm{Arcosh}( \theta ) $ we obtain that $w = \pm ( t + x_1)$ for some $x_1 \in \mathbb{R}$ and therefore, tracing all the substitutions back we obtain that $u(x) = (2 \lambda + 4)\cosh^{-2} \left(\frac{\sqrt{2\lambda + 4 }}{2} (x- x_1) \right)$.  The claim follows using that $u'(0) = 0 $. 
\end{proof}
%Finally, note that we have used the condition $u'(0)= 0 $ only to fix a translation.

%\subsection{Proof of Lemma \ref{lem:inired}}\label{subs:proofofleminired}

\subsection{Proof of Proposition \ref{prop:nonvan}}\label{app:kleeblatt}
\label{appendix_theta}
Let $\gamma$ be a globally defined elastic curve parametrized by hyperbolic arclength. We will need several lemmas to prove the claim. Recall from the proof of Theorem \ref{thm:explpara} (see \eqref{eq:wichtig} and use $T= \gamma'$) that we have a differential equation for $\gamma$ in $\mathbb{C}$, namely 
%\begin{equation}
$\theta(s) \gamma'(s) = a \gamma(s)^2 + c$ for $s \in \mathbb{R}$, where $\theta(s) := \kappa^2(s) - \lambda + 2 i \kappa'$. %\end{equation}
We can not divide by $\theta$ a priori and hence the Picard-Lindel\"of Theorem is not applicable.{ Recall also that by \eqref{eq:kilextf} 
\begin{equation}\label{eq:kilextcomp}
\widetilde{J}_{\widetilde{\gamma}} (z) = az^2 + c \quad \textrm{for all } z \in \mathbb{C} \textrm{ such that }  \mathrm{Im}(z) > 0 .
\end{equation}
 If the Killing field has a zero in $\mathbb{H}^2$, then one can infer from \eqref{eq:kilextcomp} that $ac> 0$. Therefore $\gamma$ is rotational and hence orbitlike, see Definition \ref{def:classi} and Proposition \ref{prop:closedrot}. % Proposition \ref{prop:ordred}. 
 }\\
Since $\gamma$ is an immersion, the following lemma is immediate.

\begin{lemma}[Rephrasing the Problem in Terms of $\theta$]\label{lem:vanequ}
 The function $\theta$ vanishes nowhere if and only if $ i \sqrt{\frac{c}{a}} \not \in \gamma(\mathbb{R})$%. More concretely  
 , i.e. $\theta(s) $ vanishes if and only if $a\gamma(s)^2+ c = 0$. 
\end{lemma}
% \begin{proof}
% Recalling that $a \gamma^2 + c= \theta \gamma'$ and that $\gamma$ is an immersion the claim follows immediately. 
% \end{proof}

\begin{remark}
Because of the previous lemma it suffices to show that $\theta$ vanishes nowhere for each globally defined elastic curve $\gamma$.  
\end{remark}

\begin{lemma}[Parameter Discussion]\label{lem:8.13} 
Let $\gamma$ be as above. If $\theta$ vanishes at some $t_1 \in \mathbb{R}$, then $\kappa(t_1)$ is a point of minimum curvature of $\gamma$ and the following parameter identities hold 
% \begin{enumerate}
% \item $r^2p^2 = 1$
% \item $ (\lambda + 2)^2 + 4C = 4$ 
% \item $\kappa_0^2 = \lambda + 4 $
% \item $\lambda > 0 $.
% \end{enumerate}
\begin{align*}
 &(1)\quad r^2p^2 = 1 && (3) \quad \kappa_0^2 = \lambda + 4 \\
 &(2)\quad (\lambda + 2)^2 + 4C = 4 &&(4) \quad
\lambda > 0. 
\end{align*}
Furthermore, %Parameter identity
$(1),(2)$ and $(3)$ are %not only necessary but 
also sufficient for $\theta$ having a zero. Moreover, $(1)$,$(2)$ and $(3)$ are all equivalent for orbitlike elastica.
\end{lemma}

 \begin{proof}
{As we discussed in the introduction of this subsection, $\theta$ can vanish only provided that the Killing field has a zero in $\mathbb{H}^2$, which implies that $\gamma$ is rotational and orbitlike, see the arguments in the aforementioned introduction.} 
 Observe that for each orbitlike elastica $\gamma$ it holds that $\kappa^2 \geq \kappa_0^2 (1- p^2)$, see Proposition \ref{prop:integr} and Definition \ref{def:B2}. We can compute using the definition of $\theta$ in Proposition \ref{prop:nonvan}, \eqref{eq:onceint} and Proposition \ref{prop:integr} multiple times
 \begin{align}\label{eq:param} 
 0 & = |\theta(t_1)|^2 =  ( \kappa(t_1)^2 - \lambda)^2 + 4 \kappa'(t_1)^2 = \lambda^2 + 4C + 4 \kappa(t_1)^2 \\
  & \geq \lambda^2 + 4C  + 4 \kappa_0^2 (1- p^2) = \lambda^2 + 4C + 4 (2\lambda + 4) \frac{1-p^2}{2-p^2} \nonumber
  \\ & = \lambda^2 + 4C + 4 \left( 2 + \lambda - \sqrt{(\lambda + 2)^2 +4C} \right) \nonumber
  \\ & =  4 + (\lambda + 2)^2 +4C  - 4 \sqrt{ (\lambda + 2)^2 + 4C} \nonumber\\
   & = \left( \sqrt{(\lambda + 2)^2 + 4C } - 2 \right)^2 \geq 0 .\nonumber
 \end{align}
 We infer that all inequalities in the above chain have to be equalities. From this follows that $\kappa(t_1)^2 = \kappa_0^2 (1- p^2)$ which is the minimum possible curvature (see Definition \ref{def:B2}) and parameter identity no. $(2)$ using that equality holds in the last step. For parameter identity no. (1) observe using \ref{prop:integr} that 
 \begin{equation*}
 r^2 p^2 = \frac{2\lambda  + 4}{4} \frac{p^2}{2-p^2} = \frac{\sqrt{(\lambda + 2)^2 + 4C}}{2}.
\end{equation*}  
For no. (3) observe that 
\begin{equation*}
4 = ( \lambda + 2)^2 + 4C = \lambda^2 + 4C + 4\lambda + 4 = (\kappa_0^2- \lambda)^2 - 4 \kappa_0^2 + 4 \lambda + 4, 
\end{equation*}
and thus
%\begin{equation*}
$(\kappa_0^2 - \lambda )( \kappa_0^2 - \lambda - 4 ) = 0 $
%\end{equation*}
which is equivalent to $\kappa_0^2 - \lambda - 4 = 0 $ since $\kappa_0^2 - \lambda \geq 2 > 0 $. Therefore, as an easy computation shows, parameter identity $(1), (2), (3)$ are all equivalent. For parameter identity no. (4) note that orbitlike elastica satisfy $ \kappa_0^2 < 2 \lambda + 4$, as Proposition \ref{prop:integr} 
implies. However $\lambda + 4 < 2 \lambda + 4$ holds true if and only if $\lambda > 0 $. The sufficiency of $(2)$ is clear, when we compute similar to \eqref{eq:param}: 
\begin{equation*}
  0 =   \left( \sqrt{(\lambda + 2)^2 + 4C } - 2 \right)^2 = \lambda^2 + 4C + 4 \kappa_0^2 (1-p^2) = \left\vert \theta\left( \tfrac{K(p)}{r} \right) \right\vert^2 
\end{equation*}
and $(1)$ and $(3)$ are sufficient as well since they are equivalent to $(2)$. 
 \end{proof}
 
\begin{cor}\label{lem:nonnonvan}
If $\theta$ %is not nonvanishing, it vanishes only at
has any real zeros, then they are given by $s_l = (2l+1) \frac{K(p)}{r} = (2l+1) K(p) p  $.
\end{cor}

\begin{proof}
The points $s_l$ are exactly the points of minimal curvature, see Definition \ref{def:B2} and Proposition \ref{prop:integr}. 
\end{proof}

\begin{cor}
On $( -\frac{K(p)}{r}, \frac{K(p)}{r} ) $ the reciprocal of $\theta$ satisfies% has the following representation
\begin{equation*}
\frac{1}{\theta(s) } = \frac{1}{4}- \frac{i}{2} \frac{\kappa'}{\kappa^2 - \lambda}.
\end{equation*}
\end{cor}

\begin{proof}
Observe that parameter identity $(2)$ in Lemma \ref{lem:8.13} implies $\lambda^2 + 4C = - 4 \lambda$. The rest is a short computation using \eqref{eq:onceint}:  
\begin{align*}
 \frac{1}{\theta(s) } & = \frac{1}{(\kappa^2 - \lambda) + 2i \kappa'(s) }
 %\\ & = \frac{\kappa^2 - \lambda - 2 i \kappa'}{(\kappa^2 - \lambda)^2 + 4 \kappa'^2 } 
  = \frac{\kappa^2 - \lambda - 2 i \kappa'}{\lambda^2 + 4 C + 4 \kappa^2} 
 %\\ &
 = \frac{\kappa^2 - \lambda - 2i \kappa'}{4 \kappa^2 - 4 \lambda } = \frac{1}{4}- \frac{i}{2} \frac{\kappa'}{\kappa^2 - \lambda}. \qedhere
\end{align*}
\end{proof}

\begin{lemma}[Explicit Parametrization near $s = 0$]\label{lem:rep}
 Let $\gamma$ be a globally defined elastic curve with $\theta$ %is not nonvanishing.
 vanishing somewhere. Then there exists $z_0 \in \mathbb{C} \setminus \mathbb{R}$ such that
 \begin{equation*}
 \gamma(t) = \sqrt{\frac{c}{a}} \frac{x(t) - iy(t)}{1 + i x(t) y(t)}  \quad \forall t \in \left( - \frac{K(p)}{r}, \frac{K(p)}{r} \right),
 \end{equation*}
 where 
% \begin{equation}
$ x(t) =  \tan \left( \frac{\sqrt{\lambda}}{4} t + z_0 \right)$ 
% \end{equation}
and 
%\begin{equation}
$y(t) = \tanh \left( \frac{1}{4} \log \left\vert \frac{\kappa + \sqrt{\lambda}}{\kappa- \sqrt{\lambda}} \right\vert \right)$.
%\end{equation}
\end{lemma}

\begin{proof}
First note that $\gamma(0) \neq i \sqrt{\frac{c}{a}}$ since $\theta(0) \neq 0 $, see Lemma \ref{lem:vanequ}. Therefore we can use similar arguments as %techniques like
in the proof of Theorem $\ref{thm:explpara}$ to obtain that
\begin{equation*}
\gamma(t) = \sqrt{\frac{c}{a}} \tan \left( \int_0^t \frac{\sqrt{ac}}{\theta(s)}\diff s  + z_0 \right) 
\end{equation*} 
 in a neighborhood of $t = 0 $ for some $z_0 \in \mathbb{C}$ such that $\gamma(0) = \sqrt{\frac{c}{a}} \tan(z_0)$. Such a $z_0$ exists since $\gamma(0) \in i \mathbb{R} \setminus \left\lbrace i \sqrt{\frac{c}{a}} \right\rbrace$ and $\tan$ is surjective on $\mathbb{C} \setminus \{ i, -1 \}$. Observe also that $z_0 \not \in \mathbb{R}$ since otherwise $\gamma(0) \in \mathbb{R}$, a contradiction. 
Since $\sqrt{ac} = \sqrt{\frac{-\lambda^2 - 4 C}{4}} = \sqrt{\lambda}$ we find
\begin{align*}
\gamma(t) %& = \sqrt{\frac{c}{a}} \tan \left( \sqrt{\lambda} \int_0^t \frac{1}{\theta(s)}\diff s  + z_0 \right) \\
 %&
 =   \sqrt{\tfrac{c}{a}} \tan \big( \sqrt{ \lambda } \int_0^t \tfrac{1}{4} - \tfrac{i}{2} \tfrac{\kappa'}{\kappa^2 - \lambda} \diff s  + z_0 \big) %\\
% & = \sqrt{\frac{c}{a}} \tan \left( \sqrt{ \lambda }   \frac{t}{4} + z_0 - \int_0^t\frac{i}{4} \left( \frac{\kappa'}{\kappa- \sqrt{\lambda}} - \frac{\kappa'}{\kappa + \sqrt{\lambda}}  \right) \diff s  \right) \\  
 %&
 = \tan \big( \sqrt{ \lambda }   \tfrac{t}{4} + z_0 - \tfrac{i}{4}\log \left\vert \tfrac{\kappa- \sqrt{\lambda}}{\kappa + \sqrt{\lambda}}\right\vert \diff s  \big)
\end{align*}
in a neighborhood of zero. 
Using $\tan ( z+w ) = \frac{\tan(z) + \tan(w)}{1- \tan(z) \tan(w) } $ and $\tan(iz) = i \tanh(z) $  the desired formula follows in a neighborhood of $t= 0 $. Observe now that $  4 ( \kappa^2 - \lambda)  = \lambda^2 + 4C + 4 \kappa^2 > 0 $ on $( - \frac{K(p)}{r}, \frac{K(p)}{r} ) $ and therefore the solution from above exists on $(- \frac{K(p)}{r}, \frac{K(p)}{r} )$, since otherwise this would contradict maximality of the existence interval as $\frac{1}{\theta}$ is locally Lipschitz continuous.   
\end{proof}

\begin{lemma}\label{lem:cloana}
Let $x(t),y(t)$ be defined as in Lemma \ref{lem:rep}.%Then

\begin{enumerate}
\item For all $ t \in ( - \frac{K(p)}{r}, \frac{K(p)}{r} ) $ it holds $\displaystyle
y(t)  = \frac{|\kappa- \sqrt{\lambda}|^\frac{1}{2} - |\kappa + \sqrt{\lambda}|^\frac{1}{2}}{|\kappa - \sqrt{\lambda}|^\frac{1}{2} + |\kappa + \sqrt{\lambda} |^\frac{1}{2}}$.
\item  For all $ t \in ( - \frac{K(p)}{r}, \frac{K(p)}{r} )$
%\begin{equation*}
it holds $1-y^2(t) = \frac{4 \sqrt{\kappa^2 - \lambda }}{|\kappa- \sqrt{\lambda}| + | \kappa + \sqrt{\lambda }|  + 2 \sqrt{\kappa^2- \lambda} }$,
%\end{equation*}
and in particular $\lim_{t \rightarrow \frac{K(p)}{r}} (1- y^2(t)) = 0 $. 
\item $x, x'$ are bounded on  $( - \frac{K(p)}{r}, \frac{K(p)}{r} ) $.
\item 
%\begin{equation*}
$\displaystyle y'(t) \to -1$ as $t \rightarrow \frac{K(p)}{r}$.
%\end{equation*}
\end{enumerate}
\end{lemma}

\begin{proof}
For part $(1)$ use $\tanh(z) = \frac{e^{z}- e^{-z}}{e^{z} + e^{-z}}$. Part $(2)$ follows easily from part $(1)$.  Part $(3)$ is a standard observation using that $z_0 \in \mathbb{R}$, thus the tangent expression stays away from all its poles at $\{ (2m+ 1) \tfrac{\pi}{2} | m \in \mathbb{N} \} $. 
For $(4)$ we distinguish between two cases, the first one being $\kappa > 0 $. In this case it we find $\kappa \geq \sqrt{\lambda}$.  Here we can derive a more explicit expression for $y$ from $(1)$, namely 
\begin{equation*}
y(t) = \frac{\sqrt{\kappa - \sqrt{\lambda}} - \sqrt{\kappa + \sqrt{\lambda}}}{\sqrt{\kappa - \sqrt{\lambda}} + \sqrt{\kappa + \sqrt{\lambda}}} .
\end{equation*}
Multiplying numerator and denominator by $ \sqrt{\kappa - \sqrt{\lambda}} - \sqrt{\kappa+ \sqrt{\lambda}}$ we find 
$y(t) = \frac{\sqrt{\kappa^2- \lambda}- \kappa}{\lambda}$, whence%
%\end{equation*}
%Then we compute that 
%
\begin{equation*}
y'(t) = \frac{1}{\sqrt{\lambda}} \frac{\kappa'}{\sqrt{\kappa^2- \lambda}} \left( \kappa- \sqrt{\kappa^2 - \lambda} \right) .
\end{equation*}
Now the limit of $y'(t)$ as $t \rightarrow \frac{K(p)}{r}$ is of indeterminate form. We solve this as follows: In our case the curvature is given by $\kappa(s) = \kappa_0 \dn(rs, p)$. Thus 
\begin{equation*}
\kappa'(s) = - \kappa_0 r p^2 \sn(rs,p) \cn(rs,p) = - 2 r^2 p^2 \sn(rs,p) \cn(rs, p) = - 2 \sn(rs,p) \cn(rs,p) 
\end{equation*}
where we used the first parameter identity in Lemma \ref{lem:8.13}. Using the third and first parameter identity in the very same lemma we find 
\begin{align*}
\kappa^2- \lambda &  =  \kappa_0^2 \dn^2(rs,p) - \lambda = \kappa_0^2 \dn^2(rs, p) - \kappa_0^2 + 4 
\\ & = - \kappa_0^2 p^2 \sn^2(rs,p) + 4 = 4 - 4r^2p^2 \sn^2(rs,p) = 4 \cn^2(rs,p) . 
\end{align*}
and therefore $\frac{\kappa'}{\sqrt{\kappa^2- \lambda}} = - \sn(rs,p) $ for each $s \in ( -\frac{K(p)}{r}, \frac{K(p)}{r} )$.
Using this we obtain 
\begin{equation*}
\lim_{t \rightarrow \frac{K(p)}{r}} y'(t) =\lim_{t \rightarrow \frac{K(p)}{r}}- \frac{ \kappa - \sqrt{\kappa^2- \lambda}}{\sqrt{\lambda}} \sn(rs,p) = - 1 .
\end{equation*}
The case that $\kappa< 0 $ can be treated analogously.% with slightly different expressions. 
\end{proof}

\begin{lemma}\label{lem:limit}
Let $\widehat{x} := \tan ( z_0 + \frac{1}{4} \sqrt{\lambda} pK(p) )$, where $z_0$ is as in Lemma \ref{lem:rep}. Then 
\begin{equation*}
\lim_{ t \rightarrow \frac{K(p)}{r} }\gamma'(t) = i \sqrt{\frac{c}{a}} \frac{1 \mp i \widehat{x}}{1 \pm i \widehat{x}}.
\end{equation*} 
\end{lemma}

\begin{proof}
We use Lemma \ref{lem:rep} and take the derivative of 
%\begin{equation*}
$ \gamma(t) = \sqrt{\frac{c}{a}} \frac{x(t) - iy(t) }{1+ ix(t)y(t) }$
%\end{equation*}
to obtain  
\begin{equation*}
\gamma'(t) = \sqrt{\frac{c}{a}} \frac{ x'(t) (1- y(t)^2 ) - i y'(t) (1+ x(t)^2 ) }{(1 + ix(t) y(t) )^2 }.
\end{equation*}
Using the identities derived in Lemma \ref{lem:cloana} we find
\begin{equation*}
\lim_{ t \rightarrow \frac{K(p)}{r} } \gamma'(t) =   i\sqrt{\frac{c}{a}} \frac{1+ x(\frac{K(p)}{r})^2}{(1 \pm i x(\frac{K(p)}{r}) )^2 } = i \sqrt{\frac{c}{a}} \frac{1+ \widehat{x}^2}{(1 \pm i \widehat{x})^2}
\end{equation*} 
the claim follows when we write $ 1 + \widehat{x}^2 = (1+ i\widehat{x})(1- i \widehat{x})$. 
\end{proof}

\begin{proof}[Proof of Proposition \ref{prop:nonvan}]
Assume that there exists a globally defined curve $\gamma$ such that a zero of $\widetilde{J}_\gamma$ lies in $\gamma(\mathbb{R})$ and $\gamma$ is parametrized with hyperbolic arclength. Therefore, if we look at $\gamma$ as a curve in $\mathbb{C}$ it satisfies
\begin{equation}\label{eq:8.100}
1 = \lim_{ t \rightarrow \frac{K(p)}{r}} \frac{|\gamma'(t)|}{\mathrm{Im}(\gamma(t))}.
\end{equation}
Note that $\gamma( \frac{K(p)}{r} ) = i \sqrt{\frac{c}{a}} $ because of Corollary \ref{lem:nonnonvan} and Lemma \ref{lem:vanequ}. We infer from this and Lemma \ref{lem:limit} that 
\begin{equation}\label{eq:8.101}
\lim_{ t \rightarrow \frac{K(p)}{r}} \frac{|\gamma'(t)|}{\mathrm{Im}(\gamma(t))} = \left\vert \frac{1-i\widehat{x}}{1+i \widehat{x} }\right\vert,
\end{equation}
where $\widehat{x} = \tan ( z_0 + \frac{1}{4} \sqrt{\lambda} pK(p) )$ and $z_0$ is chosen as in  Lemma \ref{lem:rep}. We infer from \eqref{eq:8.100} and \eqref{eq:8.101} that 
%\begin{equation}
$|1- i \widehat x| = |1+i \widehat{x}|$. 
%\end{equation}
Squaring both sides and using $|z+w|^2 = |z|^2+ |w|^2 + 2 \mathrm{Re}(\overline{z}w)$ we infer that $\mathrm{Re}(i\widehat{x}) = 0 $ and therefore $\widehat{x} \in \mathbb{R}$.  We proceed showing that this cannot be true. % Assume that $\widehat{x} \in \mathbb{R}$. 
We distinguish between 3 cases.

\textbf{Case 1:} $ \frac{1}{4} \sqrt{\lambda} K(p) p  = \frac{\pi}{2} + l \pi$ for some $l \in \mathbb{Z}$. In this case %\begin{equation}
$\widehat{x} = \tan ( z_0 + \nicefrac{\pi}{2} ) = - \cot(z_0).$ %\end{equation}   
 Assume that $\cot(z_0) = \beta \in \mathbb{R}$. An easy computation shows that %then
\begin{equation*}
e^{2i z_0} = \frac{i\beta - 1}{i \beta+ 1}. 
\end{equation*}
Taking absolute values on both sides we find 
%\begin{equation*}
$e^{2 \mathrm{Re}(iz_0) } = |e^{2 i z_0 } | = \left\vert \frac{i\beta - 1}{i\beta  + 1} \right\vert = 1$ 
%\end{equation*}
which implies $z_0 \in \mathbb{R}$ and contradicts the statement of Lemma \ref{lem:rep}.

\textbf{Case 2:} $\frac{1}{4} \sqrt{\lambda}  K(p) p $ is not as in Case 1 and $\tan( \frac{1}{4} \sqrt{\lambda} K(p) p) \neq \frac{1}{\widehat{x}} $ . 
In this case we can use $\tan(z+w ) = \frac{\tan(z) + \tan(w)}{1 + \tan(z) \tan(w) }$ to find 
\begin{equation*}
\widehat{x} = \frac{\tan ( \frac{1}{4}\sqrt{\lambda} K(p) p ) + \tan(z_0) }{1 + \tan( \frac{1}{4}\sqrt{\lambda} K(p) p ) \tan(z_0 ) }.
\end{equation*}
We can indeed solve for $\tan(z_0)$ to obtain 
\begin{equation}\label{eq:znought}
\tan(z_0) = \frac{\tan( \frac{1}{4}  \sqrt{\lambda} K(p) p ) }{\widehat{x} \tan( \frac{1}{4}\sqrt{\lambda} K(p) p ) -1  }.
\end{equation}
Notice that we used here that $\tan( \frac{1}{4}\sqrt{\lambda} K(p) p)\widehat{x} \neq 1 $. Observe  that the right hand side of \eqref{eq:znought} is real-valued by assumption. With similar arguments as in case 1, it can be shown that $\tan(z_0) \not \in \mathbb{R}$ if $z_0 \not \in \mathbb{R}$. Again, this leads to  a contradiction to Lemma \ref{lem:rep}.  

\textbf{Case 3:} $\tan( \frac{1}{4} \sqrt{\lambda} K(p) p ) = \frac{1}{\widehat{x}}$.
As in Case 2 we can use the addition formula to find 
\begin{equation*}
\tan \frac{\sqrt{\lambda}}{4} K(p) p = \frac{1}{\widehat{x}} = \frac{1+ \tan ( \frac{1}{4} \sqrt{\lambda } K(p) p ) \tan(z_0) }{\tan ( \frac{1}{4} \sqrt{\lambda } K(p) p) + \tan(z_0)},
\end{equation*}
which implies that $\tan^2 \left( \frac{1}{4} \sqrt{\lambda} K(p) p \right) = 1$ 
and therefore $\frac{1}{4} \sqrt{\lambda} K(p) p = \frac{\pi}{4} + l \pi $ for some $l \in \mathbb{Z}$ since $\frac{1}{4} \sqrt{\lambda} K(p) p $ is positive, we find that in particular $\sqrt{\lambda} K(p) p \geq \pi$.  Using Lemma \ref{lem:8.13} and Proposition \ref{prop:integr} we infer that
\begin{equation*}
p^2 = \frac{2 \sqrt{(\lambda + 2)^2 + 4C}}{2 + \lambda + \sqrt{(\lambda + 2)^2 + 4C} } = \frac{4}{4+ \lambda}. 
\end{equation*} 
Hence the contradiction
\begin{align*}
\pi \leq \sqrt{\lambda} K(p) p  & = \frac{2 \sqrt{\lambda}}{\sqrt{4+ \lambda}} K \left( \frac{2}{\sqrt{4+\lambda}} \right) = 2 \sqrt{\lambda} \int_0^\frac{\pi}{2} \frac{1}{\sqrt{4 + \lambda - 4 \sin^2(\theta)}} d\theta \\ 
& = 2 \int_0^\frac{\pi}{2} \sqrt{\frac{\lambda}{\lambda + 4 \cos^2(\theta)}} d \theta < \pi . \qedhere
 \end{align*}
% A contradiction to \eqref{eq:grpi}.
\end{proof}

\begin{remark}\label{rem:paramrestr}
Recalling Lemma \ref{lem:8.13}, we get also some new parameter restrictions on elastica, for example nonexistence of elastica if $\kappa_0^2 = \lambda + 4$ or, equivalently, if $r^2p^2 = 1$. 
\end{remark}

\section{Jacobi Elliptic Functions}\label{appendix:Jacobi}
We provide some elementary properties of Jacobian elliptic functions, which %are to 
can be found for example in \cite[Chapter 16]{Stegun}.% and in many textbooks (though these might use different definitions). 
\begin{definition}[Amplitude Function, Complete Elliptic Integrals]
Fix $p \in [0,1) $. We define the \textit{Jacobi-amplitude function} $ \am(\,\cdot\,,p) \colon \mathbb{R} \rightarrow \mathbb{R} $ with \textit{modulus} $p$ to be the inverse function of 
\begin{equation*}
\mathbb{R} \ni z \mapsto \int_0^z \frac{1}{\sqrt{1- p^2\sin^2(\theta)}} \diff\theta \in \mathbb{R}
\end{equation*}
We define the \textit{complete elliptic integral of first} and \textit{second kind} as
\begin{align*}
K(p) &:= \int_0^\frac{\pi}{2} \frac{1}{\sqrt{1- p^2 \sin^2(\theta)}} \diff\theta,& E(p) &:= \int_0^\frac{\pi}{2} \sqrt{1- p^2 \sin^2(\theta)} \diff \theta.
\end{align*}
\end{definition}

\begin{definition}[Elliptic Functions]\label{def:B2}
For $p \in [0,1)$ % be arbitrary. We define the following functions 
%and call them 
the \textit{Jacobi Elliptic Functions} are given by
\begin{align*}
\cn(\cdot,p)\colon \mathbb{R} \rightarrow \mathbb{R}, \;\; &\cn(x,m) := \cos(\am(x,p)), \\ \sn(\cdot,p)\colon \mathbb{R} \rightarrow \mathbb{R}, \;\; &\sn(x,m) := \sin(\am(x,p)), \\ \dn(\cdot,p)\colon \mathbb{R} \rightarrow \mathbb{R}, \;\; &\dn(x,m) := \sqrt{1- p^2 \sin^2(\am(x,p))}.
\end{align*}
\end{definition}

\begin{remark}
A lot of literature on elliptic functions defines the elliptic functions using another parameter $m$ to describe the modulus. Most of the times the relation between $m$ and $p$ is $m = p^2$. 
\end{remark}

\begin{prop}[Some identities]\label{prop:identities}
\leavevmode\begin{enumerate}
\item (Derivatives and Integrals of Jacobi Elliptic Functions)  For each $x \in \mathbb{R}$ and $p \in (0,1)$ 
% \begin{align}
% \frac{\partial}{\partial x} \cn(x,p) & = - \sn(x,p) \dn(x,p), \\
% \frac{\partial}{\partial x} \sn(x,p) & =  \cn(x,p) \dn(x,p) , \\
% \frac{\partial}{\partial x} \dn(x,p) & = - p^2 \cn(x,p) \sn(x,p), \\
% \frac{\partial}{\partial x} \am(x,p) & = \dn(x,p),
% \intertext{from which one can deduce}
%  \label{eq:intdnsquared} \int_0^{K(p)} \dn^2(s,p) \diff s &= E(p).
% \end{align}
\begin{align*}
\frac{\partial}{\partial x} \cn(x,p) & = - \sn(x,p) \dn(x,p), &
\frac{\partial}{\partial x} \sn(x,p) & =  \cn(x,p) \dn(x,p) , \\
\frac{\partial}{\partial x} \dn(x,p) & = - p^2 \cn(x,p) \sn(x,p), &
\frac{\partial}{\partial x} \am(x,p) & = \dn(x,p),
\end{align*}
from which one can deduce
\begin{equation}
 \label{eq:intdnsquared} \int_0^{K(p)} \dn^2(s,p) \diff s = E(p).
\end{equation}
\item (Derivatives of Complete Elliptic Integrals) For $p \in (0,1)$ $E$ is smooth and 
\begin{align*}
\frac{\diff}{\diff p}E(p)  & = \frac{E(p) - K(p)}{p}. 
\end{align*}
\item (Trigonometric Identities) For each $p \in [0,1)$ and $x \in \mathbb{R}$ the Jacobi Elliptic functions satisfy
\begin{align*} 
  \cn^2(x,p) + \sn^2(x,p) & = 1, &  \dn^2(x,p) + p^2 \sn^2(x,p) &= 1 .
  \end{align*}
%\item (Periodicity)
%
% Let $l \in \mathbb{Z}$ and $x \in \mathbb{R}$. 
%\begin{align} 
%  E( l \frac{\pi}{2}, m) & = l E(m), \\   
%   E(x+ l \pi, m ) & = 2 l E(m) + E(x,m) \quad \forall x \in \mathbb{R}. \\
%   F( l \frac{\pi}{2}, m) & = l F(m), \\   
%   F(x+ l \pi, m ) & = 2 l K(m) + F(x,m) \quad \forall x \in \mathbb{R}. \\
%\am(lK(m),m) & = l \frac{\pi}{2}, \\
%\am(x+ 2lK(m),m) & =  l \pi + \am(x,m)\quad \forall x \in \mathbb{R }. 
%\end{align}   
\item (Periodicity) All periods of the elliptic functions are given as follows, where $l \in \Z$ and $x \in \R$:
\begin{align} 
\am(lK(p),p) & = l \frac{\pi}{2}, & \nonumber
\cn(x+ 4 l K(p), p ) &  = \cn(x,p)  ,\\
\sn(x+ 4 l K(p), p ) &  = \sn(x,p) , &
\dn(x+ 2 l K(p), p ) &  = \dn(x,p) ,\label{eq:PeriodDN}
\end{align}\vspace{-.5cm}
\begin{equation*}
 \am(x+ 2lK(p),p)  =  l \pi + \am(x,m).
\end{equation*}
\item (Asymptotics of the complete Elliptic integral)
\begin{equation*}
\lim_{p \rightarrow 1 } K (p) = \infty, \quad \quad \quad \quad \lim_{p \rightarrow 0 } K(p) = \frac{\pi}{2}
\end{equation*}
%where $l\in \N$
 \end{enumerate} 
\end{prop}

\begin{prop}\label{prop:ellipt}
For $p \in (0,1)$ let $f(p) :=\frac{E(p)}{\sqrt{2-p^2}}$. Then, $f$ is decreasing and 
\begin{equation*}
1 < f(p) < \frac{\pi}{2\sqrt{2}} \quad \forall p \in (0,1).
\end{equation*}
\end{prop}
\begin{proof}
To show that $f'(p)< 0$ on $(0,1)$ we use %the properties %in this section to compute  
Proposition \ref{prop:identities} to compute
\begin{equation*}
 f'(p)  = \frac{E(p) - K(p)}{p \sqrt{2-p^2}} + \frac{pE(p)}{(2-p^2)^\frac{3}{2}} = \frac{p^2 K(p) + 2 (E(p) - K(p))}{(2-p^2)^\frac{3}{2} p }.  
\end{equation*}
Using the definitions and $1-2 \sin^2(\theta) =  \cos(2\theta)$ we obtain 
\begin{align*}
&p (2-p^2)^\frac{3}{2} f'(p) \\
&\quad = \int_0^\frac{\pi}{2} \left( \frac{p^2}{\sqrt{1-p^2 \sin^2\theta}} + 2 \left( \sqrt{1- p^2 \sin^2\theta } - \frac{1}{\sqrt{1-p^2\sin^2 \theta }}\right)  \right) \diff \theta \\ 
&\quad = \int_0^\frac{\pi}{2} \frac{p^2}{\sqrt{1-p^2 \sin^2\theta }} ( 1- 2 \sin^2\theta) \diff \theta = \int_0^\frac{\pi}{2} \frac{p^2 \cos(2\theta)}{\sqrt{1-p^2 \sin^2\theta}} \diff\theta.
\end{align*}
Using that the integrand is even and $\cos (d+ \pi) = - \cos(d) $ for each $d \in \mathbb{R}$ we obtain
\begin{align*}
p (2-p^2)^\frac{3}{2} f'(p) & =   \int_0^\frac{\pi}{4} \frac{p^2 \cos(2\theta)}{\sqrt{1-p^2 \sin^2\theta}} \diff\theta +  \int_\frac{\pi}{4}^\frac{\pi}{2} \frac{p^2 \cos(2\theta)}{\sqrt{1-p^2 \sin^2\theta}} \diff\theta\\
%& =  \int_\frac{-\pi}{4}^0\frac{p^2 \cos(2\theta)}{\sqrt{1-p^2 \sin^2\theta}} \diff \theta +  \int_\frac{\pi}{4}^\frac{\pi}{2} \frac{p^2 \cos(2\theta)}{\sqrt{1-p^2 \sin^2\theta}} \diff\theta \\
&= \int_\frac{\pi}{4}^\frac{\pi}{2} p^2 \cos( 2 \theta ) \Big( \frac{1}{\sqrt{1-p^2 \sin^2 \theta}} - \frac{1}{\sqrt{1-p^2\sin^2( \theta - \tfrac{\pi}{2})}} \Big) \diff\theta. 
\end{align*}
On $(\frac{\pi}{4}, \frac{\pi}{2})$ the $\cos(2(\cdot))$-function is negative and $\sin^2(\cdot)$ attains values strictly between $\frac{1}{2}$ and $1$, whereas $\sin^2(\cdot- \nicefrac{\pi}{2})$ lies strictly between $0$ and $\frac{1}{2}$. Therefore the expression in parentheses is positive, which implies that the whole integrand is negative. The claim follows since %\begin{equation}
$\lim_{p \rightarrow 0 } f(p) = \frac{\pi}{2 \sqrt{2}}$ and $\lim_{p \rightarrow 1} f(p) = 1$. \end{proof}

\subsection*{Acknowledgement} Marius M\"uller is supported by an
{LGFG grant (number 1705 LGFG-E)}. Adrian Spener is supported by the DFG (project number 355354916). Both authors would like to thank Anna Dall'Acqua and Fabian Rupp for helpful discussions.

\bibliographystyle{amsalpha}

\begin{thebibliography}{{DLLPS}18}

\bibitem[AS64]{Stegun}
Milton Abramowitz and Irene~A. Stegun, \emph{Handbook of mathematical functions
  with formulas, graphs, and mathematical tables}, National Bureau of Standards
  Applied Mathematics Series, vol.~55, 1964. \MR{0167642}
  
\bibitem[B\"ar10]{BaerDiffgeo}
Christian B\"ar, \emph{{E}lementary {D}ifferential {G}eometry}, Cambridge University Press, 2010.

\bibitem[BDF10]{MR2729304}
Matthias Bergner, Anna Dall'Acqua, and Steffen Fr\"ohlich, \emph{Symmetric
  {W}illmore surfaces of revolution satisfying natural boundary conditions},
  Calc. Var. Partial Differential Equations \textbf{39} (2010), no.~3-4,
  361--378. \MR{2729304}

\bibitem[Bla09]{Blatt}
Simon Blatt, \emph{A singular example for the {W}illmore flow}, Analysis
  (Munich) \textbf{29} (2009), no.~4, 407--430. \MR{2591055}

\bibitem[Bla18]{Blatt2}
\bysame, \emph{A note on singularities in finite time for the constrained
  willmore flow}, Preprint (2018).

\bibitem[CFS09]{chill2009}
Ralph Chill, Eva Fa\v{s}angov\'a, and Reiner Sch\"{a}tzle, \emph{Willmore
  blowups are never compact}, Duke Math. J. \textbf{147} (2009), no.~2,
  345--376.

\bibitem[Chi03]{CHILL2003572}
Ralph Chill, \emph{On the {Ł}ojasiewicz–{S}imon gradient inequality},
  Journal of Functional Analysis \textbf{201} (2003), no.~2, 572 -- 601.

\bibitem[Dav62]{Davis}
Harold~T. Davis, \emph{Introduction to nonlinear differential and integral
  equations}, Dover Publications, Inc., New York, 1962. \MR{0181773}

\bibitem[dC92]{doCarmo}
Manfredo~Perdig{\~a}o do~Carmo, \emph{Riemannian geometry}, Mathematics: Theory
  \& Applications, Birkh\"auser Boston, Inc., Boston, MA, 1992. \MR{1138207}

\bibitem[DDG08]{MR2480063}
Anna Dall'Acqua, Klaus Deckelnick, and Hans-Christoph Grunau, \emph{Classical
  solutions to the {D}irichlet problem for {W}illmore surfaces of revolution},
  Adv. Calc. Var. \textbf{1} (2008), no.~4, 379--397. \MR{2480063}

\bibitem[DKS02]{DKS}
Gerhard Dziuk, Ernst Kuwert, and Reiner Sch\"atzle, \emph{Evolution of elastic
  curves in {$\mathbb{R}^n$}: existence and computation}, SIAM J. Math. Anal.
  \textbf{33} (2002), no.~5, 1228--1245. \MR{1897710}

\bibitem[DLLPS18]{Sphera}
Anna Dall’Acqua, Tim Laux, Chun-Chi Lin, Paola Pozzi and Adrian Spener, \emph{The elastic flow of curves on the sphere}, Geometric Flows \textbf{3}(1), 1--13 (2018).

\bibitem[DP14]{MR3263933}
Anna Dall'Acqua and Paola Pozzi, \emph{A {W}illmore-{H}elfrich {$L^2$}-flow of
  curves with natural boundary conditions}, Comm. Anal. Geom. \textbf{22}
  (2014), no.~4, 617--669. \MR{3263933}

\bibitem[DPS16]{Loja}
Anna Dall'Acqua, Paola Pozzi, and Adrian Spener, \emph{The
  Łojasiewicz–{S}imon gradient inequality for open elastic curves}, Journal
  of Differential Equations \textbf{261} (2016), no.~3, 2168 -- 2209.

\bibitem[DS17]{DS17}
Anna Dall'Acqua and Adrian Spener, \emph{The elastic flow of curves in the
  hyperbolic plane}, Preprint (https://arxiv.org/abs/1710.09600) (2017).

\bibitem[DS18]{DS18}
\bysame, \emph{Circular solutions to the elastic flow in hyperbolic space},
  RIMS K\^{o}ky\^{u}roku, Proceedings of the conference Analysis on Shapes of
  Solutions to Partial Differential Equations, Kyoto 2017/06/05--2017/06/07
  (2018), no.~2082.

\bibitem[EG16]{Eichmann}
Sascha Eichmann and Hans-Christoph Grunau, \emph{Existence for willmore
  surfaces of revolution satisfying non-symmetric dirichlet boundary
  conditions}, preprint (2016).

\bibitem[Eic14]{EichmannMaster}
Sascha Eichmann, \emph{Nichtperiodische {F}ortsetzbarkeit von
  {W}illmore-{F}l\"achen unter {A}xialsymmetrie}, 2014.

\bibitem[Eic17]{EichmannPhD}
\bysame, \emph{Willmore surfaces of revolution satisfying dirichlet data}, Preprint (2017).

\bibitem[ESI92]{ElSoufi1992}
Ahmad~El~Soufi and Sa{\"i}d~Ilias, \emph{Une in{\'e}galit{\'e} du type ``reilly'' pour
  les sous-vari{\'e}t{\'e}s de l'espace hyperbolique}, Commentarii Mathematici
  Helvetici \textbf{67} (1992), no.~1, 167--181.

\bibitem[GAS06]{Abbena}
Alfred Gray, Elsa Abbena, and Simon Salamon, \emph{Modern differential geometry
  of curves and surfaces with {M}athematica{$^\circledR$}}, third ed., Studies
  in Advanced Mathematics, Chapman \& Hall/CRC, Boca Raton, FL, 2006.
  \MR{2253203}

\bibitem[Gho11]{Ghomi}
Mohammad Ghomi, \emph{A {R}iemannian four vertex theorem for surfaces with
  boundary}, Proc. Amer. Math. Soc. \textbf{139} (2011), no.~1, 293--303.
  \MR{2729091}

\bibitem[Hel14]{Heller}
Lynn Heller, \emph{Constrained {W}illmore tori and elastic curves in
  2-dimensional space forms}, Comm. Anal. Geom. \textbf{22} (2014), no.~2,
  343--369. \MR{3210758}

\bibitem[Koi00]{Koiso}
Norihito Koiso, \emph{Convergence towards an elastica in a Riemannian manifold},  Osaka J. Math. 37 (2000), no. 2, 467–487. \MR{1427766}

\bibitem[KS04]{MR2119722}
Ernst Kuwert and Reiner Sch\"{a}tzle, \emph{Removability of point singularities
  of {W}illmore surfaces}, Ann. of Math. (2) \textbf{160} (2004), no.~1,
  315--357. \MR{2119722}

\bibitem[Lin98]{MR1432203}
Anders Linn\'{e}r, \emph{Curve-straightening and the {P}alais-{S}male
  condition}, Trans. Amer. Math. Soc. \textbf{350} (1998), no.~9, 3743--3765.
  \MR{1432203}

\bibitem[Lin12]{MR2911840}
Chun-Chi Lin, \emph{{$L^2$}-flow of elastic curves with clamped boundary
  conditions}, J. Differential Equations \textbf{252} (2012), no.~12,
  6414--6428. \MR{2911840}

\bibitem[LS84a]{LangerSinger2}
Joel Langer and David Singer, \emph{Curves in the hyperbolic plane and mean
  curvature of tori in {$3$}-space}, Bull. London Math. Soc. \textbf{16}
  (1984), no.~5, 531--534. \MR{751827}

\bibitem[LS84b]{LangerSinger}
\bysame, \emph{The total squared curvature of closed curves}, J. Differential
  Geom. \textbf{20} (1984), no.~1, 1--22. \MR{772124}

\bibitem[LY82]{LiYau}
Peter Li and Shing~Tung Yau, \emph{A new conformal invariant and its
  applications to the {W}illmore conjecture and the first eigenvalue of compact
  surfaces}, Invent. Math. \textbf{69} (1982), no.~2, 269--291. \MR{674407}

\bibitem[{Man}17]{2017arXiv170502177M}
Rainer~{Mandel}, \emph{{Explicit formulas and symmetry breaking for Willmore
  surfaces of revolution}}, ArXiv e-prints (2017).

\bibitem[MN14]{MR3182810}
Andrea Mondino and Huy~The Nguyen, \emph{A gap theorem for {W}illmore tori and
  an application to the {W}illmore flow}, Nonlinear Anal. \textbf{102} (2014),
  220--225. \MR{3182810}

\bibitem[Nas56]{Nash}
John Nash, \emph{The imbedding problem for {R}iemannian manifolds}, Ann. of
  Math. (2) \textbf{63} (1956), 20--63. \MR{0075639}

\bibitem[Pol96]{Polden}
Alexander Polden, \emph{Curves and surfaces of least total curvature and fourth-order flows}, PhD Thesis, T\"ubingen (1996).
  
\bibitem[Rei77]{Reilly}
Robert~C. Reilly, \emph{On the first eigenvalue of the {L}aplacian for compact
  submanifolds of {E}uclidean space}, Comment. Math. Helv. \textbf{52} (1977),
  no.~4, 525--533. \MR{0482597}

\bibitem[Ste95]{Steinberg}
Daniel~Howard Steinberg, \emph{Elastic curves in hyperbolic space}, ProQuest
  LLC, Ann Arbor, MI, 1995, Thesis (Ph.D.)--Case Western Reserve University.
  \MR{2693537}

\bibitem[Tru83]{HistoryElasticity}
Clifford~Truesdell, \emph{The influence of elasticity on analysis: the classic
  heritage}, Bull. Amer. Math. Soc. (N.S.) \textbf{9} (1983), no.~3, 293--310.
  \MR{714991}


\bibitem[Wil93]{Willmore}
Thomas J. Willmore, \emph{Riemannian Geometry}, Oxford Science Publications. The Clarendon Press, Oxford University Press, New York, 1993. xii+318 pp. ISBN: 0-19-853253-9. \MR{1261641}

\bibitem[Whi37]{Whitney}
Hassler Whitney, \emph{On regular closed curves in the plane}, Compositio Math.
  \textbf{4} (1937), 276--284. \MR{1556973}

\bibitem[Zei90]{Zeidler}
Eberhard Zeidler, \emph{Nonlinear functional analysis and its applications.
  {II}/{B}}, Springer-Verlag, New York, 1990. \MR{1033498}

\end{thebibliography}

\end{document}